\pdfoutput=1
\documentclass[10pt,reqno]{article}
 \usepackage{tikz}
 \usepackage{graphicx}
\usepackage[backend=bibtex]{biblatex}
\usepackage[margin=1in]{geometry}
\usepackage{bbm}
\usepackage{amssymb, amsthm, mathtools}
\usepackage{amsfonts}
\usepackage{latexsym, amssymb, amsmath, amscd, amsthm, amsxtra}
\usetikzlibrary{decorations.pathmorphing}
\usepackage{mathtools}
\usepackage{enumerate}
\usepackage[all]{xy}
\usepackage{mathrsfs}
\usepackage{fancyhdr}
\usepackage{listings}
\usepackage{hyperref}
\usepackage{cleveref}
\usepackage{soul}
\usepackage{xcolor}
\usepackage{dsfont}
\usepackage{stmaryrd}
\usepackage{comment}

\allowdisplaybreaks
\DeclareMathOperator{\re}{{\mathrm{Re}}}
\DeclareMathOperator{\im}{{\mathrm{Im}}}

\hypersetup{
     colorlinks=true
}
\def\P{\mathbb{P}}

\def\E{\mathbb{E}}

\def\Z{\mathbb{Z}}
\def\R{\mathbb{R}}
\def\N{\mathbb{N}}

\def\11{{\mathbf{1}}}

\newcommand{\C}{\mathbb C}

\renewcommand{\bar}{\overline}

\DeclareMathOperator{\OO}{O}

\DeclareMathOperator{\tr}{Tr}

\newcommand{\be}{\begin{equation}}
\newcommand{\ee}{\end{equation}}

\newcommand{\e}{{\varepsilon}}

\newcommand{\cor}{\color{red}}

\theoremstyle{plain} 
\newtheorem{theorem}{Theorem}[section]
\newtheorem*{theorem*}{Theorem}
\newtheorem{lemma}[theorem]{Lemma}

\newtheorem*{lemma*}{Lemma}
\newtheorem{corollary}[theorem]{Corollary}
\newtheorem*{corollary*}{Corollary}

\newtheorem*{proposition*}{Proposition}

\newtheorem*{claim*}{Claim}

\newtheorem{definition}[theorem]{Definition}  
\newtheorem*{definition*}{Definition}         
\theoremstyle{remark}

\newtheorem*{example*}{Example}

\newtheorem*{remark*}{Remark}
\newtheorem*{remarks*}{Remarks}

 \numberwithin{equation}{section}

\allowdisplaybreaks
\def\@setthanks{\vspace{-\baselineskip}\def\thanks##1{\@par##1\@addpunct.}\thankses}

\title{Delocalization of Two-Dimensional Random Band Matrices}
\author{
    Sofiia Dubova\thanks{Harvard University, sdubova@math.harvard.edu} \and
    Kevin Yang\thanks{Harvard University, kevinyang@math.harvard.edu} \and
    Horng-Tzer Yau\thanks{Harvard University, htyau@math.harvard.edu} \and
  Jun Yin \thanks{UCLA, jyin@math.ucla.edu}
}
\date{\today}  

\begin{document}

\maketitle

\begin{abstract}
We study a random band matrix $H=(H_{xy})_{x,y}$ of dimension $N\times N$ with mean-zero complex Gaussian entries, where $x,y$ belong to the discrete torus $(\Z/\sqrt{N}\Z)^{2}$. The variance profile $\E|H_{xy}|^{2}=S_{xy}$ vanishes when the distance between $x,y$ is larger than some band-width parameter $W$ depending on $N$.

We show that if the band-width satisfies $W\geq N^{\mathfrak{c}}$ for some $\mathfrak{c}>0$, then in the large-$N$ limit, we have the following results. The first result is a local semicircle law in the bulk down to scales $N^{-1+\e}$. The second is delocalization of bulk eigenvectors. The third is a quantum unique ergodicity for bulk eigenvectors. The fourth is universality of local bulk eigenvalue statistics. The fifth is a quantum diffusion profile for the associated $T$ matrix.

{Our method is based on embedding $H$ inside a matrix Brownian motion $H_{t}$ as done in \cite{DY,YY_25} for band matrices on the one-dimensional torus. In this paper, the key additional ingredient in our analysis of $H_{t}$ is a new CLT-type estimate for polynomials in the entries of the resolvent of $H_{t}$.}

\end{abstract}

 \newpage

\tableofcontents

\section{Introduction}

The delocalization-localization transition for random operators has been a central question in probability theory and mathematical physics since the introduction of the Anderson model, a discrete random Schr\"{o}dinger operator. There has been a lot of work \cite{FroSpen_1985,Bourgain2005,Carmona1987,DingSmart2020,Damanik2002,Germinet2013,LiZhang2019} showing the existence of localized states, but the existence of delocalized states has been addressed in only one work \cite{yang2024Del} in high dimension, to our knowledge. 

A popular model \cite{PB_review,Spe} that has received a lot of attention and is conjectured to exhibit similar behavior to the random Schr\"{o}dinger operator is the $d$-dimensional \emph{random band matrix}. This is a random Hermitian matrix $H=(H_{xy})_{x,y}$ of dimension $N\times N$ whose entries are independent mean-zero complex random variables (up to the Hermitian constraint). The indices $x,y$ belong to the torus $(\Z/\sqrt[d]{N}\Z)^{d}$, and the variance profile $S_{xy}=\E|H_{xy}|^{2}$ vanishes if the distance between $x,y$ exceeds a band-width parameter $W$ depending on $N$.

In the case of dimension $d=1$, it was conjectured that \cite{PhysRevLett.64.1851, PhysRevLett.64.5, PhysRevLett.66.986} bulk eigenvector and eigenvalue statistics of $H$ match those of the \emph{Gaussian unitary ensemble} (GUE) if $W\gg N^{1/2}$. This was shown in \cite{YY_25}; see also \cite{BaoErd2015, erdHos2013local, EK_band1,ErdKno2011, HeMa2018,bourgade2017universality, bourgade2019random,bourgade2020random,yang2021random,DY}, which place stronger assumptions on the band width $W$. On the other hand, it is conjectured that if $W \ll N^{1/2}$, the eigenvectors of $H$ are localized, which differs from GUE eigenvector statistics. Current rigorous results show localization of bulk eigenvectors for $W \ll N^{1/4}$ \cite{https://doi.org/10.48550/arxiv.2206.05545, https://doi.org/10.48550/arxiv.2206.06439,Sch2009,PelSchShaSod}.

The focus of this paper is dimension $d=2$. In this case, it is conjectured \cite{PB_review} that bulk eigenvalue and eigenvector statistics of $H$ match those of the GUE when $W\gg\sqrt{\log N}$. The goal of this paper is a proof of this conjecture for dimension $d=2$ in the case of complex Gaussian entries and under the slightly stronger assumption $W\gg N^{\e}$ for any fixed $\e>0$. The precise model and main results will be stated in the next section. The method we use to do so builds on the approach of \cite{YY_25}. Let us explain this briefly.

The starting point in \cite{YY_25} comes from an idea of \cite{DY}. The idea is to embed $H$ inside a matrix Brownian motion $H_{t}$ with the correct variance profile and to study the evolution of ``loops" of resolvents of $H_{t}$. In particular, we let $G_{t}:=(H_{t}-z_{t})^{-1}$, where $z_{t}$ is an appropriately chosen time-dependent spectral parameter in the upper half-plane. The aforementioned ``loops" are objects of the form 
\begin{align*}
{\cal L}_{t,\boldsymbol{\sigma},\textbf{a}}:=\tr\left(\prod_{i=1}^{n}G_{t}(\sigma_{i})E_{a_{i}}\right),
\end{align*}
where $\boldsymbol{\sigma}=(\sigma_{1},\ldots,\sigma_{n})\in\{+,-\}^{n}$, where $\textbf{a}=(a_{1},\ldots,a_{n})\in(\Z/L\Z)^{2n}$, where $G_{t}(+)=G_{t}$ and $G_{t}(-)=G_{t}^{\dagger}$, and where $E_{a}$ certain ``elementary matrices". Of particular importance is the case $n=2$ and $\boldsymbol{\sigma}=(+,-)$, which recovers for us the so-called ``$T$-matrix" that has been a fundamental object in the study of random band matrices. We note that this ``flow method" was used also in \cite{10.1214/19-ECP278, Sooster2019} to prove a local semicircle law for generalized Wigner matrices.

In \cite{YY_25}, a detailed analysis of the \emph{loop hierarchy} $\{{\cal L}_{t,\boldsymbol{\sigma},\textbf{a}}\}_{n,\boldsymbol{\sigma},\textbf{a}}$ was performed to prove \emph{quantum diffusion}, i.e. a diffusive profile for the $T$-matrix, in dimension $d=1$, from which GUE bulk statistics were shown using ideas from \cite{bourgade2019random,bourgade2020random,yang2021random}. 

{In this paper, we use this method to do the same for $d=2$. In this case, the key new ingredient for the loop hierarchy analysis is a \emph{CLT-type estimate} for polynomials in the entries of $G_{t}$. This essentially states that $(G_{t})_{xy}$ and $(G_{t})_{x’y’}$ are almost independent if $(x,y)$ and $(x’,y’)$ are far from each other. Let us illustrate its utility with a simple example. Consider a sum of the following form with $\ell\gg1$ and $x$ fixed:
\begin{align*}
\sum_{|y|\leq \ell} (G_{t})_{xy}.
\end{align*}
In dimension $d$, the triangle inequality shows that the above scales as $O(\ell^{d})$ with respect to $\ell$. This is one of the main difficulties with analyzing the $d =2$ case compared to $d=1$. However, the new CLT-type estimates in this paper establish some square-root cancellation in the above sum and yield a sharper upper bound of $O(\ell \cdot \ell_{t})$ for $d=2$ (where $\ell_{t}$ is some dynamical length-scale). In our applications, we will take $\ell \gg \ell_{t}$, so this provides important power-saving.}

We conclude this introduction with a discussion of related results. The proof of the full delocalized phase, quantum diffusion, and universality of local eigenvalue statistics for dimension $d\geq 7$ was shown in \cite{yang2021delocalization, yang2022delocalization, Xu:2024aa}, which uses an expansion method for the resolvent of $H$. In dimension $d=3$, for special variance profiles of $H$, supersymmetric methods have been used \cite{DisPinSpe2002} to estimate the density of states for $H$. Similar methods have also been able to detect the transition between $W\gg N^{1/2}$ and $W\ll N^{1/2}$ at the level of two-point functions in dimension $d=1$ \cite{Shcherbina:2021wz}. See \cite{BaoErd2015, SchMT, Sch1, Sch2, Efe1997, Spe} for further discussion on the supersymmetric method.

\subsection*{Acknowledgements}
K.Y. was supported in part by NSF Grant No. DMS-2203075.

\section{The model and main results}\label{sec_model-results}
 \subsection{Band matrix model}\label{subsec_setting}

We will study the following two-dimensional band block matrix model. Fix positive integers $W,L>0$, and set $N:=W^{2}L^{2}$. Let $\Z_{L}:=\Z/L\Z\simeq\{1,\ldots,L\}$.

The matrix of interest is $H=(H_{xy})_{x,y\in \Z_{WL}^{2}}$, where $H_{xy}=\overline{H}_{xy}$, and up to this Hermitian constraint, the entries $H_{xy}\sim\mathcal{CN}(0,S_{xy})$ are independent centered complex Gaussian random variables of variance $\E|H_{xy}|^{2}=S_{xy}$, where $S=(S_{xy})_{x,y\in \Z_{WL}^{2}}$ is a matrix with the following block structure. We define
\begin{align*}
S^{(B)}_{ab}&=\frac{1}{5}\mathbf{1}(|a-b|_{L}\leq1),\quad(S_{W})_{ij}=W^{-2},\quad a,b\in\Z_{L}^{2}, \ i,j\in\{1,\ldots,W\}^{2}.
\end{align*}
Note that $S^{(B)}$ is $L^{2}\times L^{2}$, and that $S_{W}$ is $W^{2}\times W^{2}$. Above, $|a-b|_{L}$ is the periodic $L^{1}$ distance on $\Z_{L}^{2}$. We now define $S$ to be the following real-symmetric, stochastic matrix:
\begin{align*}
S=S^{(B)}\otimes S_{W}.
\end{align*}
Note that the entries of $H$ and $S$ are indexed by $\Z_{WL}^{2}$. 

We now introduce standard notation for the setting of this paper. We start by letting $\lambda_{1}\leq\ldots\leq\lambda_{N}$ be eigenvalues of $H$ with corresponding unit eigenvectors $\boldsymbol{\psi}_{1},\ldots,\boldsymbol{\psi}_{N}$, so that
$$H\boldsymbol{\psi}_{k}=\lambda_{k}\boldsymbol{\psi}_{k}.$$ Our analysis of these quantities is based on the resolvent $G(z):=(H-z)^{-1}$ of $H$ defined for any $z$ with $\im z>0$. 

It is well known that the empirical spectral measure of $H$ converges almost surely to the Wigner semicircle distribution as $N\to\infty$ if $W\to\infty$. The Stieltjes transform of the semicircle distribution is given by
\begin{align*}
m(z):=\int_{-2}^{2}\frac{1}{2\pi}\frac{\sqrt{(4-x^{2})_{+}}}{x-z}dx=\frac{-z+\sqrt{z^{2}-4}}{2}.
\end{align*}
Finally, we will use a standard notion of stochastic domination as used in \cite{Average_fluc}, defined below.
\begin{definition}\label{stoch_domination}
	{\rm{(i)}} Consider the following two families of non-negative random variables parameterized by a set $U^{(N)}$:
	\[\xi=\left(\xi^{(N)}(u):N\in\mathbb N, u\in U^{(N)}\right),\hskip 10pt \zeta=\left(\zeta^{(N)}(u):N\in\mathbb N, u\in U^{(N)}\right).\]
	We say that $\xi$ is stochastically dominated by $\zeta$ (uniformly in $u$) if for any fixed (small) $\tau>0$ and (large) $D>0$, we have the following for $N\geq N_{0}(\tau,D)$:
	\[\mathbb P\bigg(\bigcup_{u\in U^{(N)}}\left\{\xi^{(N)}(u)>N^\tau\zeta^{(N)}(u)\right\}\bigg)\le N^{-D}\]
	We will use the notation $\xi\prec\zeta$ or $\xi=\mathrm{O}_{\prec}(\zeta)$ to mean that $\xi$ is stochastically dominated by $\zeta$. If $\xi$ is a family of complex numbers, then $\xi\prec\zeta$ and $\xi=\mathrm{O}_{\prec}(\zeta)$ mean that $|\xi|\prec\zeta$. 
	
	\vspace{5pt}
	\noindent {\rm{(ii)}} As a convention, for deterministic $\xi,\zeta\geq0$, we write $\xi\prec\zeta$ if and only if $\xi\le N^\tau \zeta$ for any fixed $\tau>0$. 

 \vspace{5pt}
	\noindent {\rm{(iii)}} Let $A$ be a family of random matrices and $\zeta$ be a family of non-negative random variables. In this case, we use $A=\OO_\prec(\zeta)$ to mean that $\|A\|\prec \xi$, where $\|\cdot\|$ is the operator norm. 
	
	\vspace{5pt}
	\noindent {\rm{(iv)}} We say that an event $\Xi$ holds with high probability (w.h.p.) if for any $D>0$, we have $\mathbb P(\Xi)\ge 1- N^{-D}$ for large enough $N$. We say that an event $\Omega$ holds $w.h.p.$ in $\Xi$ if for any $D>0$, we have
	$\P( \Xi\setminus \Omega)\le N^{-D}$ for large enough $N$.
\end{definition}

\subsection{Main results}

Our first result states a strong delocalization estimate for bulk eigenvectors of $H$.

\begin{theorem}[Delocalization]\label{MR:decol}
Suppose that the following is satisfied for some fixed constant $ \mathfrak c>0$: 
\begin{equation}\label{Main_DEL_COND}
    W \ge N^{\mathfrak c}. 
\end{equation}
Then for any $\kappa, \tau, D>0$, there exists $N_0$ such that for all $N \geqslant N_0$ we have
$$
\mathbb{P}\left( \max_k\left\|\boldsymbol{\psi}_k\right\|_{\infty}^2 \cdot {\bf1}\big(\lambda_k\in[-2+\kappa,2-\kappa]\big)\leqslant N^{-1+\tau}  \right) \geqslant 1-N^{-D}
$$
\end{theorem}

The proof of Theorem \ref{MR:decol} is based on estimates for the Green's function $G(z)$, which we state below. First, however, we need some more notation. We start with a ``diffusive" length-scale with respect to a parameter $z$ in the upper half plane:
\begin{align}
\ell:=\ell(z):=\min\left(\eta^{-1/2},L\right)+1, \quad\eta:=\im(z).\label{def_ellz}
\end{align}
We will also introduce the fundamental control parameter to be
\begin{align}
M_{\eta}:=W^{2}\ell(z)^{2}\eta,\quad\eta=\im(z).\label{def_meta}
\end{align}
Next, we introduce the following index set given any $a=(a(1),a(2))\in\Z_{L}^{2}$:
\begin{align*}
{\cal I}^{(2)}_{a}:={\cal I}_{a(1)}\times{\cal I}_{a(2)},\quad{\cal I}_{a(i)} := \left[ (a(i)-1)W + 1, \; a(i)W \right].
\end{align*}
\begin{theorem}[Local semicircle law]\label{MR:locSC} Suppose that the assumption of Theorem \ref{MR:decol} holds. For any fixed $\kappa, \tau,D>0$, and for any $z\in\C$ of the form 
$$z=E+i\eta,\quad\quad |E|\le  2-\kappa, \quad 1\ge \eta\ge N^{-1+\tau},
$$
there exists $N_0$ such that for all $N \geqslant N_0$ we have the following estimates. First, we have the local law:
\begin{equation}\label{G_bound}
   \mathbb P\left( \max_{x,y}  \left|\left(G(z)-m(z)\right)_{xy}\right|\le  \frac{W^\tau}{\sqrt{M_{\eta}}}\right)\geqslant 1-N^{-D},\quad \ell=\ell(z).
\end{equation}
Second, we have the partial tracial local law
\begin{equation}\label{G_bound_ave}
\mathbb{P}\left(\max_a\Big| W^{-2}\sum_{x\in {\cal I}^{(2)}_a}G_{xx}(z)-m(z) \Big| \le \frac{W^\tau}{M_{\eta}} \right)\geqslant 1-N^{-D},\quad \ell=\ell(z). 
\end{equation}

\end{theorem}

We now state our next two results. The first of these is {quantum unique ergodicity}, and the second is a quantum diffusion profile for the resolvent. 
\begin{theorem}[Generalized quantum unique ergodicity]\label{MR:QUE} 
Suppose the assumptions of Theorem \ref{MR:decol} hold and 
\begin{align}\label{tau} 
0< \tau <  \frac {  \mathfrak c } 2, 
\end{align} 
where $ \mathfrak c$ was defined in \eqref{Main_DEL_COND}. Next, let $E_a$ denote the following  block identity matrix: 
\begin{equation}\label{Def_matE}
(E_a)_{xy}=  \delta_{xy}\cdot W^{-2}\cdot\textbf{1}(x\in {\cal I}^{(2)}_a),\quad x,y\in \Z_{WL}^{2}.
\end{equation}
Then  for any $\kappa,D>0$,  
there exists   $N_0$ such that for all $N \geqslant N_0$, we have
\begin{equation}\label{Meq:QUE}
\max_{E: \,|E|<2-\kappa}\; 
\max_{ a\,\in\, \mathbb Z_L^{2}}\;\mathbb{P}\left(\max_{k,\ell \in {\cal J}_E} 
 \left|N (\boldsymbol{\psi}_k^*\left(E_a-N^{-1}\right) \boldsymbol{\psi}_{\ell})\right|^2 \ge N^{-\tau/6} \right) \le  N^{-\tau/6}.
\end{equation}
Here 
$${\cal J}_E:=\left\{k: |\lambda_{k}-E|\le N^{-1-\tau} {W^{2/3}}\right\}.
$$
Moreover, for  any  subset $A\subset \mathbb Z_L^{2}$,
\begin{equation}\label{Meq:QUE2}
\max_{E: \,|E|<2-\kappa}\; 
\max_{ A\,\subset\, \mathbb Z_L^{2}}\;\mathbb{P}\left(\max_{k\in {\cal J}_E} 
\left|\,\sum_{x\in {\cal I}^{(2)}_a}\sum_{a\,\in\, A}\left|\boldsymbol{\psi}_k(x)\right |^2 -\frac{|A|\cdot W^{2}}{N}\right| \geqslant \frac{|A|\cdot W^{2}}{N^{1+\tau/6}}  \right) \le N^{-\tau/6}
\end{equation}


\end{theorem}

\begin{theorem}[Quantum diffusion]\label{MR:QDiff}
Under the assumptions of Theorem \ref{MR:locSC}, we have (for $m:=m(z)$)
 \begin{align}\label{Meq:QdW1}
 &
 \mathbb P\left(\max_{a,b }   \left|\tr G E_a G^\dagger E_b- W^{-2}\left(\frac{|m|^2}{1-|m|^2 S^{(B)}} \right)_{ab}\right|\le\frac{W^\tau}{ M_{\eta}^{ 2}}\right)\ge 1-N^{-D},
 \\ \label{Meq:QdW2}
 &
 \mathbb P\left( \max_{a,b }   \left|\tr G E_a G  E_b- W^{-2}\left(\frac{m^2}{1-m^2 S^{(B)}} \right)_{ab}\right|\le\frac{W^\tau}{M_{\eta}^{ 2}}\right)\ge 1-N^{-D}.
\end{align}
We also have the following improved estimate at the level of expectation:
\begin{align}\label{Meq:QdS1}
 \max_{a,b }   \left|\mathbb E\tr G E_a G^\dagger E_b- W^{-{2}}\left(\frac{|m|^2}{1-|m|^2 S^{(B)}} \right)_{ab}\right|\le  M_{\eta}^{-3} \cdot W^\tau,   \\
 \label{Meq:QdS2}
 \max_{a,b }   \left|\mathbb E\tr G E_a G  E_b- W^{-{2}}\left(\frac{m^2}{1-m^2 S^{(B)}} \right)_{ab}\right|\le M_{\eta}^{-3} \cdot W^\tau   .
\end{align}
\end{theorem}

Theorems \ref{MR:decol} and \ref{MR:QUE} are simple consequences of Theorems \ref{MR:locSC} and \ref{MR:QDiff}. The proof of Theorem \ref{MR:decol} is identical to that in \cite{YY_25}, so we omit it. The proof of Theorem \ref{MR:QUE} is very similar to the argument in \cite{YY_25}, but since there are slight dependencies on the dimension, we give the details below.

\begin{proof}[Proof of Theorem \ref{MR:QUE}]  
We prove \eqref{Meq:QUE}; the proof of \eqref{Meq:QUE2} is the same after replacing $E_{a}$ by $|A|^{-1}\sum_{a\in A}E_{a}$ in the argument below. Choose  $  z=E+\eta i$ and $$\eta= N^{-1-\tau }{W^{2/3}}.$$
We first have
\begin{align}\label{ssfa2}
& \E\left(\sum_{k,\ell} \textbf{1}(\lambda_k,\lambda_\ell\in {\cal J}_E)\cdot \left|N(\boldsymbol{\psi}_{k}^*\left(E_a-N^{-1}\right) \boldsymbol{\psi}_{\ell})\right|^2\right) \nonumber \\
& 
\le C \eta^4 \E\left(\sum_{k,\ell  } \frac{\left|N(\boldsymbol{\psi}_{k}^*\left(E_a-N^{-1}\right) \boldsymbol{\psi}_{\ell})\right|^2}{\left|\lambda_k-z\right|^2\left|\lambda_{\ell}-z\right|^2}\right)\\\nonumber
& \le 
 C N^{2}\eta^2\E\tr \left(\operatorname{Im} G(z)\left(E_a-N^{-1}\right) \operatorname{Im} G(z)\left(E_a-N^{-1}\right)\right)
\end{align}
For the expectation of the last term above, we have
\begin{align}\label{que0} 
 &\mathbb{E}  \tr \left(\operatorname{Im} G(z)\left(E_a-N^{-1}\right) \operatorname{Im} G(z)\left(E_a-N^{-1}\right)\right)
  \\ 
=&\frac{1}{L^2}\sum_{b,b'}
\mathbb{E}  
 \tr \left(\operatorname{Im} G(z)\left(E_a-E_b\right) \operatorname{Im} G(z)\left(E_a-E_{b'}\right)\right)
\nonumber \\ \nonumber
\le C&\max_{a,b,a',b'}
\Big|\mathbb{E}\tr \left(\operatorname{Im} G(z) E_a  \operatorname{Im} G(z)E_b \right)
-\mathbb{E}\tr \left(\operatorname{Im} G(z) E_{a'}  \operatorname{Im} G(z)E_{b'}\right) \Big| 
\end{align}
Recall that $N=W^{2}L^{2}$. By \eqref{tau},    we have $ \eta^{-1/2} \ge L$, so that $\ell(z)=L$ by \eqref{def_ellz}. We can bound the last line using \eqref{Meq:QdS1}-\eqref{Meq:QdS2} in Theorem \ref{MR:QDiff}. Recall that $M_{\eta}:=W^{2}\ell^{2}\eta$, and note that $N=W^{2}\ell^{2}$. We ultimately get
 \begin{align}
 \;& \eta^2 \mathbb{E}  \tr \left(\operatorname{Im} G(z)\left(E_a-N^{-1}\right) \operatorname{Im} G(z)\left(E_a-N^{-1}\right)\right)
\nonumber \\
\le \;& \eta^2 M_{\eta}^{-3}\cdot W^{\delta} + \eta^2 W^{-2}\max_{a,b,a',b'}\;\max_{ \xi=|m|^2 \text{ or } \xi=m^2}
\left| \left(\frac{1}{1-\xi\cdot S^{(B)}}\right)_{ab}-\left(\frac{1}{1-\xi\cdot S^{(B)}}\right)_{a'b'}\right|
\nonumber 
 \end{align}
for any $\delta>0$ small. {Next, fix any $a,b,a',b'\in\Z_{L}^{2}$, and for $\xi=|m(z)|^{2},m(z)^{2}$ write
\begin{align*}
 \left|(1-\xi\cdot S^{(B)})^{-1}_{ab}-(1-\xi\cdot S^{(B)})^{-1}_{a'b'}\right| &= \left|(1-\xi \cdot S^{(B)})^{-1}_{0(a-b)}-(1-\xi\cdot S^{(B)})^{-1}_{0(a'-b')}\right|\\
 &\leq \left|(1-\xi \cdot S^{(B)})^{-1}_{00}-(1-\xi\cdot S^{(B)})^{-1}_{0(a-b)}\right|\\
 &+ \left|(1-\xi \cdot S^{(B)})^{-1}_{00}-(1-\xi\cdot S^{(B)})^{-1}_{0(a'-b')}\right|.
\end{align*}
Choose a path $0=c_{0}\mapsto c_{1}\mapsto\ldots\mapsto a-b$ in $\Z_{L}^{2}$ of length $O(L)$ such that for each step in this path, we have $1/(|c_{k}|_{L}+1)\leq C/(k+1)$ for $C=O(1)$ and all $k$. By the first-derivative bound in Lemma \ref{lem_propTH}, we have 
\begin{align*}
\max_{\xi=|m|^{2} \ \text{or} \ \xi=m^{2}}\left|(1-\xi \cdot S^{(B)})^{-1}_{00}-(1-\xi\cdot S^{(B)})^{-1}_{0(a-b)}\right|&\prec\sum_{k=1}^{O(L)}\Big(\frac{C}{k+1}+\eta^{-\frac12}L^{-2}\Big)\prec C'+\eta^{-\frac12}L^{-1}
\end{align*}
with $C'=O(1)$. (Note that $|1-\xi|\geq c\eta$ for some $c>0$ and either choice of $\xi=|m|^{2},m^{2}$, so the derivative bounds in Lemma \ref{lem_propTH} can be applied.) The same holds if replace $a-b$ by $a'-b'$. Thus, we have
\begin{align*}
 \max_{a,b,a',b'}\;\max_{ \xi=|m|^2 \text{ or } \xi=m^2}
\left| \left(\frac{1}{1-\xi\cdot S^{(B)}}\right)_{ab}-\left(\frac{1}{1-\xi\cdot S^{(B)}}\right)_{a'b'}\right|\prec 1+\eta^{-\frac12}L^{-1}.
\end{align*}
Recalling that $\eta^{-1/2}\geq L$, the dominant term on the right-hand side is $O(\eta^{-1/2}L^{-1})$. We deduce that
\begin{align*}
    \;& \eta^2 \mathbb{E}  \tr \left(\operatorname{Im} G(z)\left(E_a-N^{-1}\right) \operatorname{Im} G(z)\left(E_a-N^{-1}\right)\right)
 \\
 \le \;&   W^{\delta}W^{-6}L^{-6}\eta^{-1}+{\eta^{\frac32}W^{-2+\delta}}L^{-1}\\\nonumber
\le \;&    N^{-2} \big [ W^{\delta}{N^{2}W^{-6}L^{-6}\eta^{-1}}  +  {N^{2}W^{-2+\delta}\eta^{\frac32}L^{-1}}  \big ] \leq   W^{\delta} N^{-2} \big [ N^{-\tau/3-\delta}  + N^{-3\tau/2}  \big ] 
<   N^{-2}  N^{-\tau/3}.
\end{align*}
Above, we used $N=W^{2}L^{2}$ and $\eta=N^{-1-\tau}W^{2/3}$ and $W\geq N^{\mathfrak{c}}\geq N^{2\tau+\delta}$ to deduce the last two inequalities. Combining the above display with \eqref{ssfa2} and \eqref{que0} gives

\begin{align}\label{que2}
 \E\Big(\sum_{k,\ell} \textbf{1}(\lambda_{k},\lambda_{\ell}\in {\cal J}_E)\cdot \left|N(\boldsymbol{\psi}_{k}^*\left(E_a-N^{-1}\right) \boldsymbol{\psi}_{\ell})\right|^2\Big)
 \le N^{-\tau/3}.
\end{align}
The bound \eqref{Meq:QUE} now follows by \eqref{que2} and the Markov inequality.}
\end{proof}

 \subsection{Universality} 
For any $k=1,\ldots,N$, the $k$-point correlation function of $H$ is the following marginal, where $\rho^{(N)}_{H}(\alpha_{1},\ldots,\alpha_{N})$ denotes the joint density for the (unordered) eigenvalues of $H$:
\[
\rho_H^{(k)}\left(\alpha_1, \alpha_2, \ldots, \alpha_k\right):=\int_{\mathbb{R}^{N-k}} \rho_H^{(N)}\left(\alpha_1, \alpha_2, \ldots, \alpha_N\right) \mathrm{d} \alpha_{k+1} \cdots \mathrm{~d} \alpha_N.
\]

 \begin{theorem}[Bulk universality]\label{Thm: B_Univ}Suppose that the assumptions of Theorem \ref{MR:decol} are satisfied. Fix any integer $k\geq1$ that is independent of $N$. Fix any $|E| \leq 2-\kappa$ and test function $\mathcal{O}\in\mathcal{C}^{\infty}_{c}(\R^{k})$. We have
 \begin{equation}\label{eq:universality}
\lim _{N \rightarrow \infty} \int_{\mathbb{R}^k} \mathrm{~d} \boldsymbol{\alpha} \mathcal{O}(\boldsymbol{\alpha})\left\{\left(\rho_H^{(k)}-\rho_{\mathrm{GUE}}^{(k)}\right)\left(E+\frac{\boldsymbol{\alpha}}{N}\right)\right\}=0.
 \end{equation}  
 \end{theorem}
 \begin{proof}[Proof of Theorem \ref{Thm: B_Univ}]
This argument largely follows the proof of Theorem 2.6 in \cite{YY_25}. In particular, it follows from the generalized QUE (Theorem \ref{MR:QUE}) and other minor details. Let us explain this below. 

The starting point is the following matrix OU process:
\begin{align*}
    \mathrm{d} {\mathbf{H}}_t=-\frac{1}{2} {\mathbf{H}}_t \mathrm{~d} t+\frac{1}{\sqrt{N}} \mathrm{~d} B_t.
\end{align*}
Above, the initial data is given by the band matrix in question, i.e. ${\mathbf{H}}_{0}=H$; we use bold notation in order to not conflict with a forthcoming time-dependent random band matrix. Also, $B_{t}$ is a standard Brownian motion in the space of complex Hermitian matrices.

Note that ${\mathbf{H}}_{\infty}:=\lim_{t\to\infty}{\mathbf{H}}_{t}$ has the law of a GUE random matrix. In particular, to complete the proof, it suffices to show the following two limits, in which $t_{*}=N^{-1+\tau_{*}}$ with a fixed (small) $\tau_{*}>0$:
\begin{align}
\lim _{N \rightarrow \infty} \int_{\mathbb{R}^k} \mathrm{~d} \boldsymbol{\alpha} \mathcal{O}(\boldsymbol{\alpha})\left\{\left(\rho_{{\mathbf{H}}_{t_*}}^{(k)}-\rho_{{\mathbf{H}}_\infty}^{(k)}\right)\left(E+\frac{\boldsymbol{\alpha}}{N}\right)\right\}&=0\label{1infyuniv}\\
\lim _{N \rightarrow \infty} \int_{\mathbb{R}^k} \mathrm{~d} \boldsymbol{\alpha} \mathcal{O}(\boldsymbol{\alpha})\left\{\left(\rho_{{\mathbf{H}}_{t_*}}^{(k)}-\rho_{{\mathbf{H}}_{0}}^{(k)}\right)\left(E+\frac{\boldsymbol{\alpha}}{N}\right)\right\}&=0.\label{univ-main}
\end{align}
The first limit \eqref{1infyuniv} follows by Theorem \ref{MR:locSC} (the local law for $H$) and Theorem 2.2 of \cite{LANDON20191137}. We spend the rest of this argument showing \eqref{univ-main}. Standard calculations for comparison of correlation functions (see Theorem 15.3 in \cite{erdHos2017dynamical} and Proposition 4.17 in \cite{Xu:2024aa}) reduce the proof of \eqref{univ-main} to showing the following claim.

\bigskip

\noindent \emph{Claim}: Fix a small $\delta_{0}>0$ and a positive integer $n$. Consider $z_{j}=E_{j}+i\eta_{j}$ for $j=1,\ldots,n$, where $|E_{j}-E|=O(N^{-1})$ and $N^{-1-\delta_{0}}\leq \eta_{j}\leq N^{-1+\delta_{0}}$ are deterministic. Under the assumptions of Theorem \ref{Thm: B_Univ}, there exist $c',C_{n}>0$ such that if $\tau_{*},\delta_{0}>0$ are small enough, then
\begin{align}\label{417}
\sup _{0 \leq t \leq t_*}\left|\mathbb{E} \prod_{i=1}^n \operatorname{Im} m_t(z_i)-\mathbb{E} \prod_{i=1}^n \operatorname{Im} m_{t_*}(z_i)\right| \leq N^{-c'+C_n \delta_0+\tau_*}
\end{align}
where $m_{t}(z_{i})=N^{-1}\tr({\mathbf{H}}_{t}-z_{i})^{-1}$ for any $t\geq0$. 

\bigskip

We now also claim that if $\tau_{*}>0$ is small enough, then Theorems \ref{MR:decol}, \ref{MR:locSC}, and \ref{MR:QUE} all hold for ${\mathbf{H}}_{t}$ for any fixed $t\in[0,t_{*}]$ and any $z=E+i\eta$ with $\eta\asymp N^{-1+\tau_{*}}$; we recall that $t_{*}=N^{-1+\tau_{*}}$. We will explain this at the end of this proof. 

Assuming that this is true, we can argue as in Step 3 in the proof of Theorem 2.6 in \cite{YY_25} to obtain the following estimate (which appears as (2.25) in \cite{YY_25}):
\begin{align}\nonumber
& \sup _{0 \leq t \leq t_*}\left|\mathbb{E} \prod_{i=1}^n \operatorname{Im} m_t(z_i)-\mathbb{E} \prod_{i=1}^n \operatorname{Im} m_{t_*}(z_i)\right|
\\\label{EMCTE2}
&\prec  N^{-1+\tau_*+C_n \delta_0}\cdot  \max_{0\le t\le t_*}\max_{u\ne v}\,\mathbb{E}\left[  L_{1, t}(z_u)  +  L_{2, t}(z_u, z_v)\right].
\end{align}
Above, $C_{n}=O(1)$ and $\delta_{0}>0$ is small. Also, we have used the notation below, where $S_{xy}^{\circ}:=S_{xy}-N^{-1}$ and $\mathbf{R}_{t}(z):=({\mathbf{H}}_{t}-z)^{-1}$:
\begin{align*}
L_{1, t}(z) & :=\sum_{\mathbf{G}_1, \mathbf{G}_2 \in\left\{{\mathbf{R}}_t, {\mathbf{R}}_t^*\right\}}\left|\frac{1}{N} \sum_{a, b}\left(\mathbf{G}_1^2(z)\right)_{a a} S_{a b}^{\circ}\left(\mathbf{G}_2(z)\right)_{b b}\right|, \\
L_{2, t}\left(z_1, z_2\right) & :=\sum_{\mathbf{G}_1, \mathbf{G}_2 \in\left\{{\mathbf{R}}_t, {\mathbf{R}}_t^*\right\}}\left|\frac{1}{N^2} \sum_{a, b}\left(\mathbf{G}_1^2\left(z_1\right)\right)_{a b} S_{a b}^{\circ}\left(\mathbf{G}_2^2\left(z_2\right)\right)_{b a}\right|.
\end{align*}
In particular, in order to prove \eqref{417}, it suffices to use \eqref{EMCTE2} and show the following two estimates for any $\mathbf{G}_{1}(z),\mathbf{G}_{2}(z)\in\{\mathbf{R}_{t},\mathbf{R}_{t}^{*}\}$, in which $C=O(1)$ and $\delta_{0}>0$ is small, and where $\mathfrak{c}>0$ is from \eqref{Main_DEL_COND}:
\begin{align}\label{jaklsdufowe}
  & \max_i\max_y\; \mathbb E\left| \sum_{ x}\left(\mathbf{G}_1^2(z_i)\right)_{x x} S_{x y}^{\circ}\left(\mathbf{G}_2(z_i)\right)_{yy}\right|\le N^{1-\mathfrak c/18+C\delta_0} \\\label{uywy7723r3rf}
 & \max_{i\ne j}\max_y\;\mathbb E\left|  \sum_{  x}\left(\mathbf{G}_1^2\left(z_i\right)\right)_{xy} S_{x y}^{\circ}\left(\mathbf{G}_2^2\left(z_j\right)\right)_{yx}\right| \prec N^{2-\mathfrak c/18+C\delta_0}.
\end{align}
The rest of this proof now proceeds with showing \eqref{jaklsdufowe}, showing \eqref{uywy7723r3rf}, and then showing that Theorems \ref{MR:decol}, \ref{MR:locSC}, and \ref{MR:QUE} all hold for ${\mathbf{H}}_{t}$ for any fixed $t\in[0,t_{*}]$ and any $z=E+i\eta$ with $\eta\geq N^{-1+\tau_{*}}$.
\subsubsection*{Proof of \eqref{jaklsdufowe}}
We cite Lemma 4.20 in \cite{Xu:2024aa} (see also (2.28) in \cite{YY_25}); by the eigen-decomposition for $G_{t}$, we have the following:
\begin{equation}\label{yw982823}
        \frac 1 N  \left|\sum_{x}\left(\mathbf{G}_1^2(z)\right)_{xx} S_{x y}^{\circ}\left(\mathbf{G}_2(z)\right)_{yy}\right|\prec N^{-3}\sum_{\alpha,\beta} |p_\alpha(z)|^2 \cdot |p_\beta(z)|\cdot \left| M_{y,\alpha} \right|.
\end{equation}
Above, $p_{\alpha}(z):=(\lambda_{\alpha}-z)^{-1}$, and for any $y,\alpha$, we have the following (in which $a_{0}$ is such that $y\in{\cal I}^{(2)}_{a_{0}}$):
\begin{align*}
     |M_{y,\alpha}|:=\left|N \sum_{x}|\boldsymbol{\psi}_\alpha(x)|^2 S^0_{xy}\right|=O(N) \left|\sum_{a: |a-a_0|_{L}\le 1}
     \left\langle \boldsymbol{\psi}_\alpha \left(E_a-N^{-1} I\right) \,\boldsymbol{\psi}^*_\alpha\right\rangle\right|.
\end{align*}
At this point, we can directly follow the proof of (2.26) in \cite{YY_25} (see Step 4 of the proof of Theorem 2.6 in \cite{YY_25}). Indeed, \eqref{yw982823} expresses the left-hand side of \eqref{jaklsdufowe} in terms of eigenvalues and eigenvectors, which we control using Theorems \ref{MR:decol}, \ref{MR:locSC}, and \ref{MR:QUE}. So, instead of providing the details verbatim, let us briefly describe this argument. 

Recall the fixed energy $E$ in the statement of Theorem \ref{Thm: B_Univ}. Let ${\cal B}$ be the bad event where there exists $\alpha$ with $|\lambda_{\alpha}-E|\leq N^{-1+\mathfrak{c}/6}$ and $|M_{y,\alpha}|\geq N^{-\mathfrak{c}/18}$. The probability of this bad event is $O(N^{-\mathfrak{c}/18})$ by \eqref{Meq:QUE} for $\tau=\mathfrak{c}/3$. If we restrict \eqref{yw982823} to this bad event ${\cal B}$, the expectation of this quantity is $O(N^{-\mathfrak{c}/18+C\delta_{0}})$, since the local law and delocalization (Theorems \ref{MR:decol} and \ref{MR:locSC}) imply that the right-hand side of \eqref{yw982823} is $\mathrm{O}_{\prec}(N^{C\delta_{0}})$; here, $C=O(1)$ and $\delta_{0}>0$ is small. On the good event ${\cal B}^{c}$, we can improve our estimate on the right-hand side of \eqref{yw982823} using $|M_{y,\alpha}|\leq N^{-\mathfrak{c}/18}$ for all $\alpha$ with $|\lambda_{\alpha}-E|\leq N^{-1+\mathfrak{c}/6}$ and using the local law in Theorem \ref{MR:locSC}. In particular, if we restrict \eqref{yw982823} to ${\cal B}^{c}$ and take expectation, the resulting quantity is $O(N^{-\mathfrak{c}/18+C\delta_{0}})$ as well. Ultimately, \eqref{jaklsdufowe} follows.
\subsubsection*{Proof of \eqref{uywy7723r3rf}}
We again use Lemma 4.20 in \cite{Xu:2024aa} to get the following:
\begin{align}
N^{-2} \left|  \sum_{  x}\left(\mathbf{G}_1^2\left(z_i\right)\right)_{xy} S_{x y}^{\circ}\left(\mathbf{G}_2^2\left(z_j\right)\right)_{yx}\right|
\prec N^{-4} \sum_{\alpha,\beta} |p_\alpha(z_i)|^2|p_\beta(z_j)|^2 \left|M_{y,\alpha,\beta}\right|,\label{yw982823-2}
\end{align}
where, if we again let $a_{0}$ be such that $y\in{\cal I}_{a_{0}}^{(2)}$, we have
\begin{align*}
|M_{y,\alpha,\beta}|:= N\left| \sum_{x}\boldsymbol{\psi}_\alpha(x) \overline{\boldsymbol{\psi}_{\beta}(x)}\cdot S^0_{xy}\right|= O(N)\left| \sum_{a: |a-a_0|_{L}\le 1}
     \left\langle \boldsymbol{\psi}_\alpha \left(E_a-N^{-1} I\right) \,\boldsymbol{\psi}^*_\beta\right\rangle\right|.
\end{align*}
At this point, we can again directly follow the proof of (2.27) in \cite{YY_25} (see Step 4 in the proof of Theorem 2.6 in \cite{YY_25}). In particular, we have reduced the estimation of the left-hand side of \eqref{uywy7723r3rf} to estimation of eigenvalues and eigenvectors, which we do using Theorems \ref{MR:decol}, \ref{MR:locSC}, and \ref{MR:QUE}. So, we do not reproduce the details here. (Roughly speaking, the additional factor of $|p_{\beta}(z_{j})|$ appearing in \eqref{yw982823-2} compared to \eqref{yw982823} is responsible for the extra factor of $N$ on the right-hand side of \eqref{uywy7723r3rf} compared to \eqref{jaklsdufowe}; this is by the local law in Theorem \ref{MR:locSC}).
\subsubsection*{Proof of Theorems \ref{MR:decol}, \ref{MR:locSC}, and \ref{MR:QUE} for ${\mathbf{H}}_{t}$ with $t= N^{-1+\tau_{*}}$ and $\mathrm{Im}(z)\asymp N^{-1+\tau_{*}}$}
First, we note that it is enough to prove Theorems \ref{MR:locSC} and \ref{MR:QDiff} for ${\mathbf{H}}_{t}$, since the proofs of Theorems \ref{MR:decol} and \ref{MR:QUE} follow from Theorems \ref{MR:locSC} and \ref{MR:QDiff}. Before we proceed, we note that the argument below relies on the structure of the proofs of Theorems \ref{MR:decol}, \ref{MR:locSC}, and \ref{MR:QUE} for the original band matrix $H$ itself. Thus, familiarity with at least the rest of Section \ref{sec_model-results} is required to following the argument below.

For any $t\in[0,N^{-1+\tau_{*}}]$, ${\mathbf{H}}_{t}$ is a band matrix with profile $(S_{t})_{xy}:=\E|({\mathbf{H}}_{t})_{xy}|^{2}=e^{-t}S_{xy}+N^{-1}(1-e^{-t})$. For the rest of this argument, it will be convenient to define $\zeta:=1-\exp(-N^{-1+\tau_{*}})$, since the variable $t$ will be used as another time parameter below. In the same spirit, we will also set $\tilde{\mathbf{H}}_{\zeta}:=\mathbf{H}_{t}$.

Since $\mathrm{Im}(z)=\eta\asymp N^{-1+\tau_{*}}$, by Lemma \ref{zztE}, we can find $t_{0}$ such that $1-t_{0}\asymp N^{-1+\tau_{*}}$ such that 
\begin{align*}
z=t_{0}^{-\frac12}z_{t_{0}}^{(E)} \quad\text{and}\quad (\tilde{\mathbf{H}}_{\zeta}-z)^{-1}\sim t_{0}^{\frac12}({\cal H}_{t_{0}}-z_{t_{0}}^{(E)})^{-1},
\end{align*}
where $z_{t}^{(E)}$ is from Definition \ref{def_flow}, and the $\sim$ notation above means identity in law. Moreover, the matrix ${\cal H}_{t}$ is the solution to the following matrix SDE with initial condition ${\cal H}_{0}=0$, whose law at time $t_{0}$ matches that of $t_{0}^{1/2}\tilde{\mathbf{H}}_{\zeta}$:
\begin{align*}
d{\cal H}_{t,xy}&=\begin{cases}\sqrt{S_{ij}}d{B}_{t,xy}&t\leq t_{1}:=(1-\zeta)t_{0}\\N^{-1/2}d{B}_{t,xy}&t\geq t_{1}\end{cases},
\end{align*}
We clarify that the bound $1-t_{0}\asymp N^{-1+\tau_{*}}$ follows from the fact that $\eta_{t}:=\mathrm{Im}(z_{t}^{(E)})\asymp 1-t$ by \eqref{eta}. 

Next, we use Lemmas \ref{ML:GLoop}, \ref{ML:GLoop_expec}, \ref{ML:GtLocal}; this gives \eqref{Eq:L-KGt}, \eqref{Eq:L-KGt2}, \eqref{Eq:Gdecay}, and \eqref{Gt_bound} at time $t_{1}$. We now claim that we can use Theorem \ref{lem:main_ind} to deduce that Lemmas \ref{ML:GLoop}, \ref{ML:GLoop_expec}, \ref{ML:GtLocal} hold at time $t_{0}$. Assuming this is true, as noted in Step 6 of the proof of Theorem \ref{lem:main_ind} (see Section \ref{subsection:step6} for details), these bounds provide Lemma \ref{ML:exp} at time $t_{0}$ as well. As explained after Lemma \ref{ML:exp}, this provides Theorems \ref{MR:locSC} and \ref{MR:QDiff}, which would complete the proof.

We now explain why we can use Theorem \ref{lem:main_ind} with $s,t$ therein equal to $t_{1},t_{0}$. Our method of proof for Theorem \ref{lem:main_ind} is based on the equations \eqref{eq:mainStoflow} and \eqref{pro_dyncalK}. In these equations, when we look at times $t>t_{1}$, we must replace $S^{(B)}$ with
\begin{align*}
\Big(S^{(B)}_{GUE}\Big)_{ab}:=L^{-2},\quad a,b\in\Z_{L}^{2}.
\end{align*}
Indeed, this is because the covariance matrix for $d{\cal H}_{t}$ for $t>t_{1}$ is given by $S^{(B)}_{GUE}\otimes S_{W}$ instead of $S^{(B)}\otimes S_{W}$. 

Our analysis of the first SDE \eqref{eq:mainStoflow} is based on estimates for the solution ${\cal K}$ to the second ODE \eqref{pro_dyncalK}. (It also requires the identity \eqref{WI_calL}, but this follows from the Ward identity for resolvents and does not need anything about the SDE \eqref{eq:mainStoflow} itself.) Our estimates for ${\cal K}$ are based on control over the following resolvent for $\xi=|m|^{2},m^{2}$:
\begin{align*}
\Theta^{(B),\text{new}}_{t\xi}:=\Big(1-t_{1}\xi\cdot S^{(B)}-(t-t_{1})\xi\cdot S^{(B)}_{GUE}\Big)^{-1},\quad t\in[t_{1},t_{0}].
\end{align*}
Let us make more precise what control we need on this and why it appears. Our analysis of the solution ${\cal K}$ to \eqref{pro_dyncalK} is based on an explicit representation of its solution; see \eqref{Kn2sol}, \eqref{Kn2sol2}, and \eqref{KKpi}. For times $t\leq t_{1}$, we use $\Theta^{(B)}$ from Definition \ref{def_Theta}, since the ${\cal K}$ equation \eqref{pro_dyncalK} uses $S^{(B)}$ for times $t\leq t_{1}$, and it uses $S^{(B)}_{GUE}$ for times $t\geq t_{1}$. Let us now explain what we require from $\Theta^{(B),\text{new}}$. In a nutshell, we require that $\Theta^{(B),\text{new}}_{t\xi}$ satisfies the estimates from Lemma \ref{lem_propTH}. More precisely:
\begin{itemize}
\item Note that $t_{1}\xi\cdot S^{(B)}-(t-t_{1})\xi\cdot S^{(B)}_{GUE}$ has the form $t\xi\tilde{S}$, where $\tilde{S}$ is a stochastic transition matrix. Since $|\xi|\leq1$ and $t<1$, the expansion below converges absolutely:
\begin{align*}
\Theta^{(B),\text{new}}_{t\xi}=\sum_{k=0}^{\infty}\Big(t\xi\cdot S^{(B)}+(t-t_{1})\xi\cdot S^{(B)}_{GUE}\Big)^{k}.
\end{align*}
Moreover, since $S^{(B)},S^{(B)}_{GUE}$ are self-adjoint matrices, so is $\Theta^{(B),\text{new}}_{\xi}$. This representation also implies 
\begin{align*}
\Big(\Theta^{(B),\text{new}}_{t\xi}\Big)_{ab}&=\Big(\Theta^{(B),\text{new}}_{t\xi}\Big)_{(a+s)(b+s)},\quad a,b,s\in\Z_{L}^{2}\\
[S^{(B)},\Theta^{(B),\text{new}}_{t\xi}]&=[S^{(B)}_{GUE},\Theta^{(B),\text{new}}_{t\xi}]=[\Theta^{(B),\text{new}}_{t'\xi},\Theta^{(B),\text{new}}_{t\xi}]=0,\quad t,t'\in[0,1)
\end{align*}
where the last line follows because $S^{(B)},S^{(B)}_{GUE}$ commute with each other.
\item First note that $\sum_{b}(1-t\xi\tilde{S})^{-1}_{ab}\leq(1-t)^{-1}=O(\eta_{t}^{-1})$ for any stochastic matrix $\tilde{S}$. If we use this and a standard resolvent expansion, we have 
\begin{align*}
\max_{a,b}|(\Theta^{(B),\text{new}}_{t\xi})_{ab}|&\leq \max_{a,b}|(\Theta^{(B)}_{t\xi})_{ab}|+O(N^{-1+\tau_{*}}\eta_{t}^{-2}L^{-2}),
\end{align*}
Now, for $t\geq t_{1}$, we have $\eta_{t}\asymp1-t\leq 1-t_{1}\asymp N^{-1+\tau_{*}}$. Thus, if $\tau_{*}$ is small enough, we have 
\begin{align*}
\ell_{t}:=\min(\eta_{t}^{-1/2},L)\asymp L,
\end{align*}
since $W\geq N^{\mathfrak{c}}$. In particular, if we use Lemma \ref{lem_propTH} and $\eta_{t}\geq \eta_{t_{0}}\asymp N^{-1+\tau_{*}}$, we deduce that
\begin{align*}
\max_{a,b}|(\Theta^{(B),\text{new}}_{t\xi})_{ab}|\prec \eta_{t}^{-1}\ell_{t}^{-2}.
\end{align*}
\item Similarly, one can use expansions and $\Theta^{(B)}$ estimates from Lemma \ref{lem_propTH} to obtain that \eqref{prop:BD1}-\eqref{prop:BD2} hold if we replace $\Theta^{(B)}$ by $\Theta^{(B),\text{new}}$ therein (for $t\geq t_{1}$ and $\eta_{t}\asymp N^{-1+\tau_{*}}$).
\end{itemize}
In particular, the results of Lemma \ref{lem_propTH} hold for $\Theta^{(B),\text{new}}$ and if we restrict to $\xi=t|m|^{2},tm^{2}$ therein with $t\geq t_{1}$. (This is except for the exponential pointwise decay in \eqref{prop:ThfadC}. However, as we mentioned above, we have $\ell_{t}\asymp L$ for $t\geq t_{1}$. In particular, the exponential pointwise decay is only used to tame factors like $\ell_{t}/\ell_{t'}$ for $t\geq t'$; see \eqref{lRB1}, for example. But for $t,t'\geq t_{1}$, these factors are equal to $1$ and thus harmless.) In particular, we can use Theorem \ref{lem:main_ind} with $s=t_{1}$ and $t=t_{0}$, and the proof is complete.
\end{proof}

\subsection{Stochastic flow and the loop hierarchy}
As in \cite{YY_25}, the key to the proofs of Theorems \ref{MR:locSC} and \ref{MR:QDiff} are estimates on a hierarchy of ``$G$-loops". Besides a dimension-dependent scaling, the structure of these loops and the hierarchy are the same as in \cite{YY_25}. Thus, much of the rest of Section \ref{sec_model-results} follows \cite{YY_25} almost identically. We start with the matrix Brownian motion
\begin{align*}
dH_{t,xy}=\sqrt{S_{xy}}dB_{t,xy}, \quad H_{0}=0, 
\end{align*}
where $B_{t,xy}$ are independent standard complex Brownian motions for all $x,y\in \Z_{WL}^{2}$ up to the constraint  $B_{yx} = \bar  B_{xy}$. Along this flow, we will consider a time-dependent spectral parameter $z_{t}$ as defined below.

\begin{definition}\label{def_flow}
For any $E \in \mathbb R$,  define $
m^{(E)}:=\lim_{\epsilon\to 0+}m_{sc}(E+i \epsilon)$, and consider the following objects:
\begin{align*}
  z^{(E)}_t&=E+(1-t) m^{(E)}  ,\quad 0\le t\le 1 \\
  G_{t}^{(E)}&:=(H_{t}-z^{(E)}_{t})^{-1}.
\end{align*}
Note that the imaginary part of $z_{t}^{(E)}$ is given by the following formula:
\begin{align}\label{eta}
\eta_t := \im  z^{(E)}_t =  (1-t)  \im m^{(E)}.
\end{align}
\end{definition}
 
By the It\^o formula, the resolvent $G_{t} := G_{t}^{(E)}$ solves the following matrix SDE:
\begin{align}\nonumber
dG_{t}=-G_{t}dH_{t}G_{t}+G_{t}\{\mathcal{S}[G_{t}]-m^{(E)}\}G_{t}
dt.
\end{align}
Above, the linear operator $\mathcal{S}:M_{N}(\C)\to M_{N}(\C)$ is defined as
\begin{align*}
\mathcal{S}[X]_{xy}:=\delta_{xy}\sum_{v=1}^{N}S_{xv}X_{vv}.
\end{align*}
The following will help us recover $G(z)$ from $G^{(E)}_{t}$; its proof is the same as Lemma 2.7 in \cite{YY_25}.
\begin{lemma}\label{zztE}
    Suppose that $z\in \mathbb C$ satisfies $0 <  \im z\le 1$ and $|\re z|\le  2-\kappa$ for some $\kappa>0$. Then there exists an $E\in\R$ such that $|E|\le 2-\kappa$ and $0\le  t<1$ and
 \begin{equation}\label{eq:zztE}
     z=t^{-1/2}\cdot z^{(E)}_t.
 \end{equation}
 Moreover, there  exists a fixed constant $c_\kappa>0$ such that 
  \begin{equation}\label{eq:zztE2}
       c_\kappa\le t\quad {\rm and}\quad  c_\kappa\im z\le \im z_t\le c_\kappa^{-1}\im z.
  \end{equation} 
Finally, if $z,E$ and $t$ satisfy the relation \eqref{eq:zztE}, then we have 
\begin{equation}\label{mtEmz}
  m_{sc}(z)=t^{1/2}\cdot m^{(E)} 
\end{equation}
as well as the distributional identity
   \begin{equation}\label{GtEGz}
    G(z) \sim t^{ 1/2}\cdot G_t^{(E)} 
   \end{equation}
\end{lemma}

For the rest of the paper, we will consider only  $G_t^{(E)}$.  We now define the aforementioned $G$-loops. 
\begin{definition}\label{Def:G_loop}
 For any fixed $E\in\R$, consider the following resolvent as a function of $+,-$ signs:
 \begin{equation}\nonumber
  G^{(E)}_{t }(\alpha):=  G^{(E)}_{t,\alpha} =
   \begin{cases}
       (H_t-z_t)^{-1}, \quad \alpha=+\\
        (H_t-\bar z_t)^{-1}, \quad \alpha=-
   \end{cases}.
 \end{equation}
By definition, $G^{(E)}_{t,+} =\left(G^{(E)}_{t,-}\right)^\dagger$. 
For simplicity,  we sometimes write 
$G_t=G^{(E)}_{t,+}$. Now, fix indices
$$
\boldsymbol{\sigma}=(\sigma_1, \sigma_2, \cdots \sigma_n),\quad 
\boldsymbol{a}=(a_1, a_2, \cdots a_n)
,\quad \sigma_i \in \{+,-\}, \quad
a_i\in \mathbb Z_L^{2}, \quad 1\le i\le n.
$$
Next, recall $E_a$  defined in \eqref{Def_matE}. We define the corresponding ``$n$-$G$ loop" to be the following complex number: 
\begin{equation}\label{Eq:defGLoop}
    {\cal L}_{t, \boldsymbol{\sigma}, \textbf{a}}=\left\langle \prod_{i=1}^n G_t(\sigma_i) E_{a_i}\right\rangle,\quad 
\text{where} \quad \left\langle A\right\rangle=\tr A.
\end{equation}
We also define 
\begin{equation}\label{def_mtzk}
m (\sigma ):=m^{(E)}(\sigma ) = \begin{cases}
     m^{(E)}, & \sigma  =+ \\
     \overline{m^{(E)}} , & \sigma = -
\end{cases},
\end{equation}
and
\begin{equation}\label{Eq:defwtG}
 \widetilde{G}_t(\sigma_k) := G_t(\sigma_k) - m (\sigma_k) .
\end{equation}
\end{definition}

With these notations, we can rewrite the quantity 
$\tr G E_a G^\dagger  E_b$ in Theorem \ref{MR:QDiff} as follows:
$$
\tr G E_a G^\dagger  E_b=t \cdot \tr (G_{t,+})\cdot E_a\cdot (G_{t,-})\cdot  E_b=t\cdot {\cal L}_{t, (+,-),(a,b)}.
$$
We now define some operations on loops; these will be important for constructing a hierarchy of loops.

\begin{definition}\label{Def:oper_loop}
Take a loop of the following form:
\begin{equation}\label{taahCal}
   {\cal L}_{t, \boldsymbol{\sigma}, \textbf{a}} = \left\langle \prod_{i=1}^n G_t(\sigma_i) E_{a_i} \right\rangle, \quad \sigma_i \in \{+, -\}, \quad 1 \le i \le n 
\end{equation}
 
 \noindent 
 \emph{1.} 
 For $1 \le k \le n$, we define a ``cut-and-glue" operator ${\cal G}^{(a)}_{k}$ as follows. Its action on a loop ${\cal L}_{t,\boldsymbol{\sigma},\textbf{a}}$ is denoted by ${\cal G}^{(a)}_{k} \circ {\cal L}_{t, \boldsymbol{\sigma}, \textbf{a}}$, and it is defined to be the loop obtained by replacing $G_t(\sigma_k)$ by $G_{t}(\sigma_k) E_a G_{t}(\sigma_k)$.
 
  In words, ${\cal G}^{(a)}_{k}$ cuts the $k$-th edge $G_t(\sigma_k)$ of the loop ${\cal L}_{t,\boldsymbol{\sigma},\textbf{a}}$ and glues the two new ends with $E_a$. The new loop is a unit longer than the input loop. Since ${\cal G}^{(a)}_{k}$ turns a loop into another loop, it is an operator on the index space of $\boldsymbol{\sigma},\textbf{a}$. Thus, we sometimes write ${\cal L}_{t, \;  {\cal G}^{(a)}_{k} (\boldsymbol{\sigma}, \textbf{a})}:={\cal G}^{(a)}_{k} \circ {\cal L}_{t, \boldsymbol{\sigma}, \textbf{a}}$.
    
   \begin{figure}[ht]
        \centering
        \begin{tikzpicture}
 
    \coordinate (a1) at (0, 0);
    \coordinate (a2) at (2, 0);
    \coordinate (a3) at (2, -2);
    \coordinate (a4) at (0, -2);

    \draw[thick, blue] (a1) -- (a2) -- (a3) -- (a4) -- cycle;
    \fill[red] (a1) circle (2pt) node[above left] {$a_1$};
    \fill[red] (a2) circle (2pt) node[above right] {$a_2$};
    \fill[red] (a3) circle (2pt) node[below right] {$a_3$};
    \fill[red] (a4) circle (2pt) node[below left] {$a_4$};
    
    \node[above, blue] at (1, 0.1) {$G_2$};
    \node[right, blue] at (2.1, -1) {$G_3$};
    \node[below, blue] at (1, -2.1) {$G_4$};
    \node[left, blue] at (-0.1, -1) {$G_1$};

    \draw[->, thick] (3, -1) -- (5, -1);

    \coordinate (b1) at (7, 0);
    \coordinate (b2) at (9, 0);
    \coordinate (b3) at (9, -2);
    \coordinate (b4) at (7, -2);
    \coordinate (a) at (8, 1);  

    \draw[thick, blue] (b1) -- (a) -- (b2) -- (b3) -- (b4) -- cycle;
    \fill[red] (b1) circle (2pt) node[above left] {$a_1$};
    \fill[red] (b2) circle (2pt) node[above right] {$a_2$};
    \fill[red] (b3) circle (2pt) node[below right] {$a_3$};
    \fill[red] (b4) circle (2pt) node[below left] {$a_4$};
    \fill[red] (a) circle (2pt) node[above] {$a$};

    \node[above left, blue] at (7.5, 0.5) {$G_2$};
    \node[above right, blue] at (8.5, 0.5) {$G_2$};
    \node[right, blue] at (9.1, -1) {$G_3$};
    \node[below, blue] at (8, -2.1) {$G_4$};
    \node[left, blue] at (6.9, -1) {$G_1$};  

   \end{tikzpicture}     
   \caption{Illustration of 
${\cal G}^{(a)}_{k}$ for $n=4$. Taken from \cite{YY_25}.}
        \label{fig:op_gka}
 \end{figure}
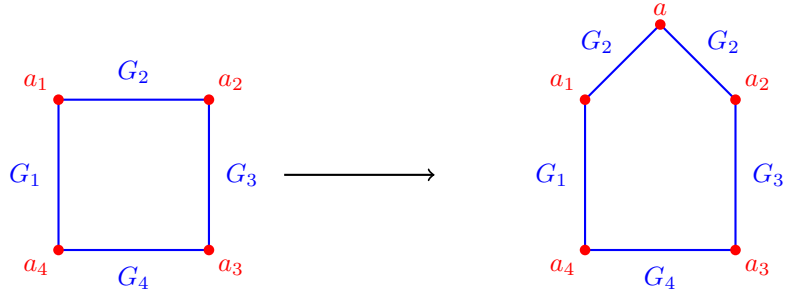

 \noindent 
 \emph{2.}  For $1 \le k < l \le n$, we define a ``left" cut-and-glue operator ${\cal G}^{(a), L}_{k, l}$. Let ${\cal G}^{(a), L}_{k, l} \circ {\cal L}_{t, \boldsymbol{\sigma}, \textbf{a}}$ be the following loop. Cut the $k$-th and $l$-th edges $G_t(\sigma_k)$ and $G_t(\sigma_l)$ (creating four end points and two ``chains"). Glue the two new ends of the chain that contains $E_{a_n}$ and insert a new $E_a$ at the gluing point.
    
    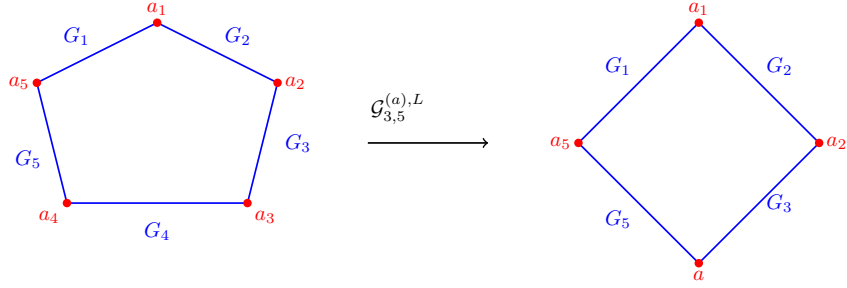
\begin{figure}[ht]
        \centering
         \scalebox{0.8}{
       \begin{tikzpicture}
        \coordinate (a1) at (0, 2);
        \coordinate (a2) at (2, 1);
        \coordinate (a3) at (1.5, -1);
        \coordinate (a4) at (-1.5, -1);
        \coordinate (a5) at (-2, 1);

        \draw[thick, blue] (a1) -- (a2) -- (a3) -- (a4) -- (a5) -- cycle;
        \fill[red] (a1) circle (2pt) node[above] {$a_1$};
        \fill[red] (a2) circle (2pt) node[right] {$a_2$};
        \fill[red] (a3) circle (2pt) node[below right] {$a_3$};
        \fill[red] (a4) circle (2pt) node[below left] {$a_4$};
        \fill[red] (a5) circle (2pt) node[left] {$a_5$};

        \node[above right, blue] at (1, 1.5) {$G_2$};
        \node[right, blue] at (2, 0) {$G_3$};
        \node[below, blue] at (0, -1.2) {$G_4$};
        \node[below left, blue] at (-1.8, 0) {$G_5$};
        \node[above left, blue] at (-1, 1.5) {$G_1$};

        \draw[->, thick] (3.5, 0) -- (5.5, 0);
        \node[above] at (4, 0.2) {$ {\cal G}^{(a),L}_{3,5}$};

        \coordinate (b1) at (9, 2);     
        \coordinate (b2) at (11, 0);     
        \coordinate (a) at (9, -2);     
        \coordinate (b5) at (7, 0);     

        \draw[thick, blue] (b1) -- (b2) -- (a) -- (b5) -- cycle;
        
        \fill[red] (b1) circle (2pt) node[above] {$a_1$};
        \fill[red] (b2) circle (2pt) node[right] {$a_2$};
        \fill[red] (a) circle (2pt) node[below] {$a$};
        \fill[red] (b5) circle (2pt) node[left] {$a_5$};

        \node[above right, blue] at (10, 1) {$G_2$};
        \node[right, blue] at (10, -1) {$G_3$};
        \node[below left, blue] at (8, -1) {$G_5$};
        \node[above left, blue] at (8, 1) {$G_1$};
    \end{tikzpicture}
    }
        \caption{Illustration of operator ${\cal G}^{(a), L}_{k, l}$. Taken from \cite{YY_25}.}
        \label{fig:op_gLeft}
    \end{figure}

  \noindent 
 \emph{3.}  For $1 \le k < l \le n$, we define a ``right" cut-and-glue operator ${\cal G}^{(a), R}_{k, l}$ the same way, but we instead glue the two new ends of the chain that does not contain $E_{a_n}$. The length of the new loop is $l - k + 1$.
     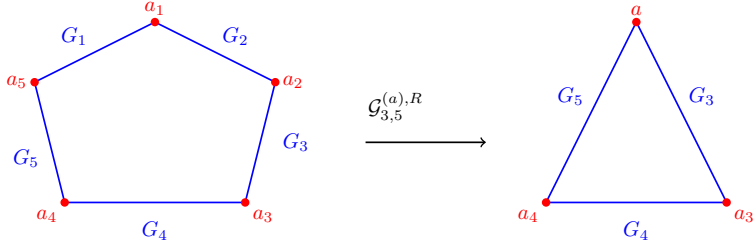
\begin{figure}[ht]
        \centering
       \scalebox{0.8}{  
    \begin{tikzpicture}
        \coordinate (a1) at (0, 2);
        \coordinate (a2) at (2, 1);
        \coordinate (a3) at (1.5, -1);
        \coordinate (a4) at (-1.5, -1);
        \coordinate (a5) at (-2, 1);

        \draw[thick, blue] (a1) -- (a2) -- (a3) -- (a4) -- (a5) -- cycle;
        \fill[red] (a1) circle (2pt) node[above] {$a_1$};
        \fill[red] (a2) circle (2pt) node[right] {$a_2$};
        \fill[red] (a3) circle (2pt) node[below right] {$a_3$};
        \fill[red] (a4) circle (2pt) node[below left] {$a_4$};
        \fill[red] (a5) circle (2pt) node[left] {$a_5$};

        \node[above right, blue] at (1, 1.5) {$G_2$};
        \node[right, blue] at (2, 0) {$G_3$};
        \node[below, blue] at (0, -1.2) {$G_4$};
        \node[below left, blue] at (-1.8, 0) {$G_5$};
        \node[above left, blue] at (-1, 1.5) {$G_1$};

        \draw[->, thick] (3.5, 0) -- (5.5, 0);
        \node[above] at (4, 0.2) {${\cal G}^{(a),R}_{3,5}$};

        \coordinate (a) at (8, 2);      
        \coordinate (a3) at (9.5, -1);  
        \coordinate (a4) at (6.5, -1);  

        \draw[thick, blue] (a) -- (a3) -- (a4) -- cycle;

        \fill[red] (a) circle (2pt) node[above] {$a$};
        \fill[red] (a3) circle (2pt) node[below right] {$a_3$};
        \fill[red] (a4) circle (2pt) node[below left] {$a_4$};

        \node[above right, blue] at (8.75, 0.5) {$G_3$};
        \node[below, blue] at (8, -1.2) {$G_4$};
        \node[above left, blue] at (7.25, 0.5) {$G_5$};

    \end{tikzpicture}
    }  
        \label{fig:op_gRight}
         \caption{Illustration of operator ${\cal G}^{(a), R}_{k, l}$. Taken from \cite{YY_25}.}
    \end{figure}  
\end{definition}

 The following result gives the evolution for a loop along the flow of $G^{(E)}_{t}$. It is derived via the It\^o formula. Throughout this paper, we let $\textbf{a}=(a_{1},\ldots,a_{n})$ be a vector of indices in $\Z_{L}^{2}$, and we let $a\in\Z_{L}^{2}$ be a single index that is separate from $\textbf{a}$. Finally, for convenience, we will use the notation $\partial_{(x,y)}:=\partial_{H_t(x,y)}$.

\begin{lemma}[The loop hierarchy] \label{lem:SE_basic}
We have the following ``loop hierarchy" SDE:
\begin{align}\label{eq:mainStoflow}
    d\mathcal{L}_{t, \boldsymbol{\sigma}, \textbf{a}} =& 
    \mathcal{E}^{(M)}_{t, \boldsymbol{\sigma}, \textbf{a}} 
    +
    \mathcal{E}^{(\widetilde{G})}_{t, \boldsymbol{\sigma}, \textbf{a}} 
    +
   W^{2}\cdot \sum_{1 \le k < l \le n} \sum_{a, b} \left( \mathcal{G}^{(a), L}_{k, l} \circ \mathcal{L}_{t, \boldsymbol{\sigma}, \textbf{a}} \right) S^{(B)}_{ab} \left( \mathcal{G}^{(b), R}_{k, l} \circ \mathcal{L}_{t, \boldsymbol{\sigma}, \textbf{a}} \right) dt. 
\end{align}
 Above, we used the notation 
  \begin{align} \label{def_Edif}
\mathcal{E}^{(M)}_{t, \boldsymbol{\sigma}, \textbf{a}} :  = & \sum_{\alpha=(x,y)} 
  \left( \partial_{  \alpha} \; {\cal L}_{t, \boldsymbol{\sigma}, \textbf{a}}  \right)
 \cdot \left(S _\alpha\right)^{1/2} \cdot
  \mathrm{d} \left({B}_t\right)_{\alpha} \\\label{def_EwtG}
  \mathcal{E}^{(\widetilde{G})}_{t, \boldsymbol{\sigma}, \textbf{a}}: = &   W^{2}\cdot \sum_{1 \le k \le n} \sum_{a, b} \; 
 \left \langle \widetilde{G}_t(\sigma_k) E_a \right\rangle
  \cdot S^{(B)}_{ab} \cdot
  \left( {\cal G}^{(b)}_{k} \circ {\cal L}_{t, \boldsymbol{\sigma}, \textbf{a}} \right) dt ,\quad \quad \widetilde{G}_t = G_t - m.
\end{align}
  \end{lemma}
  
The factor of $W^{2}$ in ${\cal E}^{(\widetilde{G})}$ comes from the normalization that $E_{a}$ has a factor of $W^{-2}$. In order to solve this hierarchy of SDEs, we introduce the notion of a primitive loop ${\cal K}$.

\begin{definition}\label{Def_Ktza}
For 
$$
\boldsymbol{\sigma}=(\sigma_1, \sigma_2, \cdots \sigma_n),\quad 
\boldsymbol{a}=(a_1, a_2, \cdots a_n)
,\quad \sigma_i \in \{+,-\}, \quad
a_i\in \mathbb Z_L^{2}, \quad 1\le i\le n
$$
and  $m(\sigma)$ defined in \eqref{def_mtzk}, we define  ${\cal K}_{t, \boldsymbol{\sigma}, \textbf{a}}$ to be the unique solution to the equation 
 \begin{align}\label{pro_dyncalK}
       \frac{d}{dt}\,{\cal K}_{t, \boldsymbol{\sigma}, \textbf{a}} 
       =  
       W^{2}\cdot  \sum_{1 \le k < l \le n} \sum_{a, b} \; \left( {\cal G}^{(a), L}_{k, l} \circ {\cal K}_{t, \boldsymbol{\sigma}, \textbf{a}} \right) S^{(B)}_{ab} \left( {\cal G}^{(b), R}_{k, l} \circ {\cal K}_{t, \boldsymbol{\sigma}, \textbf{a}} \right)  
    \end{align}
with initial data
$$
    {\cal K}_{0, \boldsymbol{\sigma}, \textbf{a}} =  {W^{-2(n+1)}} \cdot \prod_{k=1}^n m (\sigma_k) \cdot \textbf{1}(a_1 = a_2 = \cdots = a_n).
$$
Here, we define the ${\cal G}$ operator acting on ${\cal K}$ in the same way as it acts on ${\cal L}$,  i.e., ,  
  \be\label{calGonIND}
     {\cal G}^{(a),\,L }_{k,l}  \circ {\cal K}_{t, \boldsymbol{\sigma}, \textbf{a}} = {\cal K}_{t, \; {\cal G}^{(a),\,L }_{k,l}  (\boldsymbol{\sigma}, \textbf{a})}  \quad \forall t, k, \boldsymbol{\sigma}, \textbf{a},   
       \ee
and similarly for  ${\cal G}^{(b),\, R }_{k,l}$. 
For the special case $n=1$, we define 
$
{\cal K}_{t,\,+,\, a}=m$ and  ${\cal K}_{t,\,-,\, a}=\overline m$ for $0\le t\le 1$. 
\end{definition}

Note that the lengths of ${\cal G}^{(a), L}_{k, l} \circ {\cal K}_{t, \boldsymbol{\sigma}, \textbf{a}}$ and ${\cal G}^{(a), R}_{k, l} \circ {\cal K}_{t, \boldsymbol{\sigma}, \textbf{a}}$ are no greater than the  length of ${\cal K}_{t, \boldsymbol{\sigma}, \textbf{a}}$. Thus, 
the system of equations for $\cal K$ can be solved inductively in the length of ${\cal K}$; we illustrate this in the next section.

\subsection{The propagator \texorpdfstring{$\Theta_\xi^{(B)}$}{Theta} and an estimate on \texorpdfstring{${\cal K}$}{ K}}
 
The key object that we need to analyze ${\cal K}$ is the following propagator.
 
 \begin{definition}\label{def_Theta}
We define the propagator $\Theta^{(B)}_{\xi}$ to be the following $L^{2}\times L^{2}$ matrix:
\begin{equation}\label{def_Thxi}
    \Theta^{(B)}_\xi := \frac{1}{1 - \xi \cdot S^{(B)}}, \quad \xi \in \mathbb{C}, \quad |\xi| < 1.
\end{equation}
(The superscript $B$ indicates a block-level matrix, i.e. one whose indices are parameterized by $\Z_{L}^{2}$.)
\end{definition}

For convenience, we often omit the superscript $(B)$ . (We remind the readers that $\Theta$  was often  used to denote the \(N \times N\) matrix  \((1 - \xi S)^{-1}\) in the literature.   In this paper,  $\Theta_\xi= \Theta_\xi^{(B)}$.) We note that
\begin{equation}\label{deri_Thxi}
    \partial_\xi \Theta^{(B)}_\xi = \Theta^{(B)}_\xi \cdot S^{(B)} \cdot \Theta^{(B)}_\xi.
\end{equation}

The following  three special  cases are often used in this paper (note that $|m|=|m^{(E)}|=1$ in our setting):
\[
\Theta^{(B)}_{t m^2}, \quad \Theta^{(B)}_{t \bar{m}^2}, \quad \text{and} \quad \Theta^{(B)}_{t |m|^2} = \Theta^{(B)}_{t}.
\]

We now collect the following important properties of $\Theta^{(B)}$. 
\begin{lemma}\label{lem_propTH}
Suppose that  $\xi\in \mathbb C$, $\im \,\xi>0$ and $|\xi|< 1$.  Define
$$
 \hat{\ell}(\xi) := \min\left((|1-\xi|)^{-1/2},\; L\right),\quad \xi\in \mathbb C 
$$
Then  $\Theta^{(B)}_\xi$ has the following properties:
\begin{enumerate}
    \item Symmetry: $ (\Theta^{(B)}_{\xi})_{ab}= (\Theta^{(B)}_{\xi})_{ba}$.
    \item Translation invariance: $
    (\Theta^{(B)}_{\xi})_{ab}= (\Theta^{(B)}_{\xi})_{a+s, b+s}$ for any $s \in \mathbb{Z}_L^2$.
    \item Commutativity: $\left[S^{(B)}, \Theta^{(B)}_{\xi}\right]=\left[\Theta^{(B)}_{\xi\,'}, \Theta^{(B)}_{\xi}\right]=0$ for all $\xi,\xi'$.
    \item For any $ 0 \le \xi < 1$, we have the random walk representation $\Theta_\xi^{(B)}=\sum_{k=0}^{\infty}\;  \xi^k \cdot (S^{(B)})^k$.
     \item Exponential decay at length scale $\hat \ell_\xi $: 
     \begin{equation}\label{prop:ThfadC}     (\Theta^{(B)}_{\xi})_{ab}\prec \frac{e^{-c\cdot|a-b|_{L}\big/ \hat{\ell}(\xi)}}{|1-\xi|\cdot [\hat \ell (\xi)]^{2}}.  
     \end{equation}     
     \item Derivative bounds: for any $s \in \mathbb{Z}_L^2$, we have 
     \begin{equation}\label{prop:BD1}     
    \left| (\Theta^{(B)}_{\xi})_{ab}-(\Theta^{(B)}_{\xi})_{a(b+s)}\right|\prec \frac{|s|_L}{{|a-b|_{L}+1}}{+\frac{|s|_{L}}{[\hat{\ell}(\xi)]^{2}|1-\xi|^{1/2}}}   
     \end{equation}
     \begin{equation}\label{prop:BD2}     
     \left|2 (\Theta^{(B)}_{\xi})_{ab}-(\Theta^{(B)}_{\xi})_{a(b+s)}-(\Theta^{(B)}_{\xi})_{a(b-s)}\right|\prec  \frac{|s|_L^2}{|a-b|_L^{2}+1 }+{\frac{|s|_{L}^{2}}{[\hat{\ell}(\xi)]^{2}}} . 
     \end{equation}

\end{enumerate}
\end{lemma}
{The proof of Lemma \ref{lem_propTH} is based on the random walk representation and pointwise estimates for $(S^{(B)})^{k}$ that agree with the continuum Gaussian heat kernel estimates. A proof is given in Section \ref{sec_TH}; the argument is long and elementary, so we give an intuitive explanation of Lemma \ref{lem_propTH} below.

The first four properties in Lemma \ref{lem_propTH} are standard. For the exponential decay \eqref{prop:ThfadC} for the discrete Laplacian on $\Z^{2}$ (instead of $S^{(B)}-\mathrm{Id}$) can be found in Theorem 1.3 of \cite{MS22} (see also the discussion following it therein). The difference between the full space $\Z^{2}$ and the torus $\Z_{L}^{2}$ is that if $\hat{\ell}(\xi)=L$, so that $|1-\xi|^{-1/2}\geq L$, then the periodic boundary conditions on $\Z_{L}^{2}$ automatically introduces a space-cutoff of scale $L$ (before the cutoff at scale $|1-\xi|^{-1/2}$ that is imposed by the resolvent). One can also check that \eqref{prop:ThfadC} implies 
\begin{align*}
\sum_{b\in\Z_{L}^{2}}(\Theta^{(B)}_{\xi})_{ab}\prec|1-\xi|^{-1},
\end{align*}
which can be checked using the random walk representation as well. 

The gradient bounds \eqref{prop:BD1}-\eqref{prop:BD2} are standard, in that the Green's function for the Laplacian on $\R^{2}$ is $\log|x|$. Thus, the first terms on the right-hand side of \eqref{prop:BD1}-\eqref{prop:BD2} agree with its derivatives. (The support-length of the gradients is also of order $\hat{\ell}(\xi)$, but we will not need this fact.) The additional corrections come from the periodic boundary of $\Z_{L}^{2}$. The scaling of the last term in \eqref{prop:BD2} agrees with the heuristic that $\Theta^{(B)}$ is the resolvent for the diffusion of a generator, so its second-derivative should have $L^{1}$ norm of size $\mathrm{O}_{\prec}(1)$. The last term in \eqref{prop:BD1} therefore interpolates between \eqref{prop:ThfadC} and the last term in \eqref{prop:BD2}.}

Let us now illustrate ${\cal K}_{t,\boldsymbol{\sigma},(a_{1},a_{2})}$ for $n=2$. In this case, the defining primitive equation  \eqref{pro_dyncalK} is 
\begin{align}
    \frac{d}{dt}\, \mathcal{K}_{t, \boldsymbol{\sigma}, (a_1, a_2)} = W^{2} \cdot \sum_{a, b} \mathcal{K}_{t, \boldsymbol{\sigma}, (a_1, a)} \cdot S^{(B)}_{ab} \cdot \mathcal{K}_{t, \boldsymbol{\sigma}, (b, a_2)}
\end{align}
Throughout this paper, we will often use the following convenient notation: 
\begin{equation}\label{eq_defm_i}
    m_i = m(\sigma_i).
\end{equation}
Using the propagator  $\Theta^{(B)}$ in \eqref{def_Thxi} and, \eqref{deri_Thxi}, a short calculation yields
\begin{equation}\label{Kn2sol}
    \mathcal{K}_{t, \boldsymbol{\sigma}, \textbf{a}} = W^{-2} m_1 m_2 \left(\Theta^{(B)}_{t \cdot m_1 m_2}\right)_{a_1 a_2}, \quad \textbf{a} = (a_1, a_2)
\end{equation}
In particular, if we consider the two cases $\boldsymbol{\sigma}=(+,-)$ and $\boldsymbol{\sigma}=(+,+)$, then
\begin{equation}\label{Kn2sol2}
{\cal K}_{t,\boldsymbol{\sigma},(a,b)}
=   W^{-2} \cdot \begin{cases}
     |m |^2\left[\left(1- {  t }|m |^2 S^{(B)}\right)^{-1}\right]_{ab} &, \quad \boldsymbol{\sigma}=(+,-)\\
  m^2\left[\left(1- {  t } m^2 S^{(B)}\right)^{-1}\right]_{ab}  &, \quad \boldsymbol{\sigma}=(+,+)    
\end{cases}
\end{equation}

In the case where $n=3$, we have the following formula (which comes from Example 2.16 in \cite{YY_25} and comes directly from the primitive equation, hence it has no dimension dependence):
$$
{\cal K}_{t, \boldsymbol{\sigma}, \textbf{a}} = \sum_{b_1 b_2 b_3} \left(\Theta^{(B)}_{t \cdot m_1 m_2}\right)_{a_1 b_1} \left(\Theta^{(B)}_{t \cdot m_2 m_3}\right)_{a_2 b_2} \left(\Theta^{(B)}_{t \cdot m_3 m_1}\right)_{a_3 b_3} \cdot {\cal K}_{0, \boldsymbol{\sigma}, \textbf{b}}, \quad \textbf{b} = (b_1, b_2, b_3).
$$

\begin{lemma}\label{ML:Kbound}
Recall $\hat \ell $ in \eqref{prop:ThfadC}. We have
\begin{equation}\label{eq:bcal_k}
  {\cal K}_{t, \boldsymbol{\sigma}, \textbf{a}}  \prec M_{t}^{-n+1},\quad  \ell_t :  =\hat\ell(t)= \min\left((|1-t|)^{-1/2},\; L\right), \quad M_{t}:=W^{2}\ell_{t}^{2}\eta_{t}.     
\end{equation}
\end{lemma}

Finally, we note that for any fixed \(|E| < 2 - \kappa\), we have $\im z_t = \im z_t^{(E)} = \im m^{(E)} \cdot (1 - t) \sim (1 - t)$. Thus, \(\ell(z)\) from \eqref{def_ellz} and \(\ell_t\) from \eqref{eq:bcal_k} are the same order, in that $\ell(z_t) \sim \ell_t = \hat{\ell}(t)$. Futhermore, for \(z\), \(t\), and \(E\) satisfying \eqref{eq:zztE} and \eqref{eq:zztE2}, we have $\im z\sim z^{(E)}_t$ and $\ell(z) \sim \ell(z^{(E)}_t) \sim \ell_t$. Since these terms are the same order, in the following, we will use only \(\ell_t\).

\subsection{Estimates on loops}

Our main estimates for the  $G$-loop  are given below. (Recall $\prec$ in Definition \ref{stoch_domination}, and recall $\ell_{t},M_{t}$ in \eqref{eq:bcal_k}.) 

\begin{lemma}\label{ML:GLoop}
For any $\kappa, \tau>0$ and  
$E \in[-2+\kappa, 2-\kappa], \, 0\le t\le 1-N^{-1+\tau}$,  
we have 
\begin{equation}\label{Eq:L-KGt}
 \max_{\boldsymbol{\sigma}, \textbf{a}}\left|{\cal L}_{t, \boldsymbol{\sigma}, \textbf{a}}-{\cal K}_{t, \boldsymbol{\sigma}, \textbf{a}}\right|\prec M_{t}^{-n},\quad \eta_t:=\im z_t,  
\end{equation}
 \begin{equation}\label{Eq:L-KGt2}
\max_{\boldsymbol{\sigma}, \textbf{a}}\left|{\cal L}_{t, \boldsymbol{\sigma}, \textbf{a}} \right|\prec M_{t}^{-n+1}. 
\end{equation}
\end{lemma}
\begin{lemma}[$2$-loop estimate]\label{ML:GLoop_expec}
With the notations and assumptions of the previous lemma, if $\boldsymbol{\sigma}=(+,-)$ and $ \textbf{a}=(a_1, a_2)$, then
we have
\begin{equation}\label{Eq:Gdecay}
 \left| {\cal L}_{t, \boldsymbol{\sigma}, \textbf{a}}-{\cal K}_{t, \boldsymbol{\sigma}, \textbf{a}}\right|\prec M_{t}^{-2}\exp \left(-\frac{|a_1-a_2|_{L}^{1/2} }{\ell_t^{1/2}}\right)+W^{-D}. 
\end{equation}
\end{lemma}

\begin{lemma}[Local law for $G_t$]\label{ML:GtLocal}
   With the notations and assumptions of the previous lemma, we have 
\begin{equation}\label{Gt_bound}
 \|G^{(E)}_{t,+}-m^{(E)}\|_{\max}\prec M_{t}^{-1/2}.
   \end{equation} 
\end{lemma}

\begin{lemma}[Improved expectation bound on $2$-loop]\label{ML:exp}
Adopt the notations and assumptions of the previous lemma. Fix any $\boldsymbol{\sigma}\in\{+,-\}^{2}$ and $\textbf{a}=(a_{1},a_{2})\in\Z_{L}^{2}\times\Z_{L}^{2}$. We have 
\begin{align}
|\E{\cal L}_{t,\boldsymbol{\sigma},\textbf{a}}-{\cal K}_{t,\boldsymbol{\sigma},\textbf{a}}|\prec M_{t}^{-3}.\label{eq:step6main}
\end{align}
\end{lemma}

\medskip 

\begin{proof}[Proof of Theorems \ref{MR:locSC} and \ref{MR:QDiff}]
This is the same argument as the proof of Theorems 2.3 and 2.5 in \cite{YY_25}. For each $z$ in Theorems \ref{MR:locSC} and \ref{MR:QDiff}, Lemma \ref{zztE} gives $E$ and $t$ satisfying \eqref{eq:zztE}, \eqref{eq:zztE2}, and
$$
(1-t) \sim \im z_t \sim \im z \geq N^{-1+\tau}, \quad \eta_t \sim \eta, \quad \ell_t \sim \ell(z).
$$
Combining \eqref{mtEmz} and \eqref{GtEGz} gives
\begin{equation}\label{Gmt1/2Gm}
G(z) - m_{sc}(z) = t^{1/2} \left(G_t^{(E)} - m^{(E)}\right).
\end{equation}
Thus, \eqref{G_bound} follows from the estimate \eqref{Gt_bound} on $G_t - m$. Similarly, the partial tracial local law \eqref{G_bound_ave} follows from \eqref{Eq:L-KGt} for $n=1$ and the definition ${\cal K}_{t,+,a}=m^{(E)}$ from Definition \ref{Def_Ktza}. Next, by Definition \ref{Def:G_loop},
\begin{equation}\label{Gmt1/2Gm2}
\tr GE_aGE_b = t \cdot {\cal L}_{t, (+,+),(a,b)}, \quad \tr GE_aG^\dagger E_b = t \cdot {\cal L}_{t, (+,-),(a,b)}.
\end{equation}
By using the formula \eqref{Kn2sol} and recalling from \eqref{mtEmz} that $m_{sc}(z)=t^{1/2}m^{(E)}$, we get
$$
t \cdot {\cal K}_{t, (+,+),(a,b)} = W^{-2} m^2_{sc}(z) \cdot \frac{1}{1 - m^2_{sc}(z) \cdot S^{(B)}}, \quad t \cdot {\cal K}_{t, (+,-),(a,b)} = W^{-2} |m_{sc}(z)|^2 \cdot \frac{1}{1 - |m_{sc}(z)|^2 \cdot S^{(B)}}.
$$
Therefore, \eqref{Meq:QdW1} and \eqref{Meq:QdW2} in Theorem \ref{MR:QDiff} follow from \eqref{Eq:L-KGt} in the case $n=2$. Finally, \eqref{Meq:QdS1}-\eqref{Meq:QdS2} follow from Lemma \ref{ML:exp}.
 \end{proof}

\bigskip

\subsection{Strategy of the proofs of main lemmas}

It turns out that Lemma \ref{ML:exp} will follow from Lemmas \ref{ML:GLoop}, \ref{ML:GLoop_expec}, and \ref{ML:GtLocal}. Thus, let us outline the proofs of Lemmas \ref{ML:GLoop}, \ref{ML:GLoop_expec} and \ref{ML:GtLocal}. The general strategy is the same as in Section 2 of \cite{YY_25}.
By Definitions \ref{Def:G_loop} and \ref{Def_Ktza}, we have
$$
G_{0}(+)=m\cdot I_{N\times N}, \quad {\cal L}_{0, \boldsymbol{\sigma},\textbf{a}}= {\cal K}_{0, \boldsymbol{\sigma},\textbf{a}},\quad \forall \; \boldsymbol{\sigma},\textbf{a}, 
$$
where $I_{N\times N}$ is the identity matrix. By construction, we know that:
\begin{align}\label{ini)bigasya}
    \hbox{    Lemmas \ref{ML:GLoop}, \ref{ML:GLoop_expec} and \ref{ML:GtLocal} hold at  $t=0$ with no error.}
\end{align}
The following result propagates the desired estimates forwards in time.

\begin{theorem}\label{lem:main_ind} Assume for some fixed $E$: $|E|\le 2-\kappa$ and $s\in [0,1]$ that Lemmas \ref{ML:GLoop}, \ref{ML:GLoop_expec}  and \ref{ML:GtLocal} hold at time $s$. That is, assume that 
for $1 \le n\in \mathbb N$ and large $D>0$, we have
\begin{align}\label{Eq:L-KGt+IND}
 \max_{\boldsymbol{\sigma}, \textbf{a}}\left|{\cal L}_{s, \boldsymbol{\sigma}, \textbf{a}}-{\cal K}_{s, \boldsymbol{\sigma}, \textbf{a}}\right|&\prec M_{s}^{-n}  
\\\label{Eq:Gdecay+IND}
 \left| {\cal L}_{s, \boldsymbol{\sigma}, \textbf{a}}-{\cal K}_{s, \boldsymbol{\sigma}, \textbf{a}}\right|&\prec M_{s}^{-2}\exp \left(- \frac{ |a_1-a_2|_{L}^{1/2} }{\ell_s^{1/2}}\right)+W^{-D} , \quad \boldsymbol{\sigma}=(+,-)
\\\label{Gt_bound+IND}
 \|G^{(E)}_{s ,+}-m^{(E)}\|_{\max} & \prec M_{s}^{-1/2}.
\end{align}
 Then for any $t>s$ satisfying 
\begin{equation}\label{con_st_ind}
    M_{s}^{-1}\le \left(\frac{1-t}{1-s}\right)^{30}
\end{equation}
we have 
\eqref{Eq:L-KGt+IND}, \eqref{Eq:Gdecay+IND} and \eqref{Gt_bound+IND} with $s$  replaced by $t$.   
\end{theorem}

\begin{proof}[Proof of Lemmas \ref{ML:GLoop}, \ref{ML:GLoop_expec},  and \ref{ML:GtLocal}]
For any $\tau$ and  $t \le 1-N^{-1+\tau}$,  choose $\tau'>0$ and $n_0 \in \mathbb N$ such that $M_{t}^{-1}\le  W^{- 30 \tau'}$ and $(1-t) = W^{-n_0 \tau'}$. We also choose a sequence of time-scales given by $1-s_{k} = W^{-k \tau'}$ for $ k \le n_0$, where $n_{0}$ is chosen so that $ s_{n_0}=t$. Since  $M_{s}$ is decreasing in  $s\in [0,1]$, we have 
   $$
M_{s_{k+1}}^{-1}\le M_{t}^{-1}\le W^{- 30\tau'}
\le \left(\frac{1-s_{k}}{1-s_{k+1}}\right)^{30}
$$
for all $k$ such that  $k+1\le n_0$. We can now apply Theorem \ref{lem:main_ind} for $s=s_{k}$ for each $k$ inductively. (We note that $n_{0}=O(1)$, so the number of implicit $W^{\tau}$ factors we pick up from the meaning of $\prec$ is still $W^{\tau'}$ for $\tau'>0$ small.) This gives all the estimates in Lemmas \ref{ML:GLoop}, \ref{ML:GLoop_expec},  and \ref{ML:GtLocal} except for \eqref{Eq:L-KGt2}, which follows from \eqref{Eq:L-KGt} combined with the deterministic ${\cal K}$ estimate from Lemma \ref{ML:Kbound}.
\end{proof}

Theorem \ref{lem:main_ind} and Lemma \ref{ML:exp} will be proved in six steps. Their proofs are given in Section \ref{Sec:Stoflo}. The first five of these steps are meant to show Theorem \ref{lem:main_ind}. So, throughout the first five of these steps, we assume that the conditions of Theorem \ref{lem:main_ind} are satisfied. Each step builds on the conclusions of the preceding steps.

On the other hand, the last of these steps is separate in the sense that it is not part of the propagation idea used in the proofs of Lemmas \ref{ML:GLoop}, \ref{ML:GLoop_expec},  and \ref{ML:GtLocal} above.

\medskip 
\noindent 
\textbf{Step 1}   (A priori loop bounds): We will first show the following loop estimate:
 \begin{equation}\label{lRB1}
   {\cal L}_{u,\boldsymbol{\sigma}, \textbf{a}}\prec (\ell_u/\ell_s)^{2(n-1)}\cdot 
   M_{u}^{-n+1},\quad  s\le u\le t.
\end{equation}
We will also show the following weaker version of \eqref{Gt_bound+IND}:
\begin{equation}\label{Gtmwc}
    \|G_u-m\|_{\max}\prec  M_{u}^{-1/4},\quad s\le u\le t .
\end{equation}

 \bigskip
\noindent 
\textbf{Step 2}  (A priori $2$-loop decay): 
 Next, we show the following local law for $u\in [s,t]$: 
 \begin{equation}\label{Gt_bound_flow}
     \|G_u-m\|_{\max}\prec  M_{u}^{-1/2},\quad s\le u\le t.
\end{equation} 
Hence \eqref{Gt_bound+IND} in Theorem \ref{lem:main_ind} will follow.  
In addition,  for any  
 $\boldsymbol{\sigma}=(+,-)$,   $s\le u\le t$, and $D>0$, we show   
 \begin{equation}\label{Eq:Gdecay_w}
\left| {\cal L}_{u, \boldsymbol{\sigma}, \textbf{a}}-{\cal K}_{u, \boldsymbol{\sigma}, \textbf{a}}\right| \prec \left(\eta_s/\eta_u\right)^4\cdot M_{u}^{-2}\exp \left(- \frac{|a_1-a_2|_{L}^{1/2}}{\ell_u^{1/2}}\right)+W^{-D} . \quad 
\end{equation}
   \bigskip
 \noindent 
\textbf{Step 3}   (Sharp  loop  bounds): 
Next, we show the following \emph{sharp} estimate on loops:
\begin{equation}\label{Eq:LGxb}
\max_{\boldsymbol{\sigma}, \textbf{a}}\left| {\cal L}_{u, \boldsymbol{\sigma}, \textbf{a}} \right|
\prec 
M_{u}^{-n+1} ,\quad \quad s\le u\le t, \quad\forall\, n\in \mathbb N.
\end{equation}

 \bigskip
 \noindent 
\textbf{Step 4}  (A sharp  $\cal L - \cal K$  bound): 
We show a sharp bound on ${\cal L}_{t}-{\cal K}_{t}$, which gives \eqref{Eq:L-KGt+IND} in Theorem  \ref{lem:main_ind}:
\begin{equation}\label{Eq:L-KGt-flow}
 \max_{\boldsymbol{\sigma}, \textbf{a}}\left|{\cal L}_{u, \boldsymbol{\sigma}, \textbf{a}}-{\cal K}_{u, \boldsymbol{\sigma}, \textbf{a}}\right|\prec M_{u}^{-n},\quad \quad s\le u\le t, \quad \forall\, n\in \mathbb N.
\end{equation}

 \bigskip
 \noindent 
\textbf{Step 5}  (A sharp decay bound): For  $ \boldsymbol{\sigma}=(+,-)$, we show the following, which gives \eqref{Eq:Gdecay+IND} in Theorem  \ref{lem:main_ind}:
 \begin{equation}\label{Eq:Gdecay_flow}
\left| {\cal L}_{u, \boldsymbol{\sigma}, \textbf{a}}-{\cal K}_{u, \boldsymbol{\sigma}, \textbf{a}}\right| \prec   M_{u}^{-2}\exp \left(- \frac{|a_1-a_2|_{L}^{1/2}}{\ell_u^{1/2}}\right)+W^{-D}. 
\end{equation}

 \bigskip
 \noindent 
\textbf{Step 6}  (Showing Lemma \ref{ML:exp}): As we showed earlier, Theorem \ref{lem:main_ind} implies Lemmas \ref{ML:GLoop}, \ref{ML:GLoop_expec}, and \ref{ML:GtLocal}. We then use these results to prove Lemma \ref{ML:exp}.

\subsection{Notations}
Here we summarize  global notations used in this paper. 
\begin{itemize}
    \item $W$ is band width, and $L^{2}$ is the number of blocks, so that $N=W^{2}\cdot L^{2}$
    \item ${\cal I}^{(2)}_a$ is the $a$-th block, $[x]$ is the block containing $x$, and $E_a$ is the following matrix supported on ${\cal I}^{(2)}_a$: 
    $$
    x\in {\cal I}^{(2)}_{[x]}, \quad (E_a)_{xy}=\delta_{xy}\cdot W^{-2}\cdot {\bf 1}(x\in {\cal I}^{(2)}_a)$$
    \item The periodic distance on $\mathbb{Z}_L$ is denoted by $|\cdot|_L$ and the periodic $L^1$ distance on $\mathbb{Z}_L^2$ is defined for any $a = (a_1, a_2)\in\mathbb{Z}_L^2$ by
    $$
    |a| = |a|_L = |a_1|_L + |a_2|_L.
    $$
    
    \item The matrices $S\in \mathbb R^{N\times N}$, $S^{(B)}\in \mathbb R^{L\times L}$, $S_W\in \mathbb R^{W\times W}$ are all related to the variances of entries of the random matrix $H$, and $S= S^{(B)}\otimes S_{W}$.
\item In the stochastic flow, we typically omit the superscript $(E)$ for simplicity, so that
\[
z_t = z_t^{(E)}, \quad G_t := (\sqrt{t} H - z_t)^{-1}, \quad m = \lim_{\varepsilon \searrow 0} m_{sc}(E + i \varepsilon).
\]
Additionally, we use $G_t$ to mean $(H_t - z_t)^{-1}$, where $H_t$ has the same distribution as $\sqrt{t} H$. The context will make it clear which interpretation of $G_t$ is being used.
\item The length-scale $\ell_t$ is defined in \eqref{eq:bcal_k} as $\ell_t:=\hat \ell(t)=\min(|1-t|^{-1/2}, L)\sim\ell(z_{t})$.

    \item The notation $\cal L$ denotes $G$-loops, and $\cal K$ denotes its deterministic approximation (the ``primitive'),  and $\cal C$ denotes a ``$G$-chain" (see Definition \ref{def_Gchain}). 

    \item The matrix $\Theta_\xi = \Theta^{(B)}_{\xi}$ is given in Definition \ref{def_Theta}.

    \item The quantities $\Xi^{({\cal L})}$, $\Xi^{({\cal L-K})}$, $\Xi^{({\cal C})}$ are loops that are normalized to have heuristic size of order $1$:
    
    \begin{align}
\Xi^{({\cal L})}_{t, m}&:= \max_{\boldsymbol{\sigma}, \textbf{a}} 
\left|{\cal L}_{t, \boldsymbol{\sigma}, \textbf{a}} \right|\cdot M_{t}^{m-1} \cdot {\bf 1}(\sigma \in \{+,-\}^m),
\nonumber \\
\Xi^{({\cal L-K})}_{t, m}&:= \max_{\boldsymbol{\sigma}, \textbf{a}}
\left|{ ({\cal L-K})}_{t, \boldsymbol{\sigma}, \textbf{a}} \right|\cdot M_{t}^{m }\cdot {\bf 1}(\sigma \in \{+,-\}^m) ,
\nonumber \\
\Xi^{({{\cal C},\;diag })}_{t, m} &:= \max_{\boldsymbol{\sigma}, \textbf{a}} \max_{x} \left|\left({\cal C}^{(m)}_{t, \boldsymbol{\sigma}, \textbf{a}}\right)_{xx}\right| \cdot M_{t}^{m-1} \cdot {\bf 1}(\sigma \in \{+,-\}^m), \nonumber \\
\Xi^{({{\cal C},\;off})}_{t, m} &:= \max_{\boldsymbol{\sigma}, \textbf{a}} \max_{x\neq y}
\left|\left({\cal C}^{(m)}_{t, \boldsymbol{\sigma}, \textbf{a}}\right)_{xy}\right| \cdot M_{t}^{m-1/2} \cdot {\bf 1}(\sigma \in \{+,-\}^m).
\end{align}
\item The operators $\varTheta_{t,\boldsymbol{\sigma}}$ and ${\cal U}_{s, t,\boldsymbol{\sigma}}$ in Definition \ref{DefTHUST} appear in an ``integrated loop hierarchy" (Lemma \ref{Sol_CalL}). 
\item  The \(\cal G\) operators, such as \({\cal G}^{(a)}_k\), \({\cal G}^{(a),\,L}_{k,l}\), and \({\cal G}^{(a),\,R}_{k,l}\), are in Definition \ref{Def:oper_loop}. These represent the cutting and gluing of loops in the loop hierarchy.

\item  The \(\cal E\) terms represent lower-order terms in the (integrated) loop hierarchy \eqref{eq:mainStoflow} and \eqref{int_K-LcalE}. Specifically, \({\cal E}^{(M)}\) and \({\cal E}^{(\widetilde{G})}\) are defined in \eqref{def_Edif} and \eqref{def_EwtG}, respectively. The term \(\mathcal{E}^{((\mathcal{L}-\mathcal{K}) \times (\mathcal{L}-\mathcal{K}))}\) is defined in \eqref{def_ELKLK}. Additionally, \(\left( \mathcal{E} \otimes \mathcal{E} \right)\) and \(\left( \mathcal{E} \otimes \mathcal{E} \right)^{(k)}\) are defined in Definition \ref{def:CALE}.

\item  The quantity \({\cal T}^{(\cal L-\cal K)}_{t}\) is defined in \eqref{def_WTu}; it captures tail behavior of ${\cal L}-{\cal K}$. The quantity \({\cal T}_{t,D}\) represents a deterministic rough tail function, as defined in \eqref{def_WTuD}. 
\item The terms \({\cal J}_{u,D}\) and \(\cal J^*\) are introduced in \eqref{def_Ju} and \eqref{shoellJJ}, respectively, to describe the ratio between \({\cal T}^{(\cal L-\cal K)}_{t}\) and \({\cal T}_{t,D}\).

\item We often use an augmented length-scale $\ell_t^*=(\log W)^{3/2}\ell_t$. At this scale, $\Theta_t$ is exponentially small in $W$, but ${\cal T}_t$ is not.

\end{itemize}

\section{Basic properties of \texorpdfstring{$\cal K$}{K}}\label{Sec:CalK}

In this section, we record the exact properties about ${\cal K}$ that will be used in the analysis of the loop hierarchy. The results in this section are essentially taken from Section 3 of \cite{YY_25} with some slight modifications. 

We start with an estimate for a ``pure loop".
\begin{lemma}\label{lem_pureloop} Suppose that $\boldsymbol{\sigma}=(+,+\cdots +)$. We have 
\begin{align}\label{res_pureKes}\left|{\cal K}_{t, \boldsymbol{\sigma}, \textbf{a}}\right|\le C_n \exp\left(-c_n\max_{ij}|a_i-a_j|_{L}\right).
\end{align} 
\end{lemma}
\begin{proof}
We follow the proof of Corollary 3.5 in \cite{YY_25}. First, we use equations (3.5) and (3.6) in \cite{YY_25}; this gives 
\begin{align*}
{\cal K}_{t,\boldsymbol{\sigma}, \textbf{a}}=\prod_{i=1}^{n}m(\sigma_{i})\cdot W^{-2(n-1)}\sum_{\Gamma_{\textbf{a}}}\Gamma_{\textbf{a}}(t,\boldsymbol{\sigma}).
\end{align*}
The sum $\Gamma_{\textbf{a}}$ has a number of terms depending only on $n$, and each $\Gamma_{\textbf{a}}$ has the form
\begin{align*}
\Gamma_{\textbf{a}}(t,\boldsymbol{\sigma})&:=\sum_{\textbf{b}\in(\Z_{L}^{2})^{n}}\Gamma^{(\textbf{b)}}_{\textbf{a}}(t,\boldsymbol{\sigma}),\\
|\Gamma^{(\textbf{b})}_{\textbf{a}}(t,\boldsymbol{\sigma})|&\leq C_{n}\exp\left(-c_{n}\max_{i,j}|b_{i}-b_{j}|_{L}-c_{n}\max_{i,j}|a_{i}-b_{j}|_{L}\right).
\end{align*}
(Each $\Gamma^{(\textbf{b})}_{\textbf{a}}(t,\boldsymbol{\sigma})$ term is a product of $\Theta_{tm^{2}}$ matrices. The second line above, which is shown in the proof of Corollary 3.5 in \cite{YY_25}, relies only on the exponential decay of $\Theta_{tm^{2}}$, which follows from Lemma \ref{lem_propTH} and the bound $|1-tm^{2}|\geq c$ for some $c>0$.) It now suffices to combine the previous two displays.
\end{proof}
We now give a generalization of the Ward identity for both ${\cal L}$ and ${\cal K}$. This is Lemma 3.6 in \cite{YY_25}.

\begin{lemma}\label{lem_WI_K} Fix a loop ${\cal L}_{t, \boldsymbol{\sigma}, \textbf{a}}$ with 
$\sigma_1=+$ and $\sigma_n=-$. We have
\begin{equation}\label{WI_calL}
  \sum_{a_n}{\cal L}_{t, \boldsymbol{\sigma}, \textbf{a}}=
\frac{1}{2W^{2}\eta_t}\left({\cal L}_{t, \; \boldsymbol{\sigma}+,  \; \textbf{a}/a_n}- {\cal L}_{t, \; \boldsymbol{\sigma}-,  \; \textbf{a}/a_n}\right)  
\end{equation}
\begin{equation}\label{WI_calK}
  \sum_{a_n}{\cal K}_{t, \boldsymbol{\sigma}, \textbf{a}}=
\frac{1}{2W^{2}\eta_t}\left({\cal K}_{t, \; \boldsymbol{\sigma}+,  \; \textbf{a}/a_n}
- {\cal K}_{t, \; \boldsymbol{\sigma}-,  \; \textbf{a}/a_n}\right)  
\end{equation}
where 
\begin{itemize}
\item  $\boldsymbol{\sigma}\pm$ is the length $n-1$ sign vector obtained by removing $\sigma_n$ from $\boldsymbol{\sigma}$ and replacing $\sigma_1$ with $\pm$, i.e., 
    $$
\boldsymbol{\sigma}\pm=(\pm, \sigma_2, \sigma_3,\cdots \sigma_{ n-1}).$$
\item  $\textbf{a}/a_n$ is obtained by removing $a_n$ from $\textbf{a}$:
 $$
\textbf{a}/a_n=(a_1, a_2,a_3,\cdots, a_{n-1})
$$
\end{itemize}
\end{lemma}
\begin{proof}
As in the proof of Lemma 3.6 in \cite{YY_25}, \eqref{WI_calL} follows by the usual Ward identity for Green's functions, i.e. $GG^{\dagger}=\frac{1}{2\eta}(G-G^{\dagger})$. The proof of \eqref{WI_calK} in the case $n=2,3$ is an algebraic computation using the explicit formulas for ${\cal K}$. For $n\geq4$, it uses only \eqref{pro_dyncalK}. Both are given in the proof of Lemma 3.6 in \cite{YY_25} with no dimension dependence (i.e. no dependence on estimates of $\Theta^{(B)}$), so we omit them here.
\end{proof}

The proof of Corollary 3.7 in \cite{YY_25} now gives the following as a consequence of the Ward identity \eqref{WI_calK}.

\begin{corollary}\label{lem_sumAinK}
    Under the assumptions of Lemma \ref{lem_WI_K}, we have 
\begin{align}\label{sumallAinK}
      \sum_{a_2}\cdots\sum_{a_{n-1}}\sum_{a_n}{\cal K}_{t, \boldsymbol{\sigma}, \textbf{a}}=O([W^{2}\eta_t]^{-n+1})
\end{align}
 \end{corollary}

We now conclude with an estimate on ${\cal K}$. To state it, we need to recall a decomposition of the primitive ${\cal K}$ in terms of simpler objects. First, fix a length $n$ for the ${\cal K}$ loop. We have the decomposition
\begin{equation}\label{KKpi}
  {\cal K}_{t, \boldsymbol{\sigma}, \textbf{a}} =W^{-2(n-1)}\cdot m_{\boldsymbol{\sigma}}\cdot \sum_{\pi}{\cal K}_{t, \boldsymbol{\sigma}, \textbf{a}} ^{(\pi)},\quad m_{\boldsymbol{\sigma}}=\prod_i m_i. 
\end{equation}
where the notation is explained as follows. First, $\pi$ a collection of sets of the form $\{i,j\}$ with $i,j\in\Z_{n}$. (The exact constraints on $\pi$ are given in Definition 3.9 in \cite{YY_25}. We do not record them here because these constraints will not be important for us to explicitly use; in \cite{YY_25}, they are used to prove certain estimates that depend only on Lemma \ref{lem_propTH} for $|1-\xi|,\hat{\ell}(\xi)>c$ for some $c>0$, which we can therefore borrow directly.) The summands are given as follows:
\begin{align}
{\cal K}^{(\pi)}(t,\boldsymbol{\sigma},\textbf{a}):=\sum_{\textbf{d}\in(\Z_{L}^{2})^{d}}\sum_{\textbf{c}\in(\Z_{L}^{2})^{2(M-2)}}\left(\prod_{i=1}^{n}\left(\Theta^{(B)}_{tm_{i}m_{i+1}}\right)_{a_{i}d_{i}}\right)\cdot\prod_{k=1}^{M-1}\left(\Theta^{(B)}_{t|m|^{2}}-1\right)_{c_{2k-1}c_{2k}}\cdot\prod_{k=1}^{M}\Sigma^{(\emptyset)}(t,\boldsymbol{\sigma}^{(k)},\textbf{d}^{(k)}).
\end{align}
Above, $M=O(1)$ is a number determined only by $n,\pi$. The $\boldsymbol{\sigma}^{(k)}$ and $\textbf{d}^{(k)}$ are determined by $\boldsymbol{\sigma},\textbf{d},\pi,k$. However, the only properties that we will need for $\Sigma^{(\emptyset)}(t,\boldsymbol{\sigma}^{(k)},\textbf{d}^{(k)})$ that we will need are the following facts, whose proofs use only $\Theta^{(B)}$ bounds of the form in Lemma \ref{lem_propTH} for $|1-\xi|,\hat{\ell}(\xi)>c$ for some $c>0$ (and can therefore be directly cited).
\begin{itemize}
\item Suppose $n\geq4$ is even, and define the alternating sign vector $\boldsymbol{\sigma}^{(alt)}=(\sigma^{(alt)}_{1},\ldots,\sigma^{(alt)}_{n})$ via
\begin{equation}\label{assum_SZ_SinMole}
1\le k\le n, \quad \quad {\sigma}^{(alt)}_k =\begin{cases}
    +, \quad k\in 2\mathbb Z-1
    \\\nonumber
    -, \quad k\in 2\mathbb Z 
\end{cases}\; \;,\quad t\in[0,1]  ,  
\end{equation}
We have the following \emph{sum zero property}, which follows from the proof of Lemma 3.10 in \cite{YY_25}:
\begin{equation}\label{SZjadljsk}
   \sum_{d_2,\,\cdots,\,d_n}\left(\Sigma^{(\emptyset)}(t, \boldsymbol{\sigma}^{(alt)}, \textbf{d})\right)=O(\eta_t) 
\end{equation}
\item We have the short-range exponential bound below (which comes from (3.43) in \cite{YY_25}):
\begin{align}\label{res_SIGempdd}
  \Sigma^{(\emptyset)}(t, \boldsymbol{\sigma}, \textbf{d})\le C_n\exp\left(-c_n \max_{ij}|d_i-d_j|_{L}\right)   .
\end{align}
(This comes from exponential decay of $\Theta^{(B)}$ in Lemma \ref{lem_propTH}). However, in principle, all we need is that the sum of $|\Sigma^{(\emptyset)}(t,\boldsymbol{\sigma},\textbf{d})|$ over $d_{2},\ldots,d_{n}$ is $O(1)$, which follows from $\sum_{b}(\Theta^{(B)}_{\xi})_{ab}=(1-\xi)^{-1}$.) 
\item Fix $d_{1}$, and write $s_{i}=d_{i}-d_{1}$ for $i=2,\ldots,n$, and $g(\textbf{s}):=\Sigma^{(\emptyset)}(t,\boldsymbol{\sigma},\textbf{d})$. We have $g(\textbf{s})=g(-\textbf{s})$; see the proof of Lemma 3.11 in \cite{YY_25}.
\end{itemize}

The following is analogous to Lemma 3.11 in \cite{YY_25}. We explain the differences in the proofs of the result below and of Lemma 3.11 in \cite{YY_25}, namely the dimension-dependent parts.

\begin{lemma}\label{ML:Kbound+pi}
Recall $M_{t}:=W^{2}\ell_{t}^{2}\eta_{t}$. We have the estimate
\begin{equation}\label{eq:bcal_k_2}
  {\cal K}_{t, \boldsymbol{\sigma}, \textbf{a}} \prec M_{t}^{-n+1}.    
\end{equation}
\end{lemma}

\begin{proof}
In the case $n=2$, this follows by \eqref{Kn2sol}, so we focus on $n\geq3$ for the rest of this argument. Moreover, we can assume that $\boldsymbol{\sigma}$ does not consist of just one sign, i.e. it is not pure, since the estimate at hand follows by Lemma \ref{lem_pureloop} in this case. By \eqref{KKpi}, it suffices to show
\begin{equation}\label{eq:bcal_k_pi}
  {\cal K}^{(\pi)}_{t, \boldsymbol{\sigma}, \textbf{a}} \prec (\ell_t^{2}\eta_t)^{-n+1} .    
\end{equation}
We first assume that $\pi=\emptyset$. Assume without loss of generality (after re-indexing) that $\sigma_{1}\neq\sigma_{2}$. We have 
\begin{align*}
{\cal K}^{(\emptyset)}_{t, \boldsymbol{\sigma}, \textbf{a}}
=\;&
 \sum_{\textbf{d}}\left(\Sigma^{(\emptyset)}(t, \boldsymbol{\sigma}, \textbf{d})\right)
\cdot 
\prod_{i=1}^n \left(\Theta^{(B)}_{tm_im_{i+1}}\right)_{a_i, d_i}
\\\nonumber
=\;&\sum_{d_1}\left(\Theta^{(B)}_{t|m|^2}\right)_{a_1, d_1}
\cdot \sum_{d_2,\,\cdots,\, d_n}\left(\Sigma^{(\emptyset)}(t, \boldsymbol{\sigma}, \textbf{d})\right)
\cdot 
\prod_{i=2}^n \left(\Theta^{(B)}_{tm_im_{i+1}}\right)_{a_i, d_i}  
\end{align*}
We claim that to show \eqref{eq:bcal_k_pi} for $\pi=\emptyset$, it suffices to show
\begin{align}\label{spwow3}
\sum_{d_2, \cdots, d_n}\left(\Sigma^{(\emptyset)}(t, \boldsymbol{\sigma}, \mathbf{d})\right) \cdot \prod_{i=2}^n\left(\Theta_{t m_i m_{i+1}}^{(B)}\right)_{a_i, d_i} \prec (\ell^2_t \eta_t)^{-n+2} \cdot\left(\min _{2 \leq k \leq n}\left|a_k-d_1\right|_{L}^2+1\right)^{-1}+(\ell_{t}^{2}\eta_{t})^{-n+1}\eta_{t},
\end{align}
Indeed, we can bound entries of $\Theta^{(B)}_{t|m|^{2}}$ by $\ell_{t}^{-2}\eta_{t}^{-1}$ using Lemma \ref{lem_propTH}, so that
\begin{align*}
&[\ell^2_t \eta_t]^{-n+2}\sum_{d_1}\left(\Theta^{(B)}_{t|m|^2}\right)_{a_1, d_1} \cdot\left(\min _{2 \leq k \leq n}\left|a_k-d_1\right|_{L}^2+1\right)^{-1}\\
&=[\ell^2_t \eta_t]^{-n+1}\sum_{d_{1}}\left(\min _{2 \leq k \leq n}\left|a_k-d_1\right|_{L}^2+1\right)^{-1}\prec [\ell^2_t \eta_t]^{-n+1}.
\end{align*}
On the other hand, the column and row sums of $\Theta^{(B)}_{t|m|^{2}}$ are $O(\eta_{t}^{-1})$ by Lemma \ref{lem_propTH}, so
\begin{align*}
[\ell_{t}^{2}\eta_{t}]^{-n+1}\eta_{t}^{-1}\sum_{d_{1}}\Big(\Theta^{(B)}_{t|m|^{2}}\Big)_{a_{1}d_{1}}\prec [\ell_{t}^{2}\eta_{t}]^{-n+1}.
\end{align*}
To prove \eqref{spwow3}, we follow the proof of Lemma 3.11 in \cite{YY_25} and first assume that $\sigma_{j}=\sigma_{j+1}$ for some $j>1$. (Here, $n+1$ means $1$.) By Lemma \ref{lem_propTH}, we have 
\begin{align*}
    (\Theta^{(B)}_{tm_{j}m_{j+1}})_{a_{j}d_{j}}&=(\Theta^{(B)}_{tm^{2}})_{a_{j}d_{j}}\prec1,\\
    \prod_{i\neq1,j}(\Theta^{(B)}_{tm_{i}m_{i+1}})_{a_{i}d_{i}}&\prec[\ell_{t}^{2}\eta_{t}]^{-n+2}.
\end{align*}
By \eqref{res_SIGempdd}, the display \eqref{spwow3} then follows (with actually an exponential decay instead of polynomial decay).

We are left with the case $\boldsymbol{\sigma}=\boldsymbol{\sigma}^{(alt)}$. Following the proof of Lemma 3.11 in \cite{YY_25}, we write
\begin{align}
\sum_{d_{2},\ldots,d_{n}}\Sigma^{(\emptyset)}(t,\boldsymbol{\sigma},\textbf{d})\cdot\prod_{i=2}^{n}\left(\Theta^{(B)}_{tm_{i}m_{i+1}}\right)_{a_{i}d_{i}}&=\sum_{s_{2},\ldots,s_{n}}g(\textbf{s})\prod_{i=2}^{n}f(a_{i},s_{i})\\
&=\sum_{s_{2},\ldots,s_{n}}g(\textbf{s})\prod_{i=2}^{n}\sum_{\xi_{i}=0,1,2}f_{\xi_{i}}(a_{i},s_{i}),
\end{align}
where for any $d_{1},\textbf{a}$, we introduced the notation
\begin{align*}
\textbf{s}&=(s_{2},\ldots,s_{n}):=(d_{2}-d_{1},\ldots,d_{n}-d_{1})\\
g(\textbf{s})&:=\Sigma^{(\emptyset)}(t,\boldsymbol{\sigma},\textbf{d}),\\
f(a_{i},s_{i})&:=\left(\Theta^{(B)}_{t|m|^{2}}\right)_{a_{i},(d_{1}+s_{i})},\\
f_{0}(a_{i},s_{i})&:=f(a_{i},0)\\
f_{1}(a_{i},s_{i})&:=\frac{f(a_{i},s_{i})-f(a_{i},-s_{i})}{2}\\
f_{2}(a_{i},s_{i})&:=\frac12 f(a_{i},s_{i})+\frac12 f(a_{i},-s_{i})-f(a_{i},0).
\end{align*}
By Lemma \ref{lem_propTH}, we have the regularity bounds
$$\begin{aligned} & f_0\left(a_i, s_i\right)\prec \ell^{-2}_t \,\eta_t^{-1} \\ & f_1\left(a_i, s_i\right) \prec \frac{1}{\left|a_i-d_1\right|_{L}+1}+\frac{1}{\ell_{t}^{2}\eta_{t}^{1/2}} \\ & f_2\left(a_i, s_i\right)\prec \frac{1}{\left|a_i-d_1\right|_{L}^{2}+1}+\frac{1}{\ell_{t}^{2}}.\end{aligned}$$

We now follow the same reasoning in the proof of Lemma 3.11 in \cite{YY_25} to deduce \eqref{spwow3}, so \eqref{eq:bcal_k_pi} follows for $\pi=\emptyset$. The proof for general $\pi$ follows by an inductive argument given at the end of the proof of Lemma 3.11 in \cite{YY_25}.
\end{proof}

\section{\texorpdfstring{$G$}{G}-chains and their estimates}
Our analysis of the loop hierarchy in Lemma \ref{lem:SE_basic} will require estimates on the following ``chains".

\begin{definition}[$G$-chains]\label{def_Gchain}
Fix signs and indices as follows: 
\[
\boldsymbol{\sigma} = (\sigma_1, \sigma_2, \ldots, \sigma_n), \quad \textbf{a} = (a_1, a_2, \ldots, a_{n-1}), \quad\text{where } \sigma_k \in \{+, -\}, \quad a_k \in \mathbb{Z}_L^{2}.
\]
We define the associated $G$-chain to be the following $N\times N$ matrix:
\begin{equation}\label{defC=GEG}
    {\cal C}_{t, \boldsymbol{\sigma}, \textbf{a}} = G_t(\sigma_1) E_{a_1} G_t(\sigma_2) E_{a_2} \cdots E_{a_{n-1}} G_t(\sigma_n).
\end{equation}
We say that $({\cal C}_{t,\boldsymbol{\sigma},\textbf{a}})_{xy}$ is an off-diagonal term if $x\neq y$; we say that it is an on-diagonal term if $x=y$.
\end{definition}

Since $G$-chain estimates will be important for $G$-loop estimates, we need a mechanism that controls $G$-chains in terms of $G$-loops in order to close our analysis. For this, we present two results, whose proofs are identical to the proofs of Lemmas 4.2 and 4.3 in \cite{YY_25}, since these arguments depend only on resolvent identities for $G_{t}=(H_{t}-z_{t})^{-1}$ (and thus do not use the structure of $S$ or $\Theta^{(B)}$ in any way).

We start with $1$-chains, i.e. entries of $G$.

\begin{lemma}[$1$-chain estimate]\label{lem_GbEXP}  
Recall   $z_t$ and  $G_t$  from  Definitions \ref{def_flow} and \ref{Def:G_loop}. Fix any $|E|<2-\kappa$ and $0\leq t<1$. For any $c>0$, consider the event
\begin{equation}\label{def_asGMc}
    \Omega(t, c) := \big\{\|G_t - m\|_{\max} \leq W^{-c} \big\}.
\end{equation}
Then we have the following estimate for $G_{t}$ entries in terms of loops of length $2$:
\begin{equation}\label{GijGEX}
    {\bf 1}_{\Omega(t, c)} \cdot \max_{x \in {\cal I}^{(2)}_a} \max_{y \in {\cal I}^{(2)}_b} |(G_t)_{xy}|^{\,2} \prec 
    \sum_{a':|a'-a|_{L}\leq1} \; \sum_{b':|b'-b|_{L}\leq1} {\cal L}_{t, (+,-), (a', b')}
    + W^{-2} \cdot {\bf 1}(|a - b|_{L} \leq 1),
\end{equation}
\begin{equation}\label{GiiGEX}
    {\bf 1}_{\Omega(t, c)} \cdot \max_{x} \left|(G_t)_{xx} - m\right|^{\,2} \prec \max_{a, b} {\cal L}_{t, (+,-), (a, b)}.
\end{equation}
Moreover, suppose that the following is true for some $c > 0$:
\begin{equation}\label{asGMc} 
    \|G_t - m\|_{\max} \prec W^{-c}.
\end{equation}
Then \eqref{GijGEX} and \eqref{GiiGEX} hold without the indicator ${\bf 1}_{\Omega(t, c)}$, and we also have
\begin{equation}\label{GavLGEX}
    \max_{a} \left| \left\langle \left(G_t - m\right) E_a \right\rangle \right| \prec \max_{a, b} {\cal L}_{t, (+,-), (a, b)}.
\end{equation} 
\end{lemma}

We will now go from $1$-chains to general chains. Recall the notation $M_{t}=W^{2}\ell_{t}^{2}\eta_{t}$, and define
\begin{align}\label{def:XiL}
    \Xi^{({\cal L})}_{t, m} &:= \max_{\boldsymbol{\sigma}, \textbf{a}} \left|{\cal L}_{t, \boldsymbol{\sigma}, \textbf{a}} \right|\cdot M_{t}^{m-1} \cdot {\bf 1}(\boldsymbol{\sigma} \in \{+, -\}^m),
\end{align}
and (where $diag$ represents the diagonal terms, i.e., $i = j$, and $off$ represents off-diagonal terms, i.e. $i\neq j$)
\begin{align}\label{XiCDOD}
    \Xi^{({{\cal C},\;diag})}_{t, m} &:= \max_{\boldsymbol{\sigma}, \textbf{a}} \max_x \left|\left({\cal C}^{(m)}_{t, \boldsymbol{\sigma}, \textbf{a}}\right)_{xx} \right|\cdot M_{t}^{m-1} \cdot {\bf 1}(\boldsymbol{\sigma} \in \{+, -\}^m), \\
    \Xi^{({{\cal C},\;off})}_{t, m} &:= \max_{\boldsymbol{\sigma}, \textbf{a}} \max_{x \neq y} \left|\left({\cal C}^{(m)}_{t, \boldsymbol{\sigma}, \textbf{a}}\right)_{xy} \right| \cdot M_{t}^{m-1/2} \cdot {\bf 1}(\boldsymbol{\sigma} \in \{+, -\}^m).
\end{align}

\begin{lemma}[$n$-chain estimate]\label{lem_GbEXP_n2}  
Suppose that the assumptions of Lemma \ref{lem_GbEXP} and \eqref{asGMc}  hold, and that for some $n \ge  1$, we have the following a priori estimate:
\[
\Xi^{({\cal L})}_{t, m} \prec f(m, \alpha, \beta) := 1 + \alpha^{m-1} \beta^{-1}, \quad 2\le m \leq 2n.
\]
Above, $\alpha$ and $\beta$ are deterministic parameters that depending only $N,E,t$ and satisfy
\[
\alpha \geq 1, \quad \beta \geq 1, \quad \alpha^2 \leq M_{t}^{1/4}.
\]
Then, we have the following estimate for normalized $G$-chains:
\begin{align}\label{res_XICoff}
    \Xi^{({{\cal C},\;off})}_{t, n}  & \prec \big( f(2n, \alpha, \beta) \big)^{1/2},\\
    \Xi^{({{\cal C},\;diag})}_{t,\; 2n} & \prec f(2n, \alpha, \beta). \label{res_XICdi}
\end{align}
\end{lemma}

\section{Analysis of loop hierarchy}\label{Sec:Stoflo}

In this section, we prove Theorem \ref{lem:main_ind}. The argument largely follows that in Section 5 of \cite{YY_25}, except for Section \ref{subsection-sum-zero}, which requires the main new inputs of this paper. However, even for the estimates before Section \ref{subsection-sum-zero}, there are a lot of technical power-countings that are dimension-dependent. (For example, the $E_{a}$ matrices in \eqref{Def_matE} have a factor of $W^{2}$ with exponent matching the dimension. Also, the ball of radius $\ell$ in $\Z_{L}^{2}$ has size $O(\ell^{2})$. These changes are reflected in the definition of $M_{t}$ from \eqref{eq:bcal_k}, and it is important to keep track of them carefully.) For this reason, unless the argument is identical to the corresponding argument in \cite{YY_25} at the level of power-counting, we will provide details for proofs.

\subsection{Proof of Theorem \ref{lem:main_ind}: Step 1}

\begin{proof}
We recall that the goal of this first step is to prove \eqref{lRB1} and \eqref{Gtmwc}. First, we use \eqref{con_st_ind} and the definitions \(\ell_t = \ell(z_t)\) and \(\eta_t = \im z_t\). This gives
$$
1 - u \gg N^{-1}, \quad \eta_u \gg N^{-1}, \quad u \ge t.
$$
Given that \(\partial_z (H - z)^{-1} = (H - z)^{-2}\) and that the entries of \(H_t\) are jointly Gaussian, we deduce that for any \(C > 0\), there exists \(C'\) such that the following is true outside of events whose probabilities are exponentially small in $N$:
\begin{equation}\label{Gopboundu}
    \max_{u \ge N^{-1}} \max_{|u - u'| \leq N^{-C'}} \| G_u - G_{u'} \|_{\max} \leq N^{-C}.
\end{equation}
Hence, by a standard net argument, it suffices to prove \eqref{lRB1} and \eqref{Gtmwc} for \(u = t\). This continuity argument will be used repeatedly in this paper. 

\paragraph{Case 1: \(s < t \leq 1/2\).} 
We have $\|G_{t}\|_{op}\leq1/\eta_{t}=O(1)$. We also have \(\|E_a\|_{op} \leq W^{-2}\) for all \(a \in \mathbb{Z}_L^{2}\). Thus,  
\begin{equation}\label{Lboundfor1/2}
    {\cal L}_{t, \boldsymbol{\sigma}, \textbf{a}} = \left\langle \prod_{i=1}^n G_t(\sigma_i) E_{a_i} \right\rangle = O(W^{-2(n-1)}), \quad t \leq 1/2.
\end{equation}
Since $M_{t}=W^{2}\ell_{t}^{2}\eta_{t}=O(W^{2})$ for $t\leq1/2$, the bound \eqref{Lboundfor1/2} implies \eqref{lRB1}. Next, we combine \eqref{Lboundfor1/2} with \eqref{GijGEX} and \eqref{GiiGEX}, and we note that $\ell_{t},\eta_{t}$ are order $1$ for $t\leq1/2$. This gives
$$
{\bf 1}\left(\|G_t - m\|_{\max} \leq W^{-1/10}\right) \cdot \|G_t - m\|_{\max} \prec W^{-1}, \quad t \leq 1/2.
$$
By \eqref{Gopboundu} and the previous display, we deduce that $\|G_{u}-m\|_{\max}\prec W^{-1}$ implies $\|G_{u'}-m\|_{\max}\prec W^{-1}$ for any $|u'-u|\leq N^{-C'}$ if $C'>0$ is large enough. Thus, as soon as we can find a single time $u_{0}\in[s,t]$, where $t\leq1/2$, such that $\|G_{u_{0}}-m\|_{\max}\prec W^{-1}$, we would deduce this bound for all $u\in[s,t]$. By \eqref{Gt_bound+IND}, we can choose $u_{0}=s$, and thus \eqref{Gtmwc} holds, i.e.
\begin{equation}\label{Gtmt12}
    \|G_u - m\|_{\max} \prec W^{-1}, \quad u \leq 1/2.
\end{equation}

\paragraph{Case 2: \(t > s \geq 1/2\).} 
By \eqref{eq:bcal_k} from Lemma \ref{ML:Kbound}, i.e. we have ${\cal K}_{s, \boldsymbol{\sigma}, \textbf{a}} \prec M_{s}^{-n+1}$, where we recall $M_{s}=W^{2}\ell_{s}^{2}\eta_{s}$. If we combine this with the assumption \eqref{Eq:L-KGt+IND}, then we get that for any fixed \(n\) and \(s\), 
\begin{equation}\label{Lboundfors}
    {\cal L}_{s, \boldsymbol{\sigma}, \textbf{a}}\prec M_{s}^{-n+1}.
\end{equation}
We now record the following estimate, whose proof follows from the same argument as in Section 6 of \cite{YY_25}. {An outline of the proof is given in Section \ref{sec:Inh_LE}.}

\begin{lemma}\label{lem_ConArg}
Suppose that \(c < t_1 \leq t_2 \leq 1\) for some constant \(c > 0\). Assume that for any fixed \(n\geq1\), the following bounds hold at time \(t_1\): 
\begin{equation}\label{55}
    \max_{\boldsymbol{\sigma}, \textbf{a}} {\cal L}_{t_1, \boldsymbol{\sigma}, \textbf{a}} \prec M_{t_{1}}^{-n+1},\quad M_{t}:=W^{2}\ell_{t}^{2}\eta_{t}.
\end{equation}
Next, we let $\Omega$ denote the following event:
\begin{equation}\label{usuayzoo}
    \Omega := \left\{\|G_{t_2}\|_{\max} \leq 2\right\}.
\end{equation}
Then, for any fixed \(n \geq 1\), we have
\begin{equation}\label{res_lo_bo_eta}
    {\bf 1}_\Omega \cdot \max_{\boldsymbol{\sigma}, \textbf{a}} {\cal L}_{t_2, \boldsymbol{\sigma}, \textbf{a}} \prec (W^{2}\ell_{t_{1}}^{2}\eta_{t_{2}})^{-n+1} = \left( \frac{\ell_{t_{2}}}{\ell_{t_{1}}} \right)^{2(n-1)} \cdot M_{t_{2}}^{-n+1}.
\end{equation}
\end{lemma}

By \eqref{Lboundfor1/2} and \eqref{Lboundfors}, the a priori bound \eqref{55} for \({\cal L}_s\) holds. Thus, we can use Lemma \ref{lem_ConArg} to obtain
\begin{equation}\label{jsajufua}
    {\bf 1}\left(\|G_{u}\|_{\max} \leq 2\right) \cdot \max_{\boldsymbol{\sigma}, \textbf{a}} {\cal L}_{u, \boldsymbol{\sigma}, \textbf{a}} \prec \left( \frac{\ell_u}{\ell_s} \right)^{2(n-1)} \cdot M_{u}^{-n+1}, \quad s \leq u \leq t.
\end{equation}
Now, we use \eqref{jsajufua} for $n=2$ with Lemma \ref{lem_GbEXP} gives us
$$
{\bf 1}\left(\|G_u - m\|_{\max} \leq M_{u}^{-1/6}\right) \cdot \|G_u - m\|_{\max} \prec \frac{\ell_u}{\ell_s} \cdot M_{u}^{-1/2}, \quad s \leq u \leq t.
$$
By \eqref{con_st_ind}, we know $\ell_{u}\ell_{s}^{-1}\leq M_{s}^{1/60}$. Combining this with the previous display, we deduce that for any $D>0$, we have the following for $N$ large enough:
\begin{equation}
    \mathbb{P}\left(M_{u}^{-1/4} \leq \|G_u - m\|_{\max} \leq M_{u}^{-1/6}\right) \leq N^{-D}, \quad s \leq u \leq t.
\end{equation}
Another standard continuity argument (e.g. \eqref{Gopboundu}) allows us to insert either a maximum or minimum over $u\in[s,t]$ in the event above. 

Now, by assumption, we have the initial data bound $\|G_s - m\|_{\max} \prec M_{s}^{-1/2}$. Together with the previous display and continuity of $G_{u}$ in $u$, we deduce $\|G_t - m\|_{\max} \prec M_{t}^{-1/4}$. Combining this bound with \eqref{jsajufua} gives \eqref{lRB1} and \eqref{Gtmwc} for $t>s\geq1/2$.

\paragraph{Case 3: \(t > 1/2 > s\).} From Case 1, we know that \eqref{Lboundfor1/2} and \eqref{Gtmt12} hold for \(t = 1/2\). The proof of Case 2 uses only these two conditions. Thus, the current case follows as a consequence of Case 2 with \(s = 1/2\).
\end{proof}

\subsection{Dynamics of  \texorpdfstring{$\cal L-\cal K$}{L-K}}

 Recall the loop hierarchy from Lemma \ref{lem:SE_basic}. We now introduce the notation
$$
\mathcal{L}_{t,\;  \mathcal{G}^{(a), \,L}_{k, \,l}\left(\boldsymbol{\sigma},\, \textbf{a}\right)} := \left( \mathcal{G}^{(a), L}_{k, l} \circ \mathcal{L}_{t, \boldsymbol{\sigma}, \textbf{a}} \right), \quad  
\mathcal{L}_{t,\;  \mathcal{G}^{(b), \,R}_{k, \,l}\left(\boldsymbol{\sigma},\, \textbf{a}\right)} := \left( \mathcal{G}^{(b), R}_{k, l} \circ \mathcal{L}_{t, \boldsymbol{\sigma}, \textbf{a}} \right),
$$
which allows us write the hierarchy in a more convenient form that we describe below. By (5.15) in \cite{YY_25}, we have the following equation, which uses notation to be clarified afterwards:
\begin{align}\label{eq_L-Keee}
    d(\mathcal{L} - \mathcal{K})_{t, \boldsymbol{\sigma}, \textbf{a}} = &
    \Big[\mathcal{K} \sim (\mathcal{L} - \mathcal{K})\Big]^{l_\mathcal{K} = 2}_{t, \boldsymbol{\sigma}, \textbf{a}}dt + \sum_{l_\mathcal{K} > 2} \Big[\mathcal{K} \sim (\mathcal{L} - \mathcal{K})\Big]^{l_\mathcal{K}}_{t, \boldsymbol{\sigma}, \textbf{a}}dt + \mathcal{E}^{((\mathcal{L} - \mathcal{K}) \times (\mathcal{L} - \mathcal{K}))}_{t, \boldsymbol{\sigma}, \textbf{a}} +
    \mathcal{E}^{(M)}_{t, \boldsymbol{\sigma}, \textbf{a}} +
    \mathcal{E}^{(\widetilde{G})}_{t, \boldsymbol{\sigma}, \textbf{a}}.
\end{align} 
Above, the last two terms are introduced in Lemma \ref{lem:SE_basic}. The first two term are defined by the formula
\begin{align}\label{DefKsimLK}
&\; W^{2}\cdot \sum_{1 \leq k < l \leq n} \sum_{a, b}  
   \Big( 
   (\mathcal{L} - \mathcal{K})_{t,\;  \mathcal{G}^{(a), \,L}_{k, \,l}\left(\boldsymbol{\sigma},\, \textbf{a}\right)} 
   \cdot S^{(B)}_{ab}
   \cdot 
   \mathcal{K}_{t,\;  \mathcal{G}^{(b), \,R}_{k, \,l}\left(\boldsymbol{\sigma},\, \textbf{a}\right)} \cdot \textbf{1}(\text{length of $\mathcal{K}$ equals } l_\mathcal{K})   \nonumber \\
   & + \left(\mathcal{K} \iff \mathcal{L-K} \right) \Big) 
    := \sum_{l_{\cal K}=2}^n\Big[\mathcal{K} \sim (\mathcal{L} - \mathcal{K})\Big]^{l_\mathcal{K}}_{t, \boldsymbol{\sigma}; \textbf{a}},
\end{align}
here, \( \mathcal{K} \iff \mathcal{L-K} \) denotes the terms obtained by swapping \(\mathcal{K}\) and \(\mathcal{L-K}\). Next, we introduce
\begin{equation}\label{def_ELKLK}
    \mathcal{E}^{((\mathcal{L} - \mathcal{K}) \times (\mathcal{L} - \mathcal{K}))}_{t, \boldsymbol{\sigma}, \textbf{a}} :=
    W^{2}\cdot\sum_{1 \leq k < l \leq n} \sum_{a, b} (\mathcal{L} - \mathcal{K})_{t,\;  \mathcal{G}^{(a), \,L}_{k, \,l}\left(\boldsymbol{\sigma},\, \textbf{a}\right)} \cdot  S^{(B)}_{ab} \cdot (\mathcal{L} - \mathcal{K})_{t,\;  \mathcal{G}^{(b), \,R}_{k, \,l}\left(\boldsymbol{\sigma},\, \textbf{a}\right)} \, dt.
\end{equation}

We will turn \eqref{eq_L-Keee} into an integral equation. First, we introduce the main objects in this integral equation construction.

\begin{definition}\label{DefTHUST}
Define the linear map ${\varTheta}_{t, \boldsymbol{\sigma}}$ via the following action on a tensor ${\cal A}: \mathbb Z^n_L\to \mathbb C$: 
\begin{align}
    \left({{\varTheta}}_{t, \boldsymbol{\sigma}} \circ \mathcal{A}\right)_{\textbf{a}} & = \sum_{i=1}^n \sum_{b_i} \left(\frac{m_i m_{i+1}  }{1 - t m_i m_{i+1} S^{(B)}}\right)_{a_i b_i} \cdot \mathcal{A}_{\textbf{a}^{(i)}},\quad m_i=m(\sigma_i) \nonumber \\
    \quad \textbf{a}^{(i)}&  = (a_1, \ldots, a_{i-1}, b_i, a_{i+1}, \ldots, a_n).
\end{align}  
We also define the semigroup operator $\mathcal{U}_{s,t,\boldsymbol{\sigma}}$ associated to $\Theta_{t,\boldsymbol{\sigma}}$ as follows
    \begin{align}\label{def_Ustz}
        \left(\mathcal{U}_{s, t, \boldsymbol{\sigma}} \circ \mathcal{A}\right)_{\textbf{a}} = \sum_{b_1, \ldots, b_n} \prod_{i=1}^n \left(\frac{1 - s \cdot m_i m_{i+1} S^{(B)}}{1 - t \cdot m_i m_{i+1} S^{(B)}}\right)_{a_i b_i} \cdot \mathcal{A}_{\textbf{b}}, \quad \textbf{b} = (b_1, \ldots, b_n), \quad m_i=m(\sigma_i)
    \end{align}
\end{definition}
Note that for any $\xi \in \mathbb{C}$, we have the following identity:
 \begin{align}\label{def_Ustz_2}
\left(\frac{1 - s \xi S^{(B)}}{1 - t \xi S^{(B)}}\right) = I - (s - t) \xi \cdot \Theta^{(B)}_{t \xi} \cdot S^{(B)}.
\end{align}
By the explicit formula \eqref{Kn2sol} for ${\cal K}$ of rank $2$, the first term on the right-hand side of \eqref{eq_L-Keee} can be rewritten as $({\varTheta}_{t, \boldsymbol{\sigma}} \circ (\mathcal{L} - \mathcal{K}))_{\textbf{a}} = [\mathcal{K} \sim (\mathcal{L} - \mathcal{K})]^{l_\mathcal{K} = 2}_{t, \boldsymbol{\sigma}, \textbf{a}}$. Using this and the Duhamel principle, we have the following.

\begin{lemma}[Integrated loop hierarchy] 
    \label{Sol_CalL}
For any $t\geq0$, we have the SDE
\begin{align}\label{LK_SDE}
d({\cal L}-{\cal K})_{t,\boldsymbol{\sigma},\textbf{a}}=(\varTheta_{t,\boldsymbol{\sigma}}\circ({\cal L}-{\cal K})_{t,\boldsymbol{\sigma},\cdot})_{\textbf{a}}dt+\sum_{l_\mathcal{K} > 2} \Big[\mathcal{K} \sim (\mathcal{L} - \mathcal{K})\Big]^{l_\mathcal{K}}_{t, \boldsymbol{\sigma}, \textbf{a}}dt + \mathcal{E}^{((\mathcal{L} - \mathcal{K}) \times (\mathcal{L} - \mathcal{K}))}_{t, \boldsymbol{\sigma}, \textbf{a}} +
    \mathcal{E}^{(M)}_{t, \boldsymbol{\sigma}, \textbf{a}} +
    \mathcal{E}^{(\widetilde{G})}_{t, \boldsymbol{\sigma}, \textbf{a}}.
\end{align}
For any $s<t$, we have the following ``integrated" loop hierarchy:
\begin{align}\label{int_K-LcalE}
    (\mathcal{L} - \mathcal{K})_{t, \boldsymbol{\sigma}, \textbf{a}} \;=\; \; &
    \left(\mathcal{U}_{s, t, \boldsymbol{\sigma}} \circ (\mathcal{L} - \mathcal{K})_{s, \boldsymbol{\sigma}}\right)_{\textbf{a}} \nonumber \\
    &+ \sum_{l_\mathcal{K} > 2} \int_{s}^t \left(\mathcal{U}_{u, t, \boldsymbol{\sigma}} \circ \Big[\mathcal{K} \sim (\mathcal{L} - \mathcal{K})\Big]^{l_\mathcal{K}}_{u, \boldsymbol{\sigma}}\right)_{\textbf{a}} du \nonumber \\
    &+ \int_{s}^t \left(\mathcal{U}_{u, t, \boldsymbol{\sigma}} \circ \mathcal{E}^{((\mathcal{L} - \mathcal{K}) \times (\mathcal{L} - \mathcal{K}))}_{u, \boldsymbol{\sigma}}\right)_{\textbf{a}} du  \nonumber \\
    &+ \int_{s}^t \left(\mathcal{U}_{u, t, \boldsymbol{\sigma}} \circ \mathcal{E}^{(\widetilde{G})}_{u, \boldsymbol{\sigma}}\right)_{\textbf{a}} du + \int_{s}^t \left(\mathcal{U}_{u, t, \boldsymbol{\sigma}} \circ \mathcal{E}^{(M)}_{u, \boldsymbol{\sigma}}\right)_{\textbf{a}}  
\end{align}
Furthermore, let $T\geq s$ be a stopping time with respect to the matrix Brownian motion $H_t$ and set $\tau:= T\wedge t$. We have the stopped integrated loop hierarchy below for any $t\geq s$:
\begin{align}\label{int_K-L_ST}
   (\mathcal{L} - \mathcal{K})_{\tau , \boldsymbol{\sigma}, \textbf{a}} \;=\; \; &
    \left(\mathcal{U}_{s, \tau, \boldsymbol{\sigma}} \circ (\mathcal{L} - \mathcal{K})_{s, \boldsymbol{\sigma}}\right)_{\textbf{a}}  \nonumber \\
    &+ \sum_{l_\mathcal{K} > 2} \int_{s}^\tau \left(\mathcal{U}_{u, \tau, \boldsymbol{\sigma}} \circ \Big[\mathcal{K} \sim (\mathcal{L} - \mathcal{K})\Big]^{l_\mathcal{K}}_{u, \boldsymbol{\sigma}}\right)_{\textbf{a}} du  \nonumber \\
    &+ \int_{s}^\tau \left(\mathcal{U}_{u, \tau, \boldsymbol{\sigma}} \circ \mathcal{E}^{((\mathcal{L} - \mathcal{K}) \times (\mathcal{L} - \mathcal{K}))}_{u, \boldsymbol{\sigma}}\right)_{\textbf{a}} du  \nonumber \\
    &+ \int_{s}^\tau \left(\mathcal{U}_{u, \tau, \boldsymbol{\sigma}} \circ \mathcal{E}^{(\widetilde{G})}_{u, \boldsymbol{\sigma}}\right)_{\textbf{a}} du + \int_{s}^\tau \left(\mathcal{U}_{u, \tau, \boldsymbol{\sigma}} \circ \mathcal{E}^{(M)}_{u, \boldsymbol{\sigma}}\right)_{\textbf{a}}.
\end{align}

 
\end{lemma}

The last term in \eqref{int_K-L_ST} has the following meaning. Since ${\cal U}$ is deterministic, the object obtained by replacing ${\cal U}_{u,\tau,\boldsymbol{\sigma}}$ by ${\cal U}_{u,t,\boldsymbol{\sigma}}$ is well-defined, and this object is continuous in $t$, so we can evaluate it at time $t=\tau$. Of course, one must be careful in using It\^{o} calculus to estimate it.

We now introduce notation that is necessary to compute the quadratic variation of the martingale.

\begin{definition}[Definition of $\mathcal{E}\otimes  \mathcal{E}$]\label{def:CALE} 
Denote 
  \begin{align}\label{defEOTE}
 \left( \mathcal{E}\otimes  \mathcal{E} \right)_{t, \,\boldsymbol{\sigma}, \,\textbf{a}, \,\textbf{a}'} \; := \; &
 \sum_{k=1}^n\left( \mathcal{E}\otimes  \mathcal{E} \right)^{(k)}_{t, \,\boldsymbol{\sigma}, \,\textbf{a}, \,\textbf{a}'}
 \nonumber \\
 \left( \mathcal{E}\otimes  \mathcal{E} \right)^{(k)}_{t, \,\boldsymbol{\sigma}, \,\textbf{a}, \,\textbf{a}'} :
\; = \; & W^{2} \sum_{b,b'} S^{(B)}_{b,b'}{\cal L}_{t, \boldsymbol{\sigma}^{(k ) },\textbf{a}^{(k )}},\quad \boldsymbol{\sigma}^{(k ) }\in \{+,-\}^{2n+2}, 
\end{align}
where ${\cal L}_{t, \boldsymbol{\sigma}^{(k ) },\textbf{a}^{(k )}}$ is obtained by cutting the $k$-th edge of ${\cal L}_{t,\boldsymbol{\sigma},\textbf{a}}$ and then attaching itself (with indices $\textbf{a}$) to its complex conjugate loop (with indices $\textbf{a}'$) into a bigger loop that is labeled by indices $b$ and $b'$, so that
\begin{align}\label{def_diffakn_k}
   & \textbf{a}^{(k )}=( a_k,a_{k+1},\cdots a_n, a_1,\cdots a_{k-1}, b', a'_{k-1}\cdots a_1', a_n'\cdots a'_{k},b)
    \nonumber \\
  &  \boldsymbol{\sigma}^{(k )}=( \, \sigma_k, \, \sigma_{k+1},\cdots \sigma_n,  \, \sigma_1,\cdots \sigma_{k}, \,  \overline\sigma_{k},  \cdots \overline\sigma_1 , \,  \overline\sigma_n \cdots \overline\sigma_{k })
\end{align}
The symbol $\otimes$ is used to emphasize the symmetric structure of $\mathcal{E} \otimes \mathcal{E}$; it is not a tensor product.

\begin{figure}[ht]
\centering
\begin{minipage}{0.45\textwidth} 
    \centering
    \begin{tikzpicture}[scale=1.0]

    \node[draw, circle, inner sep=1pt, fill=black] (A1) at (-2, 2) {};
    \node at (-2.3, 2) {$a_1$}; 

    \node[draw, circle, inner sep=1pt, fill=black] (A2) at (0, 3) {};
    \node at (0 , 3.3) {$a_2$}; 

    \node[draw, circle, inner sep=1pt, fill=black] (A3) at (2, 2) {};
    \node at (2.3, 2 ) {$a_3$}; 

    \node[draw, circle, inner sep=1pt, fill=black] (B) at (-0.5, 1) {};
    \node at (-0.9, 1) {$b$}; 

    \node[draw, circle, inner sep=1pt, fill=black] (B1) at (0.5, 1) {};
    \node at (0.9, 1) {$b'$}; 

    \node[draw, circle, inner sep=1pt, fill=black] (A1b) at (-2, 0) {};
    \node at (-2.3, -0) {$a'_1$}; 

    \node[draw, circle, inner sep=1pt, fill=black] (A2b) at (0, -1) {};
    \node at (0, -1.3) {$a'_2$}; 

    \node[draw, circle, inner sep=1pt, fill=black] (A3b) at (2, 0) {};
    \node at (2.3, -0) {$a'_3$}; 

    \draw (A1) -- (A2) node[midway, above left] {\small $G_2$}; 
    \draw (A2) -- (A3) node[midway, above right] {\small $G_3$}; 
    \draw (A1) -- (B) node[midway, below left, xshift=-2pt, yshift= 5pt] {\small $G_1$}; 
    \draw (A3) -- (B1) node[midway, below right, xshift=2pt, yshift= 5pt] {\small $G_1$}; 

    \draw (B) -- (A1b) node[midway, above left,  xshift=-2pt, yshift=-5pt] {\small $\overline{G_1}$}; 
    \draw (A1b) -- (A2b) node[midway, below left, yshift=-2pt] {\small $\overline{G_2}$}; 
    \draw (A2b) -- (A3b) node[midway, below right, yshift=-2pt] {\small $\overline{G_3}$}; 
    \draw (B1) -- (A3b) node[midway, above right, xshift=2pt, yshift=-5pt] {\small $\overline{G_1}$}; 

    \draw[decorate, decoration={snake, amplitude=0.4mm, segment length=2mm}, thick] (B) -- (B1)
        node[midway, above] {$S^{(B)}$};

    \end{tikzpicture}
   
\end{minipage}%
\hfill
\begin{minipage}{0.45\textwidth} 
    \raggedright
    $$k=1 $$\\ 
    $$\textbf{a}^{(1 )}=( a_1, a_2, a_3, b', a_3', a_2', a_1',b);$$
    \\  
     $$\boldsymbol{\sigma}^{(1 )}=(\;\sigma_1, \sigma_2, \sigma_3,\sigma_1, \overline{\sigma_1},\overline{\sigma_3}, \overline{\sigma_2},\overline{\sigma_1}, ).$$
\end{minipage}
 \caption{Example of the $8-G$ loop in $\mathcal{E}\otimes  \mathcal{E}$: for $\sigma\in \{+,-\}^3$. Taken from \cite{YY_25}.}
    \label{fig:8shapedgraph}
\end{figure}

\end{definition}

 We now estimate the martingale term in \eqref{int_K-L_ST} via Lemma 5.5 in \cite{YY_25}, which we record below.

\begin{lemma}\label{lem:DIfREP}
For any stopping time $T$ with respect to the filtration generated by $H_t$ we have
 \begin{equation}\label{alu9_STime}
   \mathbb{E} \left [    \int_{s}^\tau\left(\mathcal{U}_{u, t, \boldsymbol{\sigma}} \circ \mathcal{E}^{(M)}_{u, \boldsymbol{\sigma}}\right)_{\textbf{a}} \right ]^{2p} 
   \le C_{n,p} \; 
   \mathbb{E} \left(\int_{s}^\tau 
   \left(\left(
   \mathcal{U}_{u, t,  \boldsymbol{\sigma}}
   \otimes 
   \mathcal{U}_{u, t,  \overline{\boldsymbol{\sigma}}}
   \right) \;\circ  \;
   \left( \mathcal{E} \otimes  \mathcal{E} \right)
   _{u, \,\boldsymbol{\sigma}  }
   \right)_{{\textbf a}, {\textbf a}}du 
   \right)^{p}
  \end{equation} 
  where $\tau:=t\wedge T$, where  $\overline{\boldsymbol{\sigma}}$ is the conjugate of ${\boldsymbol{\sigma}}$, and where
  $$
  \left[\left(
   \mathcal{U}_{u, t,  \boldsymbol{\sigma}}
   \otimes 
   \mathcal{U}_{u, t,  \overline{\boldsymbol{\sigma}}}
   \right)\circ \cal A\right]_{\textbf{a},\textbf{a}'}=
   \sum_{\textbf{b},\;\textbf{b}'} \; \prod_{i=1}^n \left(\frac{1 - u \cdot m_i m_{i+1} S^{(B)}}{1 - t \cdot m_i m_{i+1} S^{(B)}}\right)_{a_i b_i}
   \;\cdot \;\prod_{i=1}^n \left(\frac{1 - u \cdot \overline{m_i m_{i+1}} S^{(B)}}{1 - t \cdot \overline{m_i m_{i+1}} S^{(B)}}\right)_{a'_i b'_i} \cdot \mathcal{A}_{\textbf{b},\textbf{b}'}.
  $$
\end{lemma}

\subsection{Proof of Theorem \ref{lem:main_ind}, Step 2}
In this subsection, we will fix $\boldsymbol{\sigma}=(+,-)$ and thus drop it from the notation.  We prove \eqref{Eq:Gdecay_w} first;  \eqref{Gt_bound_flow} will be proved at the end of this  section. 
We start by defining the following exponential decay profiles:
\begin{equation}\label{def_WTu}
      {\cal T}^{(\cal L-K)}_u (\ell):= \max_{a,\,b\; : \; |a-b|\ge \ell} \left| ({\cal L-K})_{u,  \boldsymbol{\sigma},  ( a, b)}\right|,\quad  \boldsymbol{\sigma}=(+,-),
\end{equation}
\begin{equation}\label{def_WTuD}
{\cal T} _{u, D} (\ell) := M_{u}^{-2}
\exp \left(- \left( \ell /\ell_u \right)_+^{1/2} \right)+W^{-D}.
\end{equation} 
Both  ${\cal T}^{(\cal L-K)}_u$ and ${\cal T} _{u, D}$ are non-decreasing in the parameter $\ell$, so for any $0\leq\ell_{1}\leq\ell_{2}$, we have 
 ${\cal T}^{(\cal L-K)}_u (\ell_1)\ge {\cal T}^{(\cal L-K)}_u (\ell_2)$ and ${\cal T} _{u,D} (\ell_1)\ge{\cal T} _{u,D}(\ell_2)$. We will denote the ratio between \eqref{def_WTu} and \eqref{def_WTuD} by  
\begin{equation}\label{def_Ju}
     {\cal J} _{u,D} (\ell):=   \left({\cal T}^{(\cal L-K)}_u (\ell) \Big/{\cal T}_{u,D} (\ell)\right)+1
\end{equation} 
Our goal is to control $ {\cal J} _{u,D} (\ell)$ for any $u\in [s,t]$ and any large $D>0$. In particular, we aim to show 
 \begin{align}\label{shoellJJ}
    {\cal J}^* _{u,D} :=\max_{\ell}   {\cal J} _{u,D} (\ell)\le (\eta_s/\eta_u)^4 
  \end{align}
In this subsection, we will often use the following slightly enlarged length-scale:
$$
  \ell^{*}_t := (\log W)^{3/2}\cdot \ell_t .
$$
In this scale $\ell_t^*$,  $\Theta_t$ is exponentially small, but ${\cal T}_t$ is not. 
More precisely, we have the following lemma, which is Lemma 5.6 in \cite{YY_25}. (The proof is identical to that of Lemma 5.6 in \cite{YY_25}, so we omit it. In any case, the first and third estimates below are elementary, and the last follows by decay of ${\cal K}$, which follows from the explicit formula \eqref{Kn2sol}, decay of $\Theta$ from \eqref{prop:ThfadC}, and decay of ${\cal L}-{\cal K}$.)
\begin{lemma}\label{lem_dec_calE_0}
For any fixed $D,\delta>0$, we have
 \begin{align}
 \label{Kell*}
   |b -a|_{L}\ge \delta \cdot \ell_t^* &\implies 
   \left({\Theta}_t\right)_{ab } \le W^{-D} ,\quad    \left({\Theta}_s^{-1}{\Theta}_t\right)_{ab } \le W^{-D}
   \\
  |b -a|_{L}\ge \delta \cdot \ell_t^* 
 & \implies  {\cal L}_{t, (-,+),(a, b)} \prec {\cal J}^*_{t,D}\cdot{\cal T}_{t,D}(|a-b|_{L}) . \label{auiwii}
\end{align}
Moreover, for any $C=O(1)$, we have \begin{equation}\label{Tell*}
       { {\cal T}_{u,D} \left(\ell-C\cdot \ell_u^*\right)}\;\prec \;  {\cal T}_{u,D} \left(\ell \right).
\end{equation} 

\end{lemma}

Our proof of \eqref{Eq:Gdecay_w} relies on the loop hierarchy \eqref{int_K-L_ST} for $n=2$. 
We begin by bounding  terms in \eqref{int_K-L_ST}.  

\begin{lemma}\label{lem_dec_calE}  Suppose  the assumptions of Theorem  \ref{lem:main_ind} and the conclusion  of Step 1, i.e., \eqref{lRB1} and \eqref{Gtmwc}, 
hold. Assume that the following are satisfied, in which $\textbf{a}=(a_{1},a_{2})$ and $\textbf{a}'=(a_{1}',a_{2}')$ and $\boldsymbol{\sigma}=(+,-)$:
\begin{align} 
s\le u\le t, \quad  D\ge 10, \quad 
\max_i |a_i-a_i'|_{L}\le \ell_t ^*. \label{57}
\end{align}
Then we have
\begin{align}\label{res_deccalE_lk}
       {\cal E}^{( (L-K)\times(L-K) )}
       _{u, \boldsymbol{\sigma},\textbf{a}} \Big/\; 
  {\cal T}_{t,D}(|a_1-a_2|_{L}) \quad 
 \;  &\; \prec \; 
   \left(\eta_u\right)^{-1} 
   \cdot  
   M_{u}^{-1}
    \cdot 
 \left({\cal J}^{*}_{u,D} \right)^{2}
 \\
      \mathcal{E}^{( \widetilde{G} )}_{u, \boldsymbol{\sigma},\textbf{a}} 
    \;\Big/\; 
 {\cal T}_{t,D}(|a_1-a_2|_{L})
 \;  &\; \prec \; 
 \left(\eta_u\right)^{-1} 
  \cdot \left(\ell_u/\ell_s\right)^{4}\cdot{\bf 1}(|a_1-a_2|_{L}\le \ell_u^*)
    \nonumber  \\
       \;  &\; + \;  \left(\eta_u\right)^{-1} 
  \cdot
       M_{u}^{-1/3}  \cdot 
       \left({\cal J}^{*}_{u,D} \right)^{3}
 \label{res_deccalE_wG}\\
\label{res_deccalE_dif}
    \left(
   \mathcal{E}\otimes  \mathcal{E} \right)
   _{u, \,\boldsymbol{\sigma},  \,{\textbf a} ,\, {\textbf a}'
     }
     \;\Big/\; 
 \left( {\cal T}_{t,D}(|a_1-a_2|_{L})\right)^2
 \;  &\; \prec \; 
  \left(\eta_u\right)^{-1} 
  \cdot \left(\ell_u/\ell_s\right)^{10}\cdot{\bf 1}(|a_1-a_2|_{L}\le 4\ell_t^*)
     \nonumber \\
       \;  &\; + \;  \left(\eta_t\right)^{-1} 
  \cdot
       M_{u}^{-1/2}  \cdot 
    \left({\cal J}^{*}_{u,D} \right)^{3}
\end{align} 
  \end{lemma}

Assuming  Lemma \ref{lem_dec_calE}, we now  prove  \eqref{Eq:Gdecay_w}. We will use extensively for $\cal U$ from Section \ref{ks}.

\begin{proof}[Proof of \eqref{Eq:Gdecay_w}]
By  assumption \eqref{Eq:Gdecay+IND} on $({ \cal L-\cal K})_s$, the operator norm bound on $\cal U$ in Lemma \ref{lem:sum_Ndecay},  and the  tail estimate  \eqref{neiwuj}, 
we can  bound the first term on the right side  of \eqref{int_K-L_ST} by 
\begin{equation} \label{res_deccalE_0}
\left(\mathcal{U}_{s, t, \boldsymbol{\sigma}} \circ (\mathcal{L} - \mathcal{K})_{s, \boldsymbol{\sigma}}\right)_{\textbf{a}}\;\Big/\; 
  {\cal T}_{t, D}(|a_1-a_2|_{L}) \quad 
  \prec \;   (\ell_t/\ell_s)^{4}\cdot {\bf 1}(|a_1-a_2|_{L}\le \ell_t^*) + 1,
\end{equation}
where the last term of order one comes from applying \eqref{neiwuj}. The factor $\left(\eta_u / \eta_t\right)^2$  from applying Lemma \ref{lem:sum_Ndecay} becomes 
$ (\ell_t/\ell_s)^{4}$ due to the prefactors in $ {\cal T}_{s, D}$ and  ${\cal T}_{t, D}$.
For  any $\textbf{a}$ fixed and any function $f$,  we decompose $f  = f (\textbf{b}) {\bf 1} ( \| \textbf{b}-\textbf{a}\|\le \ell_t^*)  + f_2   $. 
From the decay of ${\cal U}_{u,t}$,  we know that $( {\cal U}_{u,t, \boldsymbol{\sigma}} \circ f_2 )_\textbf{a}$ is exponentially small.  
Hence we only have to bound $  {\cal U}_{u,t, \boldsymbol{\sigma}} \circ f_1$, for which we apply  Lemma \ref{lem:sum_Ndecay}. Ultimately, for the third  term of \eqref{int_K-L_ST}, we have the following estimate, in which $\|\textbf{b}-\textbf{a}\|:=\max_{i}|b_{i}-a_{i}|_{L}$:
$$
 \Big({\cal U}_{u,t, \boldsymbol{\sigma}}\circ  {\cal E}^{( (L-K)\times(L-K) )}_{u, \boldsymbol{\sigma}}\Big)_\textbf{a}
 \prec (\eta_u/\eta_t)^2 \max_{\| \textbf{b}-\textbf{a}\|\le \ell_t^*}{\cal E}^{( (L-K)\times(L-K) )}_{u, \boldsymbol{\sigma}, \textbf{b}}+W^{-D}.
$$
Because we restrict to $\|\textbf{b}-\textbf{a}\|\leq\ell_{t}^{*}$, \eqref{Tell*} implies $ {\cal T}_{t, D}(|b_1-b_2|_{L})\prec  {\cal T}_{t, D}(|a_1-a_2|_{L})$. Thus, to control the right-hand side of the previous estimate, we can use the estimate \eqref{res_deccalE_lk} and closeness of $s,t$ in \eqref{con_st_ind} to get 
\begin{align} 
\label{res_deccalE_1}
     \Big({\cal U}_{u,t, \boldsymbol{\sigma}}\circ  {\cal E}^{( (L-K)\times(L-K) )}_{u, \boldsymbol{\sigma}}\Big)_\textbf{a}\;\Big/\; 
 {\cal T}_{t, D}(|a_1-a_2|_{L}) \quad 
  & \prec \; 
  \frac{1}{\eta_u}\left(\eta_u / \eta_t\right)^2 
   \cdot  
   M_{u}^{-1 }
    \cdot 
\left({\cal J}^*_{u,D}\right)^2
\end{align} 
Similarly,   the estimate \eqref{res_deccalE_wG} on ${\cal E}^{(\widetilde{G} )}$ implies   
\begin{align}  \label{res_deccalE_2}
     \Big(\mathcal{U}_{u,t, \boldsymbol{\sigma}} \circ  \mathcal{E}^{( \widetilde{G} )}_{u, \boldsymbol{\sigma}}\Big)_{\mathbf{a}} 
    \;\Big/\; 
{\cal T}_{t, D}(|a_1-a_2|_{L})
 &\prec \;  {\bf1}(|a_1-a_2|_{L}\le 3\ell_t^*)\cdot  \frac{1} {\eta_u} \cdot \left(\eta_u / \eta_t\right)^2\cdot \left(\ell_u / \ell_s\right)^2 
 \nonumber \\
  &+\;    \frac{1}{\eta_u} 
  \cdot\left(\eta_u / \eta_t\right)^4 \cdot
      M_{u}^{-1/3}  \cdot \left({\cal J}^{*}_{u,D} \right)^{3 }, 
       \end{align} 
and  the  estimate  \eqref{res_deccalE_dif} on $\cal E\otimes\cal E$  implies  
\begin{align} \label{res_deccalE_3}
  \left(\left(
   \mathcal{U}_{u, t,  \boldsymbol{\sigma}}
   \otimes 
   \mathcal{U}_{u, t,  \overline{\boldsymbol{\sigma}}}
   \right) \;\circ  \;
   \left( \mathcal{E} \otimes  \mathcal{E} \right)
   _{u, \,\boldsymbol{\sigma}  }
   \right)_{{\textbf a}, {\textbf a}}
     \;\Big/\; 
 \left({\cal T}_{t, D}(|a_1-a_2|_{L})\right)^2
 &\prec \;  {\bf1}(|a_1-a_2|_{L}\le 6\ell_t^*)\cdot  \frac{1} {\eta_u} \cdot \left(\eta_u / \eta_t\right)^4\cdot \left(\ell_u / \ell_s\right)^{10} 
 \nonumber \\
 &+ \;  \frac{1}{\eta_u} 
  \cdot
       M_{u}^{-1/3}  \cdot \left({\cal J}^{*}_{u,D} \right)^{3 }.
\end{align}
In the last inequality, we absorbed the factor $\left(\eta_u / \eta_t\right)^4$ in the change of the exponent for $M_{u}=W^{2} \eta_u \ell_u^{2}$ from $-1/2$ to $-1/3$. We now insert  these bounds  into the stopped equation \eqref{int_K-L_ST} with $\tau:=T\wedge t$, where
\begin{align}
T:=\min \{u\geq s: {\cal J}^*_{u,D} \ge \left(\eta_s / \eta_t\right)^4\}
\end{align}
(If we divide each estimate in Lemma \ref{lem_dec_calE} by its right-hand side, the resulting estimates all hold with probability $1-O(N^{-D})$ for any $D>0$ simultaneously for all $u\in[s,t]$. This follows by another continuity argument.) With this stopping time, the quadratic variation of the martingale in \eqref{int_K-L_ST} is bounded as follows:
\begin{align}\label{51}
& \int_s^\tau  \left( \left(\mathcal{U}_{u,t,  \boldsymbol{\sigma}}
   \otimes 
   \mathcal{U}_{u, t,  \overline{\boldsymbol{\sigma}}}
   \right) \;\circ  \;
   \left( \mathcal{E}\otimes  \mathcal{E} \right)
   _{u, \,\boldsymbol{\sigma},\, \overline{\boldsymbol{\sigma}} }
   \right)_{{\textbf a}, {\textbf a}}  d u
 \Big/{\cal T}_{t, D}(|a_1-a_2|_{L})^2  \nonumber \\
 \prec  &   \int_s^\tau   d u \Big \{ 
 {\bf1}(|a_1-a_2|_{L}\le 6\ell_t^*)\cdot  \frac{1} {\eta_u} \cdot \left(\eta_u / \eta_t\right)^4\cdot \left(\ell_u / \ell_s\right)^{10}
+ \;  \frac{1}{\eta_u} 
  \cdot
       M_{u}^{-1/3}  \cdot \left({\cal J}^{*}_{u,D} \right)^{3 }  \Big \}   \nonumber \\
\le  &    \big [ (\eta_s/\eta_t)^{4}  \cdot {\bf 1}(|a_1-a_2|_{L}\le 6\ell_t^*)+1 \big ]
 \end{align}
Now, by Lemma \ref{lem:DIfREP} and the previous display, we deduce
\begin{align*}
\int_{s}^{\tau}\Big({\cal U}_{u,t,\boldsymbol{\sigma}}\circ{\cal E}^{(M)}_{u,\boldsymbol{\sigma}}\Big)_{\textbf{a}}\Big/{\cal T}_{t, D}(|a_1-a_2|_{L})\prec (\eta_{s}/\eta_{t})^{2}\mathbf{1}(|a_{1}-a_{2}|_{L}\leq6\ell_{t}^{*}]+1). 
\end{align*}
Because this bound holds at the level of stochastic domination, a union bound and a net argument lets us replace $t$ on the left-hand side by a supremum over $t'\in[s,t]$. In particular, we can evaluate it at $t'=\tau$. (We emphasize that the stochastic domination nature of the previous estimate is crucial for this argument.) Combining this  bound with  \eqref{res_deccalE_0}, \eqref{res_deccalE_1},   \eqref{res_deccalE_2}  and  \eqref{int_K-L_ST}, we have 
$$
\left({\cal L}-{\cal K}\right)_{\tau, \textbf{a}}\Big/{\cal T}_{t, D}(|a_1-a_2|_{L}) \prec   \big [ (\eta_s/\eta_t)^{2}  \cdot {\bf 1}(|a_1-a_2|_{L}\le 6\ell_t^*)+1 \big ], 
$$
where we used ${\cal J}^*_{s,D} \prec 1$ from \eqref{Eq:L-KGt+IND} and \eqref{Eq:Gdecay+IND}. By another continuity argument, the above estimate holds for all $t$ satisfying the constraint \eqref{con_st_ind} simultaneously with probability $1-O(N^{-D})$ for any $D>0$. In particular, it holds even at the random time $t=\tau$, which implies
\begin{align}\label{52}
 {\cal J}^*_{\tau,D}\prec     (\eta_s/\eta_t)^{2}  
\end{align}
 Hence 
${\mathbb P} (T\le t )\leq N^{-D}$ for any large $D>0$, and \eqref{Eq:Gdecay_w} follows.    
\end{proof}

 We note that the previous argument gives
\begin{align}\label{53}
\left({\cal L}-{\cal K}\right)_{t, \textbf{a}}\Big/{\cal T}_{t, D}(|a_1-a_2|_{L}) \prec   \big [ (\eta_s/\eta_t)^{2}  \cdot {\bf 1}(|a_1-a_2|_{L}\le 6\ell_t^*)+1 \big ].
\end{align}

 \begin{proof}[Proof of Lemma \ref{lem_dec_calE}]
 This argument is essentially the same as the proof of Lemma 5.7 in \cite{YY_25}, but we give the proof here because it is fairly long and technical, and because there are delicate power-countings that are dimension-dependent (as discussed at the beginning of this section).

\emph{Proof of \eqref{res_deccalE_lk}.}  
Recall the monotonicity property $u\leq t\Rightarrow\ell_{u}\leq \ell_{t}$, which itself implies $u\leq t\Rightarrow{\cal T}_{u,D}\leq{\cal T}_{t,D}$. By definition, we have
\begin{align} {\cal E}^{( ({\cal L-K})\times({\cal L-K}) )}_{u,  \boldsymbol{\sigma},  \textbf{a} }
= &W^{2}\sum_{b_1,b_2}({\cal L-K})_{u,
\boldsymbol{\sigma}, 
( a_1, b_1)
}
S^{(B)}_{b_1,b_2}
({\cal L-K})_{u,  \boldsymbol{\sigma},  ( b_2, a_2)}
\end{align}
By definition of $S^{(B)}$,  we can restrict the sum to $|b_1-b_2|_{L}\le 1$. Next, we use the elementary inequalities: 
$$
\int_{0 }^a\exp\left(-\sqrt{(a-x)}-\sqrt x+ \sqrt a\right)dx\le C\approx 6.12
\quad \text{ and } \quad 
\int_{0 }^\infty\exp\left({-\sqrt x}\right)dx=2.
$$
This gives $\sum_{x}{\cal T}_{u,D}(|a_{1}-x|_{L}){\cal T}_{u,D}(|a_{2}-x|_{L})\prec M_{u}^{-2}\ell_{u}^{2}{\cal T}_{u,D}(|a_{1}-a_{2}|_{L})$, where the factor of $M_{u}^{-2}$ comes from counting $M_{u}$ factors on each side of this identity using \eqref{def_WTuD}, and the factor of $\ell_{u}^{2}$ comes from re-scaling in $\ell$ of $\ell_{u}$ in \eqref{def_WTuD}. Therefore, we have the following, which implies \eqref{res_deccalE_lk}:
\begin{align}\nonumber
{\cal E}^{( ({\cal L-K})\times({\cal L-K}) )}_{u,  \boldsymbol{\sigma},  \textbf{a} }
\; \prec \;& W^{2} ({\cal J}_{u,D}^*) ^2\cdot 
\sum_{x }  {\cal T}_{u, \,D}(|a_1-x|_{L}) \cdot {\cal T}_{u, \,D}(|a_2-x|_{L})   
\\\label{mmxiaoxi}
\; \prec \;&
({\cal J}_{u,D}^*) ^2 \cdot 
 \eta_u^{-1}\cdot M_{u}^{-1}\cdot  {\cal T}_{u,D }(|a_1-a_2|_{L}) .
\end{align}

\medskip
\noindent 
\emph{Proof of   \eqref{res_deccalE_wG}}.    By definition, we  have 
\begin{align}
 \max_{\textbf{a}}{\cal E}^{(\widetilde G)}_{u,  \boldsymbol{\sigma},  \textbf{a} }
 \prec W^{2} \sum_{b_1,b_2} \langle \widetilde G_u E_{b_1}\rangle\cdot S^{(B)}_{b_1,b_2}\cdot  {\cal L}_{u, (-,+,+),(a_1, b_2, a_2)}+c.c.,\quad \widetilde{G}=G-m
 \end{align}
Recall $\Xi^{({\cal L})}_{t,m}:=\max_{\boldsymbol{\sigma},\textbf{a}}|{\cal L}_{t,\boldsymbol{\sigma},\textbf{a}}|\cdot M_{t}^{m-1}\cdot\mathbf{1}(\boldsymbol{\sigma}\in\{+,-\}^{m})$. By \eqref{GavLGEX}, \eqref{def_WTu}, \eqref{def_WTuD}, \eqref{def_Ju}, and \eqref{shoellJJ}, we have 
$$
\langle \widetilde G_u E_{b_1}\rangle\prec \Xi^{(\cal L)}_{u,2}\cdot M_{u}^{-1}\prec  M_{u}^{-1}+ {\cal J}^*_{u,D}\cdot M_{u}^{-2}.
$$
Therefore 
 \begin{align}\label{jiizziy}
\max_{\textbf{a}}{\cal E}^{(\widetilde G)}_{u,  \boldsymbol{\sigma},  \textbf{a} } \prec \frac{1}{ \ell_u^{2}\eta_u} \cdot \sum_{b }   \left|{\cal L}_{u, (-,+,+),(a_1, b , a_2)}\right|(1+{\cal J}^*_{u,D}\cdot M_{u}^{-1})
\end{align}
Now, let us first restrict to the case  $|a_1-a_2|_{L}\le \ell_u^*$. Define 
$$
\ell_u^{**}:= \ell_u\cdot (\log W)^3=\ell^*_u\cdot (\log W)^{3/2}
$$
Note that  ${\cal T}_{u,D}(\ell_u^{**})\prec W^{-D}$. Now, consider $b$ in \eqref{jiizziy} such that $|b-a_1|_{L}\le  \ell_u^{**}$; there are $\prec(\ell_{u}^{**})^{2}\prec\ell_{u}^{2}$ many such indices since we are in dimension $2$. For this restricted summation, we use \eqref{lRB1} to bound the loops in \eqref{jiizziy}. This gives
\begin{align}\label{alkkj}
 \sum_{b }{\bf 1}\big(|b-a_1|_{L}\le  \ell_u^{** }\big)\cdot  \left|{\cal L}_{u, (-,+,+),(a_1, b , a_2)}\right|\prec (\ell_u/\ell_s)^{4}\cdot M_{u}^{-2} \cdot \ell_u^{2}.
\end{align}
We are left to handle indices $|b-a_1|_{L}\ge \ell_u^{**}$. For this, we use a simpler estimate 
 $$
 \left|{\cal L}_{u, (-,+,+),(a_1, b , a_2)}\right|\prec \max_{x_1\in {\cal I}_{a_1}^{(2)},y\in {\cal I}_{b}^{(2)}}|G_{x_1,y}|.
$$
Next, we use \eqref{GijGEX} to estimate the right-hand side of the previous display:
$$
\max_{x_1\in {\cal I}_{a_1}^{(2)},y\in {\cal I}_{b}^{(2)}}|G_{x_1,y}|\prec \sum_{a_1', b'}\left({\cal L}_{u,(+,-),{(a_1',b')}}\right)^{1/2}{\bf1}\left(|a'_1-a_1|_{L}\le 1, |b-b'|_{L}\le 1\right).
$$
Since $|b-a_1|\ge \ell_u^{** }$, by  \eqref{auiwii} and ${\cal T}_{u,D}(|a_1-b|_{L})\prec W^{-D}$, we have  
$$
{\cal L}_{u,(+,-),{(a_1',b')}} \prec {\cal J}^*_{u,D}\cdot W^{-D}.
$$
Therefore, the contribution from indices $b$ in \eqref{jiizziy} satisfying $|b-a_1|\ge  \ell_u^{** }$ part is negligible, i.e.
\begin{align}\label{alkkj2}
 \sum_{b }{\bf 1}\big(|b-a_1|\ge  \ell_u^{** }\big)\cdot  \left|{\cal L}_{u, (-,+,+),(a_1, b , a_2)}\right|\prec  {\cal J}^*_{u,D}\cdot W^{-D}. 
\end{align}
We will not track the contribution of the last term; it does not affect the argument given below. Now, since $|a_{1}-a_{2}|_{L}\leq\ell_{u}^{\ast}$, we know that ${\cal T}_{u,D}(|a_{1}-a_{2}|_{L})$ is bounded away from $0$; this follows by \eqref{Tell*}. Combining this with \eqref{alkkj} and  \eqref{jiizziy} gives the following, which implies the desired estimate \eqref{res_deccalE_wG} for $|a_1-a_2|_{L}\le \ell_u^*$:
\begin{align}\label{82jjdopasj}
 |a_1-a_2|\le \ell_u^*\implies  \mathcal{E}_{u, \boldsymbol{\sigma}, \mathbf{a}}^{(\widetilde{G})} \Big/{\cal T}_{u,D}(|a_1-a_2|_{L})\prec 
 (\eta_u)^{-1} \cdot \left(\ell_u / \ell_s\right)^4\cdot   \left(1+  (\mathcal{J}_{u, D}^*)^2  M_{u}^{-1 }\right).
 \end{align}

We now prove the desired estimate \eqref{res_deccalE_wG} under the assumption that $|a_1-a_2|\ge \ell_u^*$. For this, we consider the following two cases:
 $$
 (1): \min_i|a_i-b|\le \ell_u^*/2,\quad \quad 
 (2): \min_i|a_i-b|\ge \ell_u^*/2.
 $$
In the first case,  we assume without loss of generality that $|a_1-b|_{L}\le \ell_u^*/2$. Applying the Schwarz inequality to  the two $G$ edges  in ${\cal L}_{u, (-,+,+),(a_1, b , a_2)}$ connecting the block $a_2$ gives
$$
G_{x_1x_2}G_{x_2y}G_{yx_1}\prec |G_{x_1x_2}|^2|G_{yx_1}|+
|G_{ x_2 y}|^2|G_{yx_1}|.
$$
Therefore, we have
\begin{align}\label{kskjw}
  {\cal L}_{u, (-,+,+),(a_1, b , a_2)}
  \;\prec\; &\;  W^{-2} \Big ( {\cal L}_{u, (-,+),(a_1, a_2)} \ \max_{x_1\in {\cal I}_{a_1}}\sum_{y\in {\cal I}_{b}}|G_{x_1y}|+ {\cal L}_{u, (-,+),(a_2, b)}\ \max_{y\in {\cal I}_{b}}\sum_{x_1\in {\cal I}_{a_1}}|G_{x_1y}| \Big )  .
\end{align}
(The additional factor of $W^{-2}$ follows because we are going from a $3$-loop to a $2$-loop by bounding a $G$-factor; we must also account for a $E_{\cdot}$ matrix, which carries a factor of $W^{-2}$.) If we use \eqref{GijGEX},  the $\cal K$ bound and $\sqrt { 1+c } \le 1+c $, \eqref{def_WTu}, \eqref{def_WTuD}, \eqref{def_Ju}, and \eqref{shoellJJ},  we have
$$
{\bf 1}(x_1\ne x_2)|G_{x_1x_2}|\prec M_{u}^{-1/2}(1+{\cal J}^*_{u,D}M_{u}^{-1}).
$$
Thus, for any $a,b$, we have 
\begin{align}\label{jaysw2002}
  W^{-2}  \max_{y\in{\cal I}_b^{(2)}}\sum_{x\in{\cal I}_a^{(2)}}|G_{xy}|\prec M_{u}^{-1/2}(1+{\cal J}^*_{u,D}M_{u}^{-1})
\end{align}
Plugging this bound  into \eqref{kskjw} gives 
\begin{align}
  {\cal L}_{u, (-,+,+),(a_1, b , a_2)}
     \;\prec\; &\;\left( {\cal L}_{u, (-,+),(a_1, a_2)}   + {\cal L}_{u, (-,+),(a_2, b)}  \right)  \cdot M_{u}^{-1/2}(1+{\cal J}^*_{u,D}M_{u}^{-1})
\end{align}
 Since $|a_2-b|\ge \ell_u^*/2$ and $|a_1-a_2|\ge  \ell_u^*$,  we can estimate  $ {\cal L}_{u, (-,+),(a_1, a_2)}   + {\cal L}_{u, (-,+),(a_2, b)}$ 
with \eqref{auiwii} to have  
$$
 {\cal L}_{u, (-,+,+),(a_1, b , a_2)} \;\prec\; {\cal J}^*_{u,D}\cdot \left(  {\cal T}_{u,D}(|a_1-a_2|)+{\cal T}_{u,D}(|b-a_2|) \right)  \cdot M_{u}^{-1/2}(1+{\cal J}^*_{u,D}M_{u}^{-1})
 $$
Since $|b-a_2|_{L}\ge |a_1-a_2|_{L}-|a_1-b|_{L}\ge |a_1-a_2|_{L}-\ell_u^*/2$, we can use \eqref{Tell*} to deduce
$${\cal T}_{u,D}(|b-a_2|) \prec  {\cal T}_{u,D}(|a_1-a_2|)
$$
Combining these bounds along with bound ${\cal J}^{*}_{u,D}\geq1$ gives the following for $|a_1-a_2|_{L}\ge \ell_u^*$:
\begin{align}
    \label{lwjufw}
    \sum_{b } ^{\text{Case } 1} \left|{\cal L}_{u, (-,+,+),(a_1, b , a_2)}\right|\prec \left({\cal J}^*_{u,D}\right)^2\cdot{\cal T}_{u,D}(|a_1-a_2|_{L}) \cdot \ell_u^{2} M_{u}^{-1/2} .
\end{align}
We now consider case (2), defined below \eqref{82jjdopasj}. We bound  ${\cal L}_{u, (-,+,+),(a_1, b , a_2)}$ as follows
\begin{align}\label{LGGGXXX}
{\cal L}_{u, (-,+,+),(a_1, b , a_2)}
\prec \max_{x_1,x_2, y} \left|G_{x_1y}\right| \left|G_{yx_2}\right|\left|G_{x_1x_2}\right|{\bf 1}(x_1\in {\cal I}_{a_1}^{(2)}){\bf 1}(x_2\in {\cal I}_{a_2}^{(2)}){\bf 1}(y\in {\cal I}_{b}^{(2)})
\end{align}
We note that since $a_1$, $a_2$ and $b$ are all different,  $x_1$, $x_2$ and $y$ must be all different as well.

Next we use \eqref{GijGEX} to bound each $G$ entry above, and then we apply \eqref{auiwii}. Next, recall that the distance between any two of $a_{1},a_{2},b$ is at least $\ell^{*}_{u}/2$. Thus, because of the indicators on the right-hand side of \eqref{LGGGXXX}. we deduce that $|x_{1}-y|_{L}\geq |a_{1}-b|_{L}-c\ell_{u}^{\ast}$. Doing this for the other pairs of indices above, namely $(y,x_{2})$ and $(x_{1},x_{2})$, we ultimately have
 \begin{align}\label{GGTLJ}
 \left|G_{yx_2}\right|\cdot \left|G_{x_1x_2}\right| \cdot
\left|G_{x_1y}\right|
\;\prec \; & ({\cal J}_{u,D}^*)^{3/2} \Big( {\cal T}_{u,D}(|a_1-a_2|_{L}) {\cal T}_{u,D}(|a_1-b|_{L}){\cal T}_{u,D}(|a_2-b|_{L})\Big)^{1/2}
 \end{align}
By combining \eqref{LGGGXXX} and \eqref{GGTLJ}, we arrive at
\begin{align}
    \sum_{b}^{ \text {Case } 2 }{\cal L}_{u, (-,+,+),(a_1, b , a_2)} 
 \;\prec \;  &\; ({\cal J}_{u,D}^*)^{3/2} \Big( {\cal T}_{u,D}(|a_1-a_2|_{L})\Big)^{1/2} \sum_b \Big({\cal T}_{u,D}(|a_1-b|_{L}){\cal T}_{u,D}(|a_2-b|_{L})\Big)^{1/2}
 \nonumber \\
  \;\prec \; &\;({\cal J}_{u,D}^*)^{3/2}  {\cal T}_{u,D}(|a_1-a_2|_{L}) \cdot  \ell_u^{2}\cdot  M_{u}^{-1}
\end{align}
where, to deduce the second line, we use
$$
\int_{0 }^a\exp\left(-\sqrt{(a-x)}/2-\sqrt x/2+ \sqrt a/2\right)dx\le C .
$$
We now combine this with \eqref{lwjufw}. We deduce that if $|a_1-a_2|_{L}\ge \ell_u^*$ then
\begin{align}
    \sum_{b} {\cal L}_{u, (-,+,+),(a_1, b , a_2)} 
  \;\prec \; &\; ({\cal J}_{u,D}^*)^{3/2}  {\cal T}_{u,D}(|a_1-a_2|_{L}) \cdot  \ell_u^{2}\cdot  M_{u}^{-1/2}
\end{align}
Plugging this into \eqref{jiizziy} and combining it with  \eqref{82jjdopasj} completes the proof of  \eqref{res_deccalE_wG}.

 \medskip
\emph{Proof of \eqref{res_deccalE_dif}. } By definition \eqref{defEOTE},  $(\mathcal{E} \otimes \mathcal{E})$ is the sum of  $(\mathcal{E} \otimes \mathcal{E})^{(k)}$, which can be written in terms of $
{\cal L}^{(k )}:= {\cal L}_{u, \boldsymbol{\sigma}^{(k ) },\textbf{a}^{(k )}} $. Here
$ \boldsymbol{\sigma}^{(1 )}=( +,-,+,-,+,-)$, $\textbf{a}^{(1 )}=( a_1,a_2, b', a_2',a_1',b)$
and $\boldsymbol{\sigma}^{( 2)}= $ $(-,+,-,+,-,+)$, $\textbf{a}^{( 2)}=( a_2,a_1, b', a_1',a_2',b)$. By symmetry, we only need to prove \eqref{res_deccalE_dif} for $(\mathcal{E} \otimes \mathcal{E})^{(1)}$. Note that $|b-b'|_{L} = 1$ in this case, so we can treat $b=b'$ for all practical purposes in the following argument. 

\noindent 
{\bf Case 1: $|a_1-a_2|_{L}\le 4\ell_t^*$ } 
We split the sum $\sum_{b,b'}$ into the following two regions: 
$$
|b-a_1|_{L}\le \ell_u^{** }:=(\log W)^3 \ell_u \quad\text{and}\quad |b-a_1|_{L}\ge \ell_u^{** } 
$$
Note that ${\cal T}_{u,D}(\ell_u^{** }) $ is exponentially small in $W$. This, and arguments similar to those used in \eqref{alkkj2}, gives
$$
\sum_{b,b'} {\bf 1}\left(|b-a_1|_{L}\ge \ell_u^{** }\right){\cal L}^{(1)}\prec {\cal J}^*_{u,D}\cdot W^{-3}.
$$
For $|b-a_1|_{L}\le  \ell_u^{** } $, we use \eqref{lRB1} for $n=6$. This gives
\begin{align}\label{alkkj3}
\sum_{b,b'}{\bf 1}\big(|b-a_1|_{L}\le  \ell_u^{** }\big) {\cal L}^{(1)}  \prec (\ell_u/\ell_s)^{10}\cdot M_{u}^{-5} \cdot \ell_u^{2}   .
\end{align}
If $|a_{1}-a_{2}|_{L}\leq 4\ell_{t}^{*}$, then ${\cal T}_{u,D}(|a_{1}-a_{2}|_{L})\prec1$ by \eqref{Tell*}. By this and the last two displays, we get \eqref{res_deccalE_dif}.
 \medskip

 \noindent 
{\bf Case 2: $|a_1-a_2|_{L}\ge 4\ell_t^*$ .} 
Recall  the assumption  $|a'_i-a_i|_{L}\le \ell_t^*$ \eqref{57} and the fact  that we can treat $b=b'$  in the following proof.  
We again split the sum $\sum_{b,b'}$ into two regions: 
$$
(1): |b -a_1|_{L}\le |b-a_2|_{L},\quad (2): |b-a_1|_{L}\ge |b-a_2|_{L}.
$$
By symmetry, it suffices to only consider case (1). Similar to \eqref{LGGGXXX}, we  bound ${\cal L}^{(1)}$ by the product of four $G$'s and  $G^\dagger E_b G$ as follows (recall that ${\cal L}^{(1)}$ is a six-$G$ loop, hence six factors of $G$):
\begin{align}\label{L6G6X}
{\cal L}^{(1)}\le \max^*_{x_1,x_2, y',x_1,x_2'} 
\left|G_{x_1x_2}\right|
\left|G_{x_2y'}\right|
\left|G_{y'x_2'}\right|
\left|G_{ x_2'x_1'}\right|
\left|\left(G^\dagger E_b G\right)_{ x_1x_1'}\right|
\left({\bf 1}_{x_1\ne x_1'} +W^{-1}{\bf 1}_{x_1= x_1'}\right).
\end{align}
The maximum is over all indices satisfying the condition
$$
 {\bf 1}(x_1\in {\cal I}_{a_1}^{(2)}){\bf 1}(x_2\in {\cal I}_{a_2}^{(2)}){\bf 1}(y'\in {\cal I}_{b'}^{(2)}){\bf 1}(x'_1\in {\cal I}_{a'_1}^{(2)}){\bf 1}(x'_2\in {\cal I}_{a'_2}^{(2)})=1.
$$
We now claim that 
\begin{align} \label{GGTLJ4G}
|G_{y'x_2}|\cdot |G_{y'x_2'}| 
\cdot 
|G_{x_1x_2}|\cdot |G_{x_1'x_2'}|
\prec 
\left({\cal J}^*_{u,D}\right)^2\cdot 
{\cal T} _{t,D}(|a_1-a_2|_{L})\cdot {\cal T} _{t,D}( |b-a_2|_{L})    
\end{align}
Let us prove this claim. Since we assumed $|a_1-a_2|_{L}\ge 4\ell_t^*$, we have, similar to \eqref{GGTLJ}, that 
\begin{align} 
 |G_{x_1x_2}|^2\le  & {\cal J}^*_{u,D}\cdot {\cal T}_{u,D}(|a_1-a_2|_{L}), \\
 |G_{x_1'x_2'}|^2\le & {\cal J}^*_{u,D}\cdot {\cal T}_{u,D}\left( |a_1'-a_2'|_{L}\right)  \prec {\cal J}^*_{u,D}\cdot {\cal T}_{t,D}\left( |a_1-a_2|_{L}\right), 
\end{align} 
where  we used ${\cal T}_u\le {\cal T}_t$, the assumption $|a_i'-a_i|_{L}\le \ell_t^*$ and \eqref{Tell*} to derive the second inequality.  Thus, 
$$
|G_{x_1x_2}|\cdot |G_{x_1'x_2'}|\prec   {\cal J}^*_{u,D}\cdot {\cal T}_{t,D}\left( |a_1-a_2|_{L}\right).
$$ 
For edges connected to $b'$, we similarly have
$$
|G_{y'x_2}|\cdot |G_{y'x_2'}| \prec 
{\cal J}^*_{u,D}\cdot  {\cal T}^{1/2}_{u,D}\left( |b-a_2|_{L} \right)\cdot   {\cal T}^{1/2}_{u,D}\left(  |b-a_2'|_{L} \right).
$$
Using  ${\cal T}_u\le {\cal T}_t$ and $|a_2-a_2'|_{L}\le \ell_t^*$, we have 
$$
|G_{y'x_2}|\cdot |G_{y'x_2'}|\le {\cal J}^*_{u,D}\cdot  {\cal T} _{t,D}\left(  |b-a_2|_{L} \right).
$$
By combining  these  bounds, we get \eqref{GGTLJ4G}. We now use Lemma \ref{lem_GbEXP_n2} with $\alpha=(\ell_u/\ell_s)^{2}$ and $\beta=1$.  
Doing so with the $\cal L$ estimates \eqref{lRB1} that we proved in Step 1 gives
$$
\left|\left(G^\dagger E_b G\right)_{ x_1x_1'}\right|
\left({\bf 1}_{x_1\ne x_1'} +W^{-2}{\bf 1}_{x_1= x_1'}\right) \prec (\ell_u/\ell_s)^{6}M_{u}^{-3/2}.
$$
We will now consider the following two cases:
\begin{align}
(1a): \;&\; |a_1-b|_{L}\le \ell_u^*,\quad \text{or}\quad |a_1'-b|_{L}\le \ell_u^*\nonumber \\
(1b):   \;&\;  |a_1-b|_{L}\ge \ell_u^*,\quad and\quad |a_1'-b|_{L}\ge \ell_u^*
\end{align} 
Consider the case (1a), so  $|a_2-b|_{L}\ge |a_1-a_2|_{L}-\ell_t^*$. We have ${\cal T} _{t,D}\left(  |b-a_2|_{L} \right) \prec 
{\cal T} _{t,D}\left(  |a_1-a_2|_{L} \right)$ by \eqref{Tell*}. 
Combining  these bounds  with \eqref{GGTLJ4G}, we have the following in which we bound ${\cal L}^{(1)}$ pointwise and sum over the $O(\ell_{u}^{\ast})^{2}$-many indices $b$ in case (1a):
\begin{align}
  \sum_{b,b'}^{(1a)} {\cal L}^{(1)}
 \prec &(\ell_u/\ell_s)^{6} \cdot \ell_u^{2} \cdot M_{u}^{-3/2}
\cdot\left({\cal J}^*_{u,D}\right)^2  \cdot {\cal T}_{t,D}(|a_1-a_2|_{L}) ^2.
\end{align}
We now consider the case (1b). Here, we have the estimate
$$
(G^\dagger E_bG)_{x_1x_1'}\le \max_{y\in {\cal I}_b^{(2)}}|G_{x_1y}||G_{x'_1y}|.
$$
Since  $ |a_1-b|_{L}$ and $|a_1'-b|_{L} \ge \ell_u^*$ and $|a_{1}-a_{1}'|_{L}\leq\ell_{t}^{*}$,  we have, similar to \eqref{GGTLJ}, 
$$
(G^\dagger E_bG)_{x_1x_1'}
\prec {\cal J}^*_{u,D} \cdot  {\cal T}_{t,D}(|b-a_1|_{L}) .
$$
Combining this with \eqref{GGTLJ4G} gives the following, in which the second line is bounded as in the proof of \eqref{mmxiaoxi}:
\begin{align}
  \sum_{b,b'}^{(1b)} {\cal L}^{(1)}
 \prec\; & \left( {\cal J}^*_{u,D}\right)^3 \cdot  {\cal T}_{t,D}(|a_1-a_2|_{L})\cdot \sum_b {\cal T}_{t,D}(|b-a_1|_{L}) {\cal T}_{t,D}(|b-a_2|_{L})
 \nonumber \\
\prec \; &  \left( {\cal J}^*_{u,D}\right)^3 \cdot \ell_t^{2} \cdot M_{t}^{-2} 
\cdot \left({\cal T}_{t,D}(|a_1-a_2|_{L})\right)^{2}.
 \end{align}
The bounds for cases (1a) and (1b) give \eqref{res_deccalE_dif} and thus complete the proof of Lemma  \ref{lem_dec_calE}. 
\end{proof}

\noindent 
\begin{proof}[Proof of \eqref{Gt_bound_flow}] Combining  the $\cal L-\cal K$ estimate  \eqref{Eq:Gdecay_w} and the $\cal K$ bound in Lemma \ref{ML:Kbound} gives
\begin{equation}
\label{lk2safyas}
 \max_{a, b}{\cal L }_{u, (+,-), (a,b)}\prec \left(\eta_s/\eta_u\right)^4\cdot M_{u}^{-2}+M_{u}^{-1}\prec M_{u}^{-1}  ,\quad u\in [s,t]    
\end{equation}
where we used the inductive hypothesis  \eqref{con_st_ind} to control $(\eta_{s}/\eta_{u})^{4}$. Now, we use the weak local law \eqref{Gtmwc}; this lets us apply \eqref{GiiGEX} and \eqref{GijGEX}. Using these bounds with \eqref{lk2safyas} gives 
$$
\|G_t-m\|^2_{\max}\prec  \max_{a, b}{\cal L }_{u, (+,-), (a,b)}\prec M_{u}^{-1} .
$$
This completes the proof of \eqref{Gt_bound_flow}.   
\end{proof}

\subsection{Fast decay property}
In the second step of the proof for Theorem \ref{lem:main_ind}, we established a strong decay property \eqref{Eq:Gdecay_w} for $2$-loops. We now extend this decay to general loops $\mathcal{L}$ and $\mathcal{L} - \mathcal{K}$. The idea is to use \eqref{GijGEX} to control entries of $G$ in terms of $2$-loops, and then use this to control general loops

We now define the desired $\ell_{u}$-decay property for general loops. We then state the decay property for ${\cal L}_{u}$, whose proof is identical to that of Lemma 5.9 in \cite{YY_25} and is thus omitted.

\begin{definition}\label{Def_decay}
   Let  $\cal A$ be a tensor $\mathbb Z_L^n\to \mathbb R$, and fix $\tau,D>0$. We say $\cal A$ has $(u, \tau, D)$ decay at  time $u$ if   
\begin{equation}\label{deccA}
    \max_i |a_i-a_j|_{L}\ge \ell_u W^{\tau} \implies  {\cal A}_{\textbf{a}}=O(W^{-D}),\quad \textbf{a}=(a_1,a_2\cdots, a_n) 
\end{equation}
\end{definition}
\begin{lemma}\label{lem_decayLoop}  
Assume that \eqref{Gt_bound_flow} and \eqref{Eq:Gdecay_w} hold. Then  for 
any $n\ge 2$, and any $\tau,D,D'>0$, the loop  ${\cal L}_u$ has $(u, \tau, D)$ decay  with probability $1-O(W^{-D'})$: 
\begin{align}\label{res_decayLK}
\mathbb P\left( \max_{\boldsymbol{\sigma}}  \Big(\left|{\cal L}_{u,\boldsymbol{\sigma},\textbf{a}}\right|+\left|{(\cal L-K})_{u,\boldsymbol{\sigma},\textbf{a}}\right|\Big)\cdot{\bf 1}\left(\max_{  i, j } |a_i-a_j|_{L}\ge \ell_u W^{\tau}\right) \ge W^{-D}\right)\le W^{-D'}. 
\end{align}
 \end{lemma}

We now use Lemma \ref{lem_decayLoop} to provide the following estimate for ${\cal E}$ terms.

\begin{lemma}\label{lem_BcalE} 
Assume  that \eqref{Gt_bound_flow} and \eqref{Eq:Gdecay_w} hold.  
Recall $\Xi^{({\cal L})}$ from \eqref{def:XiL}, and define
\begin{align}\label{def:XIL-K}
     \Xi^{({\cal L}-{\cal K})}_{t, m} &:= \max_{\boldsymbol{\sigma}, \textbf{a}} \left|({\cal L}-{\cal K})_{t, \boldsymbol{\sigma}, \textbf{a}} \right|\cdot M_{t}^{m} \cdot {\bf 1}(\boldsymbol{\sigma} \in \{+, -\}^m),
\end{align}Then, we have
\begin{align}\label{CalEbwXi}
\big[\mathcal{K} \sim(\mathcal{L}-\mathcal{K})\big]_{u, \boldsymbol{\sigma}}^{l_{\mathcal{K}}}
   \; \prec \; &
   \left( \max_{k<n }\Xi^{(\cal L-\cal K)}_{u,k}\right)\cdot M_{u}^{-n }\cdot \frac{1}{\eta_u}   
   \nonumber \\
    \mathcal{E}_{u, \boldsymbol{\sigma}}^{((\mathcal{L} - \mathcal{K}) \times (\mathcal{L} - \mathcal{K}))}
    \; \prec \; &
   \left( \max_{k:\;2\le k\le n }\Xi^{(\cal L-\cal K)}_{u,k}\cdot \Xi^{(\cal L-\cal K)}_{u,n-k+2}\cdot M_{u}^{-1}\right)\cdot M_{u}^{-n }\cdot \frac{1}{\eta_u}  
  \nonumber \\  \mathcal{E}_{u, \boldsymbol{\sigma}}^{(\widetilde{G})}
 \; \prec \; &
     \Xi^{(\cal L)}_{u,n+1} \cdot M_{u}^{-n }\cdot \frac{1}{\eta_u}    
  \nonumber \\   
(\mathcal{E} \otimes \mathcal{E}) _{u, \boldsymbol{\sigma}}
\; \prec \; &  \; \Xi^{(\cal L)}_{u, {2n+2}}\cdot   M_{u}^{-2n } \cdot \frac1{\eta_u}
\end{align}
Furthermore, for any $u\in[s,t]$, all of the $\cal E$ terms at time $u$ on the left-hand side have $(u, \tau, D)$ decay.  
 \end{lemma}
 
 \begin{proof}[Proof of Lemma \ref{lem_BcalE}] The proof  follows from  the definitions of these terms and Lemma \ref{lem_decayLoop}. Below, all  $\ell_u^{2}$ factors come from summing an index restricted to a length $\ell_u$-neighborhood in $\Z^{2}$, and $W^{2}$ factors come from our choice of scaling.
     \begin{itemize}
    \item We use the formula \eqref{DefKsimLK} for $[\mathcal{K} \sim(\mathcal{L}-\mathcal{K})]_{u, \boldsymbol{\sigma}}^{l_{\mathcal{K}}}$ and bound $\cal K$  with \eqref{eq:bcal_k}. This gives  
    \begin{align}
    \big[\mathcal{K} \sim(\mathcal{L}-\mathcal{K})\big]_{u, \boldsymbol{\sigma}}^{l_{\mathcal{K}}}
 \; \prec \; &
 W^{2}\cdot M_{u}^{-l_{\mathcal{K}}+1}\cdot \ell_u^{2}\cdot \Xi^{(\cal L-\cal K)}_{u,(n-l_{\mathcal{K}}+2)}\cdot M_{u}^{-(n-l_{\mathcal{K}}+2)}
    \end{align}
    where we used $l_{\mathcal{K}}\ge 3$. The $\ell_u^{2}$ factor comes from the sum over one index $a$ that is restricted by the decay of $\mathcal K$ and another index $b$ that is restricted by $S^{(B)}_{ab}$.
    
    \item By \eqref{def_ELKLK},  $ \mathcal{E}^{((\mathcal{L} - \mathcal{K}) \times (\mathcal{L} - \mathcal{K}))}$ is a product of two loops, whose total length is $n+2$. In particular, we have
    \begin{align}
   \mathcal{E}^{((\mathcal{L} - \mathcal{K}) \times (\mathcal{L} - \mathcal{K}))}
 \; \prec \; &
 W^{2}\cdot \sum_{k=2}^{n} \Xi^{(\cal L-\cal K)}_{u, \,k}
 \cdot M_{u}^{-k}
 \cdot \ell_u\cdot \Xi^{(\cal L-\cal K)}_{u,(n-k+2)}
 \cdot M_{u}^{-(n-k+2)}
      \end{align}
    
    \item By \eqref{def_EwtG}, 
    $\mathcal{E}^{(\widetilde{G})}$  is a product of $(\cal L-\cal K)$ with length one and an $(n+1)$-$G$  loop. We use  \eqref{GavLGEX} to bound said $G$ loop, and then we use \eqref{Eq:Gdecay_w} and \eqref{eq:bcal_k}, which implies $\Xi^{\cal L }_{u, \,2}\prec 1$. This gives 
    \begin{align}
  \mathcal{E}^{(\widetilde{G})}
 \; \prec \; &
 W^{2}\cdot   \Xi^{\cal L }_{u, \,2}
 \cdot M_{u}^{-1}
 \cdot \ell_u\cdot \Xi^{\cal L }_{u,(n+1)}
 \cdot M_{u}^{-n} \prec W^{2}   
 \cdot M_{u}^{-n-1}
 \cdot \ell_u\cdot \Xi^{\cal L }_{u,(n+1)}.
    \end{align} 
 \item By  \eqref{def_diffakn_k}, 
    $(\mathcal{E} \otimes \mathcal{E})$ can be written in terms of  $(2n+2)$-$G$-loops.  Thus 
    \begin{align}
(\mathcal{E} \otimes \mathcal{E}) \prec  \Xi^{(\cal L)}_{u, {2n+2}}\cdot  W^{2}\cdot\ell_u^{2}\cdot  M_{u}^{-2n-1 } .
    \end{align} 
\end{itemize}
This completes the proof.
 \end{proof}

\begin{lemma}\label{lem:STOeq_NQ}  
Suppose that the   assumptions of  Theorem \ref{lem:main_ind},  the local law \eqref{Gt_bound_flow} and the decay property for $\cal L-\cal K$ \eqref{Eq:Gdecay_w} hold, and that $\boldsymbol{\sigma}$ is not the alternating sign vector, in that
\begin{equation}\label{NALsigm}
    \exists k,\; s.t.\; \sigma_k=\sigma_{k+1}
\end{equation}. 
If $\max_{u\in [s,t]}\, \Xi^{(\cal L)}_{u, {2n+2}}\prec {  \Lambda}$ for some deterministic $ \Lambda \ge 1$, then 
\begin{align}\label{am;asoiuw}
 \Xi^{({\cal L-\cal K})}_{u,n}
 \;\prec\;  &\; { \Lambda}^{1/2}+\max_{u\in [s,t]}
  \left(\; \max_{k<n }\Xi^{({\cal L-\cal K})}_{u,k}+\max_{k:\;2\le k\le n }\Xi^{({\cal L-\cal K})}_{u,k}\cdot \Xi^{({\cal L-\cal K})}_{u,n-k+2}\cdot M_{u}^{-1}+\Xi^{({\cal L})}_{u,n+1}\right)
\end{align}

\end{lemma}

\begin{proof}
By Lemma \ref{lem:sum_decay}, the $\max\to\max$ norm for ${\cal U}_{u,t,\boldsymbol{\sigma}}$ for fast decay tensors and for $\boldsymbol{\sigma}$ satisfying \eqref{NALsigm} is $\prec[(\ell_{u}^{2}\eta_{u})/(\ell_{t}^{2}\eta_{t})]^{n}$. Using this operator norm bound, Lemma \ref{lem_BcalE}, and \eqref{int_K-LcalE}, we have 
\begin{align}\label{sahwNQ}
M_{t}^n\cdot      (\mathcal{L} - \mathcal{K})_{t, \boldsymbol{\sigma}, \textbf{a}} \;\prec\; \; &
  1+   \int_{s}^t  \left( \max_{k<n }\Xi^{({\cal L-\cal K})}_{u,k}\right) \cdot \frac{1}{\eta_u}   du 
    \nonumber \\
    &+ \int_{s}^t \left( \max_{k:\;2\le k\le n }\Xi^{({\cal L-\cal K})}_{u,k}\cdot \Xi^{({\cal L-\cal K})}_{u,n-k+2}\cdot M_{u}^{-1}\right)\cdot \frac{1}{\eta_u} du 
    \nonumber \\
    &+ \int_{s}^t \Xi^{(\cal L)}_{u,n+1}\cdot \frac1{\eta_u} du 
   +
   M_{t}^n\cdot\int_{s}^t \left(\mathcal{U}_{u, t, \boldsymbol{\sigma}} \circ    \mathcal{E}^{(M)}_{u, \boldsymbol{\sigma}}\right)_{\textbf{a}} du.
\end{align}
 Above, the operator norm bound of $\prec[(\ell_{u}^{2}\eta_{u})/(\ell_{t}^{2}\eta_{t})]^{n}$ is used to turn the factor of $M_{t}^{n}$ on the left-hand side into $M_{u}^{n}$, which is why we have $\Xi_{u,k}$-terms on the right-hand side. Next, we use Lemma \ref{lem:DIfREP} to get
\begin{equation} \label{sahwNQ2}
   \mathbb{E} \left [   M_{t}^{n} \int_{s}^t \left(\mathcal{U}_{u, t, \boldsymbol{\sigma}} \circ \mathcal{E}^{(M)}_{u, \boldsymbol{\sigma}}\right)_{\textbf{a}} \right ]^{2p} 
   \le C_{n,p} \; 
   \mathbb{E}    \left(M_{t}^{n}\int_{s}^t 
   \left(\left(
   \mathcal{U}_{u, t,  \boldsymbol{\sigma}}
   \otimes 
   \mathcal{U}_{u, t,  \overline{\boldsymbol{\sigma}}}
   \right) \;\circ  \;
   \left( \mathcal{E} \otimes  \mathcal{E} \right)
   _{u, \,\boldsymbol{\sigma}  }
   \right)_{{\textbf a}, {\textbf a}}du 
   \right)^{p}
  \end{equation} 
By \eqref{CalEbwXi}, the right-hand side of \eqref{sahwNQ2} is $\mathrm{O}_{\prec}(\int_{s}^{t}\Xi_{u,2n+2}^{({\cal L})}\eta_{u}^{-1}du)^{p}$, which completes the proof.
\end{proof}

\subsection{Sum zero operator \texorpdfstring{${\cal Q}_t$}{Qt} {and mean zero operator \texorpdfstring{${\mathbb  Q}$}{Q}}}\label{subsection-sum-zero}
By Lemma \ref{lem:STOeq_NQ}, we are left to consider the alternating sign vector now. \emph{In particular, we assume $\boldsymbol{\sigma}=\boldsymbol{\sigma}^{(alt)}$ for the rest of this subsection}. To do this, we now introduce the first of two operators that will be important for analyzing the loop hierarchy. 

\begin{definition}\label{Def:QtPt} Let ${\cal A}_\textbf{a}$ be a tensor with $\textbf{a}=(a_{1},\ldots,a_{n})\in (\mathbb Z_L^2)^n$ and $n\ge 2$. Define
$$
\left( {\cal P} \circ {\cal A}\right)_{a_1}= \sum_{a_2,\, a_3\cdots a_n}  {\cal A}_{\textbf{a}}. 
$$
In particular, we fix $a_{1}$ and sum over all other indices. We now define the sum-zero operator to be
$$
 \left({\cal Q}_t\circ {\cal A}\right)_{\textbf{a}}={\cal A}_{\textbf{a}}-\left({\cal P} \circ {\cal A}\right)_{a_1}
 \cdot {  \vartheta}_{t,\,\textbf{a}},\quad 
 {  \vartheta}_{t, \textbf{a}}:=\big(1-t\big)^{n-1}\prod_{i=2}^n \left(\Theta^{(B)}_t\right)_{a_1a_i}.$$
 Finally, we say that a tensor has the sum-zero property if ${\cal P}\circ{\cal A}=0$.
\end{definition}
By the random walk representation of $\Theta_{\xi}^{(B)}$ from Lemma \ref{lem_propTH}, the row sums of $\Theta^{(B)}$ are equal to $(1-\xi)^{-1}$. This implies ${\cal P}\circ\vartheta_{t,\textbf{a}}=1$; thus, ${\cal P}\circ{\cal Q}=0$. 

The following lemma collects important facts about ${\cal Q}_{t}$; it is Lemma 5.13 in \cite{YY_25}, and it relies on rapid decay of $\vartheta_{t,\textbf{a}}$.

\begin{lemma}\label{lem_+Q} 
If ${\cal A}:(\Z_{L}^{2})^{n}\to\R$ has ($t, \tau, D$) decay, then ${\cal Q}_{t}\circ{\cal A}:(\Z_{L}^{2})^{n}\to\R$ has $(t,\tau,D)$ decay. Moreover, for some $C_{n}=O(1)$, we have  
\begin{align}\label{normQA}
        \max_{\textbf{a}} \left({\cal Q}_t\circ \cal A\right)_{\textbf{a}} \le W^{C_n \tau} \cdot \max_{\textbf{a}} {\cal A}_{\textbf{a}}+W^{-D+C_n}.
\end{align}
\end{lemma}
We now introduce a version of the loop hierarchy obtained after applying the sum-zero operator. First, define $\dot{\vartheta}_{t,\textbf{a}}=\frac{d}{dt}\vartheta_{t,\textbf{a}}$. We have the following equation, which follows from (5.90) in \cite{YY_25} (the proof is identical, since we only use that $\varTheta_{t, \boldsymbol{\sigma}}$  from \eqref{DefTHUST}  preserves the sum-zero property, which itself follows because the row and column sums of $\Theta^{(B)}_{\xi}$ are equal to $(1-\xi)^{-1}$):
 \begin{align}\label{int_K-L+Q}
  {\cal Q}_t\circ   (\mathcal{L} - \mathcal{K})_{t, \boldsymbol{\sigma}, \textbf{a}} \;=\; \; &
    \left(\mathcal{U}_{s, t, \boldsymbol{\sigma}} \circ  {\cal Q}_s\circ (\mathcal{L} - \mathcal{K})_{s, \boldsymbol{\sigma}}\right)_{\textbf{a}} \nonumber \\
    &+ \sum_{l_\mathcal{K} > 2} \int_{s}^t \left(\mathcal{U}_{u, t, \boldsymbol{\sigma}} \circ  {\cal Q}_u\circ \Big[\mathcal{K} \sim (\mathcal{L} - \mathcal{K})\Big]^{l_\mathcal{K}}_{u, \boldsymbol{\sigma}}\right)_{\textbf{a}} du \nonumber \\
    &+ \int_{s}^t \left(\mathcal{U}_{u, t, \boldsymbol{\sigma}} \circ  {\cal Q}_u\circ \mathcal{E}^{((\mathcal{L} - \mathcal{K}) \times (\mathcal{L} - \mathcal{K}))}_{u, \boldsymbol{\sigma}}\right)_{\textbf{a}} du \nonumber \\
    &+ \int_{s}^t \left(\mathcal{U}_{u, t, \boldsymbol{\sigma}} \circ  {\cal Q}_u\circ \mathcal{E}^{(\widetilde{G})}_{u, \boldsymbol{\sigma}}\right)_{\textbf{a}} du + \int_{s}^t \left(\mathcal{U}_{u, t, \boldsymbol{\sigma}} \circ  {\cal Q}_u\circ \mathcal{E}^{(M)}_{u, \boldsymbol{\sigma}}\right)_{\textbf{a}} du
    \nonumber \\
     &+ \int_{s}^t 
     \left(\mathcal{U}_{u, t, \boldsymbol{\sigma}} 
     \circ \left( \big[{\cal Q}_u , \varTheta_{u,\boldsymbol{\sigma}} \big]\circ (\mathcal{L} - \mathcal{K}) _{u, \boldsymbol{\sigma} }\right) \right)_{\textbf{a}} du
        \\\nonumber
    &+ \int_{s}^t 
     \left(\mathcal{U}_{u, t, \boldsymbol{\sigma}} 
     \circ \left[ {\cal P} \circ
      \left(\mathcal{L} - \mathcal{K}\right)_{u, \boldsymbol{\sigma}}
        \cdot \dot \vartheta_{u } \right]\right)_{\textbf{a}} du.
\end{align}
Above, $[{\cal Q}_{u},\varTheta_{u,\boldsymbol{\sigma}}]:={\cal Q}_{u}\circ\varTheta_{u,\boldsymbol{\sigma}}-\varTheta_{u,\boldsymbol{\sigma}}\circ{\cal Q}_{u}$ is the commutator for operators. We also record the following identity below, which is (5.88) in \cite{YY_25}:
\begin{align}\label{zjuii2}
{\cal Q}_t \circ \Big[\mathcal{K} \sim (\mathcal{L} - \mathcal{K})\Big]^{l_\mathcal{K} = 2}_{t, \boldsymbol{\sigma}, \textbf{a}}
& = {\cal Q}_t \circ \varTheta_{t, \boldsymbol{\sigma}} \circ (\mathcal{L} - \mathcal{K})_{t, \boldsymbol{\sigma}, \textbf{a}}
\\\nonumber
& = \varTheta_{t, \boldsymbol{\sigma}} \circ {\cal Q}_t \circ (\mathcal{L} - \mathcal{K})_{t, \boldsymbol{\sigma}, \textbf{a}}
+ \big[{\cal Q}_t , \varTheta_{t, \boldsymbol{\sigma}} \big] \circ (\mathcal{L} - \mathcal{K})_{t, \boldsymbol{\sigma}, \textbf{a}},
\end{align}
Not only does ${\cal Q}_{t}\circ{\cal A}$ satisfy the sum zero property, but so does $\Theta_{t,\boldsymbol{\sigma}}\circ{\cal Q}_{t}$ by definition of $\varTheta_{t, \boldsymbol{\sigma}}$  from \eqref{DefTHUST} and because the row and column sums of $\Theta_{\xi}^{(B)}$ are $1/(1-\xi)$. Thus, the last term in \eqref{zjuii2} is also sum-zero:
\begin{align}\label{pqthlk}
   {\cal P}\circ \big[{\cal Q}_t , \varTheta_{t, \boldsymbol{\sigma}} \big] \circ (\mathcal{L} - \mathcal{K})_{t, \boldsymbol{\sigma}, \textbf{a}}=0.
\end{align}
Now, note that if ${\cal A}$ is sum zero, then ${\cal A}={\cal Q}_{t}\circ{\cal A}$. Using this and \eqref{pqthlk} gives
$$ \int_{s}^t 
     \left(\mathcal{U}_{u, t, \boldsymbol{\sigma}} 
     \circ \left( \big[{\cal Q}_u , \varTheta_{u,\boldsymbol{\sigma}} \big]\circ (\mathcal{L} - \mathcal{K}) _{u, \boldsymbol{\sigma} }\right) \right)_{\textbf{a}} du=\int_{s}^t 
     \left(\mathcal{U}_{u, t, \boldsymbol{\sigma}} 
     \circ {\cal Q}_u\circ \left( \big[{\cal Q}_u , \varTheta_{u,\boldsymbol{\sigma}} \big]\circ (\mathcal{L} - \mathcal{K}) _{u, \boldsymbol{\sigma} }\right) \right)_{\textbf{a}} du.
$$
Since ${\cal P}\vartheta=1$, differentiating this identity gives ${\cal P}\dot{\vartheta}=0$. In particular, we deduce
$$ \int_{s}^t 
     \left(\mathcal{U}_{u, t, \boldsymbol{\sigma}} 
     \circ \left[ {\cal P} \circ
      \left(\mathcal{L} - \mathcal{K}\right)_{u, \boldsymbol{\sigma}}
        \cdot \dot \vartheta_{u } \right]\right)_{\textbf{a}} du=\int_{s}^t 
     \left(\mathcal{U}_{u, t, \boldsymbol{\sigma}} 
     \circ {\cal Q}_u\circ \left[ {\cal P} \circ
      \left(\mathcal{L} - \mathcal{K}\right)_{u, \boldsymbol{\sigma}}
        \cdot \dot \vartheta_{u } \right]\right)_{\textbf{a}} du.
$$
Using the previous two identities, we can rewrite \eqref{int_K-L+Q} into a compact form:

\begin{align}\label{int_K-L+Q2}
  {\cal Q}_t\circ   (\mathcal{L} - \mathcal{K})_{t, \boldsymbol{\sigma}, \textbf{a}} \;=\; \; &
    \left(\mathcal{U}_{s, t, \boldsymbol{\sigma}} \circ  {\cal Q}_s\circ {\cal B}_0\right)_{\textbf{a}} \\\nonumber
    &+   \int_{s}^t \left(\mathcal{U}_{u, t, \boldsymbol{\sigma}} \circ  {\cal Q}_u\circ  \sum_{m=1}^5{\cal B}_m  \right)_a du
   \\\nonumber
    &+  \int_{s}^t \left(\mathcal{U}_{u, t, \boldsymbol{\sigma}} \circ  {\cal Q}_u\circ \mathcal{E}^{(M)}_{u, \boldsymbol{\sigma}}\right)_{\textbf{a}} du
\end{align}
where the integrands are defined as follows:
\begin{align}
{\cal B}_0\;&\;:=\;(\mathcal{L} - \mathcal{K})_{s, \boldsymbol{\sigma}}
   \\[5pt]\nonumber {\cal B}_1\;&\;:=\;\sum_{l_\mathcal{K}>2}\Big[\mathcal{K} \sim (\mathcal{L} - \mathcal{K})\Big]^{l_\mathcal{K}}_{u, \boldsymbol{\sigma}}\\[5pt]\nonumber
    {\cal B}_2\;&\;:=\mathcal{E}^{((\mathcal{L} - \mathcal{K}) \times (\mathcal{L} - \mathcal{K}))}_{u, \boldsymbol{\sigma}}\;\\[5pt]\nonumber
    {\cal B}_3\;&\;:=\mathcal{E}^{(\widetilde{G})}_{u, \boldsymbol{\sigma}}\;\\[5pt]\nonumber
    {\cal B}_4\;&\;:= \big[{\cal Q}_u , \varTheta_{u,\boldsymbol{\sigma}} \big]\circ (\mathcal{L} - \mathcal{K}) _{u, \boldsymbol{\sigma} }\;\\[5pt]\nonumber
    {\cal B}_5\;&\;:= {\cal P} \circ
      \left(\mathcal{L} - \mathcal{K}\right)_{u, \boldsymbol{\sigma}}
        \cdot \dot \vartheta_{u } \; 
\end{align}
Before we analyze the integrals of ${\cal B}_{m}$, we provide estimates for these terms. For the rest of this subsection, we fix the length $n$ and adopt the following assumption, in which $\Lambda\geq1$ constants are deterministic:
 \begin{align}
  \Xi^{({\cal L})}_{u, { m }}\prec \Lambda^{({\cal L})}_{u, { m }},\quad  \Xi^{({\cal L-\cal K})}_{u, { m}}\prec \Lambda^{({\cal L-\cal K})}_{u, {m}},\quad \forall m\le 2n+2 \label{STOeq_Qt_assume_0}
 \end{align}
First, ${\cal B}_0$ is bounded using the assumption \eqref{Eq:L-KGt+IND}. For $1\le m\le 3$, the ${\cal B}_{m}$ are bounded using Lemma  \ref{lem_BcalE}. Next, we have
\begin{align}\label{jfasiuu}
\|{\cal B}_5\|_{\max}  \prec 
          \max_{a} \left|\left[{\cal P} \circ
      \left(\mathcal{L} - \mathcal{K}\right)_{u, \boldsymbol{\sigma}}\right]_a\right| \cdot   \max_{\textbf{a}} \left| 
          \dot \vartheta_{u,\textbf{a} } \right|
\end{align}  
Recall that $\boldsymbol{\sigma}$ is alternating. In this case, we claim the following holds for any loop of length $n$:
\begin{align}\label{jywiiwsoks}
    \left[{\cal P} \circ
    \left(\mathcal{L} - \mathcal{K}\right)_{u, \boldsymbol{\sigma}}\right]_a \prec (W^2\eta_u)^{-n}\cdot  \ell_u^{2(n-2)} \cdot \ell_{u}^{-2(n-1)} \cdot \Xi^{(\mathcal{L}-\mathcal{K})}_{u, n-1}.
\end{align}
To see this, since $\boldsymbol{\sigma}$ is non-constant, there exists $1 < k \leq n$ such that $\sigma_k = \overline{\sigma}_{k+1}$. By Lemma \ref{lem_WI_K}, we can sum $\mathcal{L} - \mathcal{K}$ of rank $n$ over $a_k$ and write the resulting sum in terms of $\mathcal{L} - \mathcal{K}$ of rank $n-1$ with a multiplicative factor $(W^{2} \eta_u)^{-1}$. The summation over the remaining $n-2$ many indices results in a factor of $\prec\ell_{u}^{2(n-2)}$ because ${\cal L}-{\cal K}$ has the $(u,\tau,D)$ decay property for any $\tau,D>0$ by assumption \eqref{Eq:Gdecay_w}. Finally, the additional $(W^{2}\ell_{u}^{2}\eta_{u})^{-n+1}=M_{u}^{-n+1}$ comes from compensating the factor of $M_{u}^{n-1}$ in the definition of $\Xi^{({\cal L}-{\cal K})}_{u,n-1}$. Next, 
    \begin{align}
   \max_{\textbf{a}}\left| 
       \dot \vartheta_{u,\textbf{a} } \right|\prec \ell_u^{-2n+2}\cdot \eta_u^{-1}  \label{eq:thetadot_bound}
   \end{align}
by definition of $\vartheta$ and the estimate \eqref{prop:ThfadC} for $\Theta^{(B)}$. Inserting these two bounds into \eqref{jfasiuu} and using \eqref{STOeq_Qt_assume_0} gives
\begin{align}\label{kkuuwsaf}
  \|{\cal B}_5\|_{\max}  \prec (M_u)^{-n}\cdot \eta_u^{-1}\cdot   \Lambda^{(\cal L-K)}_{u,\; n-1}
 \end{align}
By a similar argument, for ${\cal B}_4$, we have
 \begin{align}\label{kkuuwsaf5}
  \|{\cal B}_4\|_{\max}
   \;=\; & \; \left[ \left({\cal P} \circ \left(\varTheta_{u, \boldsymbol{\sigma}} \circ (\mathcal{L} - \mathcal{K})\right)\right) \cdot \vartheta_u \right]_{\textbf{a}} 
    - \varTheta_{u, \boldsymbol{\sigma}} \circ \left|\left[ {\cal P} \circ \left( \mathcal{L} - \mathcal{K} \right) \cdot \vartheta_t \right]_{\textbf{a}}\right|.
\\\nonumber
 \;\prec \; & \;  \|\vartheta_u\|_{\max}\cdot \eta_u^{-1}\max_a\left[ {\cal P} \circ \left(  \mathcal{L} - \mathcal{K} \right)_{u,\boldsymbol{\sigma}} \right]_{a} 
  \\\nonumber
 \;\prec \; & \;(M_u)^{-n}\cdot \eta_u^{-1}\cdot   \Lambda^{(\cal L-K)}_{u,\; n-1}
\end{align} 

In \cite{YY_25}, the sum zero operator and the previous bounds on ${\cal B}_{m}$ terms are enough to analyze the loop hierarchy. For dimension $2$, we need another ``mean zero operator" $\mathbb Q:=1-\mathbb E$. We now decompose \eqref{int_K-L+Q2} according to $\mathbb E$ and $\mathbb Q$:
\begin{align}\label{int_K-L+QQ}
  {\cal Q}_t\circ \mathbb Q\circ  (\mathcal{L} - \mathcal{K})_{t, \boldsymbol{\sigma}, \textbf{a}} \;=\; \; &
    \left(\mathcal{U}_{s, t, \boldsymbol{\sigma}} \circ  {\cal Q}_s\circ \mathbb Q\circ  {\cal B}_0\right)_{\textbf{a}} \\\nonumber
    &+   \int_{s}^t \left(\mathcal{U}_{u, t, \boldsymbol{\sigma}} \circ  {\cal Q}_u\circ \mathbb Q\circ  \sum_{m=1}^5{\cal B}_m  \right)_a du
   \\\nonumber
    &+  \int_{s}^t \left(\mathcal{U}_{u, t, \boldsymbol{\sigma}} \circ  {\cal Q}_u\circ \mathcal{E}^{(M)}_{u, \boldsymbol{\sigma}}\right)_{\textbf{a}} du
\end{align}
\begin{align}\label{int_K-L+QE}
  {\cal Q}_t\circ \mathbb E\circ  (\mathcal{L} - \mathcal{K})_{t, \boldsymbol{\sigma}, \textbf{a}} \;=\; \; &
    \left(\mathcal{U}_{s, t, \boldsymbol{\sigma}} \circ  {\cal Q}_s\circ \mathbb E\circ  {\cal B}_0\right)_{\textbf{a}} \\\nonumber
    &+   \int_{s}^t \left(\mathcal{U}_{u, t, \boldsymbol{\sigma}} \circ  {\cal Q}_u\circ \mathbb E\circ  \sum_{m=1}^5{\cal B}_m  \right)_a du
\end{align}
By definition, each of ${\cal Q}_{s}\circ {\cal B}_{0}$ and ${\cal Q}_{u}\circ {\cal B}_{m}$ for $m=1,\ldots,5$ has fast decay, the sum-zero property, and the form \eqref{a-local-form}. Thus, since we have the mean-zero operator $\mathbb{Q}$, we can use the enhanced ${\cal U}_{u,t,\boldsymbol{\sigma}}$ estimate from case 3 of Lemma \ref{lem:sum_decay}. In the case of ${\cal B}_{0}$, we have $\|{\cal B}_{0}\|_{\max}\prec M_{s}^{-n}$, and thus
$$
M_{t}^{n}\cdot  \left(\mathcal{U}_{s, t, \boldsymbol{\sigma}} \circ  {\cal Q}_s\circ \mathbb Q\circ  {\cal B}_0\right)_{\textbf{a}}\prec M_{t}^{n}\cdot\left(\frac{\ell_{s}^{2}\eta_{s}}{\ell_{t}^{2}\eta_{s}}\right)^{n}M_{s}^{-n}+W^{-D+C_{n}}\prec \Lambda^{(\cal L-\cal K)}_{s,n}+W^{-D+C_{n}}
 $$
for $D>0$ large and $C_{n}=O(1)$, where the last bound follows since $\Lambda\geq1$. If we let $\Lambda_{m}\geq1$ be deterministic numbers such that $\|{\cal B}_{m}\|_{\max}\prec \Lambda_{m}$, then we also get
 $$ M_{t}^{n}\cdot \mathcal{U}_{u, t, \boldsymbol{\sigma}} \circ  {\cal Q}_u\circ \mathbb Q\circ   {\cal B}_m  \prec M_{t}^{n}\left(\frac{\ell_{u}^{2}\eta_{u}}{\ell_{t}^{2}\eta_{t}}\right)^n \Lambda_{m}+W^{-D+C_{n}}\prec M_{u}^{n}\|{\cal B}_{m}\|_{\max}+W^{-D+C_{n}}.
$$
For clarity of presentation, we will not keep track of the $W^{-D+C_{n}}$ errors; they will not affect the argument below in any case. For $1\le m\le 3$, the ${\cal B}_{m}$ terms are bounded in Lemma  \ref{lem_BcalE}, whereas ${\cal B}_4$, ${\cal B}_5$ are bounded in \eqref{kkuuwsaf} and \eqref{kkuuwsaf5}. Ultimately, we obtain that
\begin{align}\label{juaspuwp}
   & \left(\mathcal{U}_{s, t, \boldsymbol{\sigma}} \circ  {\cal Q}_s\circ \mathbb Q\circ  {\cal B}_0\right)_{\textbf{a}} +  \int_{s}^t \left(\mathcal{U}_{u, t, \boldsymbol{\sigma}} \circ  {\cal Q}_u\circ \mathbb Q\circ  \sum_{m=1}^5{\cal B}_m  \right)_a du
    \\\nonumber
    &\; \prec \;(M_t)^{-n}\max_{u\in[s,t]}\left(\; \max_{k<n }\Lambda^{({\cal L-\cal K})}_{u,k}+\max_{k:\;2\le k\le n }\Lambda^{({\cal L-\cal K})}_{u,k}\cdot \Lambda^{({\cal L-\cal K})}_{u,n-k+2}\cdot M_u^{-1}+\Lambda^{({\cal L})}_{u,n+1}\right)
\end{align}
We now analyze the terms in \eqref{int_K-L+QE}. First, we note that each of ${\cal Q}_s\circ \mathbb E\circ   {\cal B}_0$ and ${\cal Q}_u\circ \mathbb E\circ   {\cal B}_m$ for $m=1,\ldots,5$ has fast decay, sum zero and symmetry, in the sense that we have
 \begin{align}
{\cal P} \left( {\cal Q} \circ \mathbb E\circ   {\cal B} \right)=0,\quad {\cal R}\circ  \left( {\cal Q} \circ \mathbb E\circ   {\cal B} \right)= {\cal Q} \circ \mathbb E\circ   {\cal B} \label{eq:case4_B}
 \end{align}
where $$\left( {\cal R}\circ {\cal A}\right)_{a,a+s_2,\ldots a+s_n} : = {\cal A}_{a,a-s_{2}\cdots a-s_{n}}$$ for any $a, s_2,\ldots, s_n\in \mathbb{Z}_L^2$. For such tensors, we have another enhanced ${\cal U}_{u,t,\boldsymbol{\sigma}}$ estimate from case 4 of Lemma \ref{lem:sum_decay}. This gives
 $$ M_{t}^{n}\cdot \mathcal{U}_{u, t, \boldsymbol{\sigma}} \circ  {\cal Q}_u\circ \mathbb E\circ   {\cal B}_m  \prec M_{u}^n \left\|\mathbb E \,{\cal B}_m\right\|_{\max} 
$$
Similar to \eqref{juaspuwp},  With estimates for ${\cal B}_{m}$ from \eqref{Eq:L-KGt+IND}), Lemma  \ref{lem_BcalE}, \eqref{kkuuwsaf}, and \eqref{kkuuwsaf5},  we deduce
\begin{align} \label{juaspuwp234}
   & \left(\mathcal{U}_{s, t, \boldsymbol{\sigma}} \circ  {\cal Q}_s\circ \mathbb E\circ  {\cal B}_0\right)_{\textbf{a}} +  \int_{s}^t \left(\mathcal{U}_{u, t, \boldsymbol{\sigma}} \circ  {\cal Q}_u\circ \mathbb E\circ  \sum_{m=1}^5{\cal B}_m  \right)_a du
    \\\nonumber
    &\; \prec \;(M_t)^{-n}\max_{u\in[s,t]}\left(\; \max_{k<n }\mathbb E \,\Xi^{({\cal L-\cal K})}_{u,k}+\max_{k:\;2\le k\le n }\mathbb E \, \Xi^{({\cal L-\cal K})}_{u,k}\cdot \Xi^{({\cal L-\cal K})}_{u,n-k+2}\cdot M_u^{-1}+\mathbb E \,\Xi^{({\cal L})}_{u,n+1}\right)
\end{align}
We are left to estimate the martingale in \eqref{int_K-L+QQ}. First, we introduce the notation
\begin{align}
\left(\left( \mathcal{Q} _u\otimes \mathcal{Q}_u \right)\circ {\cal A}\right)_{\textbf{a}, \textbf{b}} 
 \;=\; & \; 
{\cal A }_{\textbf{a}, \textbf{b} }-
\delta_{a_1', a_1}\sum_{a_2'\cdots a_n'}{\cal A }_{ \textbf{a}' , \,\textbf{b}}\cdot \vartheta_{u, \textbf{a}}
-\delta_{b_1', b_1}
\sum_{b_2'\cdots b_n'}{\cal A }_{\textbf{a}, \textbf{b}'}\cdot \vartheta_{u, \textbf{b} }
\nonumber \\ \;+\; & \; 
\delta_{a_1', a_1}
\delta_{b_1', b_1}
\sum_{a'_2\cdots a'_n}\sum_{b_2'\cdots b_n'}{\cal A }_{\textbf{a}', \textbf{b}'}\cdot \vartheta_{u, \textbf{a}}\cdot \vartheta_{u, \textbf{b}}  , 
\end{align}
where the used the notation
$
\textbf{a}=(a_1,\cdots, a_n), \textbf{a}'=(a'_1,\cdots, a'_n)$ and similarly for 
 $\textbf{b}$ and $\textbf{b}'$. The definition of  $\vartheta$ and an elementary calculation gives 
$$
\sum_{a_2,\cdots,a_n} \left(\left( \mathcal{Q} _u\otimes \mathcal{Q}_u \right)\circ {\cal A}\right)_{\textbf{a}, \textbf{b}}=\sum_{b_2,\cdots,b_n} \left(\left( \mathcal{Q} _u\otimes \mathcal{Q}_u \right)\circ {\cal A}\right)_{\textbf{a}, \textbf{b}}=0.
$$
By case 5 in Lemma \ref{lem:sum_decay}, we have the following for any tensor ${\cal A}$ of rank-$2n$ that has fast decay at scale $\ell_{u}$: 
$$
\left(\left(\mathcal{U}_{u, t, \boldsymbol{\sigma}} \otimes \mathcal{U}_{u, t, \bar{\sigma}}\right) \circ\left(\mathcal{Q}_u \otimes \mathcal{Q}_u\right) \circ {\cal A}\right)_{\mathrm{a}, \mathrm{a}} \prec \|{\cal A}\|_{\max}\cdot \left(\frac{\ell^2_u\eta_u}{\ell^2_t\eta_t}\right)^{2n}.
$$
By this estimate applied to ${\cal E}\otimes{\cal E}$, we have
$$
\left(\left(\mathcal{U}_{u, t, \boldsymbol{\sigma}} \otimes \mathcal{U}_{u, t, \bar{\sigma}}\right) \circ\left(\mathcal{Q}_u \otimes \mathcal{Q}_u\right) \circ(\mathcal{E} \otimes \mathcal{E})_{u, \sigma}\right)_{\mathrm{a}, \mathrm{a}} \prec \|(\mathcal{E} \otimes \mathcal{E})\|_{\max} \cdot \left(\frac{\ell^2_u\eta_u}{\ell^2_t\eta_t}\right)^{2n}
$$
Next, by \eqref{alu9_STime}, if for some deterministic quantity $\Lambda^{(L)}_{u, 2n+2}$, we have 
$ \Xi_{u, 2 n+2}^{(\mathcal{L})} \prec\Lambda^{(L)}_{u, 2n+2}$, then
\begin{equation}\label{uwyu92osk}
    \int_{s}^t \left(\mathcal{U}_{u, t, \boldsymbol{\sigma}} \circ  {\cal Q}_u\circ \mathcal{E}^{(M)}_{u, \boldsymbol{\sigma}}\right)_{\textbf{a}}\prec \max_{u\in [s,t]} \left(\Lambda^{(L)}_{u, 2n+2}\right)^{1/2}.
\end{equation}
 We now combine the estimates in this subsection to prove the following.

\begin{lemma}\label{lem:STOeq_Qt} 
Suppose the assumptions of 
 Theorem \ref{lem:main_ind}, local law \eqref{Gt_bound_flow} and decay of $\cal L-\cal K$ in \eqref{Eq:Gdecay_w} all hold. 
Assume that for some  deterministic quantities $\Lambda\geq1$, we have 
 \begin{align}
  \Xi^{({\cal L})}_{u, { m }}\prec \Lambda^{({\cal L})}_{u, { m }},\quad  \Xi^{({\cal L-\cal K})}_{u, { m}}\prec \Lambda^{({\cal L-\cal K})}_{u, {m}},\quad \forall m\le 2n+2. \label{STOeq_Qt_assume}
 \end{align}
Then, we have 
\begin{align}\label{am;asoi222}
 \Xi^{({\cal L-\cal K})}_{t,n}
 \;\prec\;  &\;   \max_{u\in [s,t]}
  \left(\; \max_{k<n }\Lambda^{({\cal L-\cal K})}_{u,k}+\max_{2\le k\le n }\Lambda^{({\cal L-\cal K})}_{u,k}\cdot \Lambda^{({\cal L-\cal K})}_{u,n-k+2}\cdot M_{u}^{-1}+\Lambda^{({\cal L})}_{u,n+1}+\left(\Lambda^{(\cal L)}_{u,2n+2}\right)^{1/2}\right)\\
  &+ \max_{u\in[s,t]}\left(\; \max_{k<n }\mathbb E \,\Xi^{({\cal L-\cal K})}_{u,k}+\max_{k:\;2\le k\le n }\mathbb E \, \Xi^{({\cal L-\cal K})}_{u,k}\cdot \Xi^{({\cal L-\cal K})}_{u,n-k+2}\cdot M_u^{-1}+\mathbb E \,\Xi^{({\cal L})}_{u,n+1}\right).\nonumber
\end{align}
\end{lemma}

\begin{proof}
Lemma \ref{lem:STOeq_NQ} treats \eqref{am;asoi222} for non-alternating case $\boldsymbol{\sigma}$, so we assume $\boldsymbol{\sigma}$ is alternating, so that $\sigma_{k}=\overline{\sigma}_{k+1}$ for all $k$. We combine \eqref{int_K-L+QQ}, \eqref{int_K-L+QE}, \eqref{juaspuwp}, \eqref{juaspuwp234} and \eqref{uwyu92osk},  to get

\begin{align}
    {\cal Q}_t\circ   (\mathcal{L} - \mathcal{K})_{t, \boldsymbol{\sigma}, \textbf{a}} &\prec M^{-n}_t\cdot \max_{u\in[s,t]} \left(\; \max_{k<n }\Lambda^{({\cal L-\cal K})}_{u,k}+\max_{2\le k\le n }\Lambda^{({\cal L-\cal K})}_{u,k}\cdot \Lambda^{({\cal L-\cal K})}_{u,n-k+2}\cdot M_{u}^{-1}+\Lambda^{({\cal L})}_{u,n+1}+\left(\Lambda^{\cal L}_{u,2n+2}\right)^{1/2}\right)\nonumber\\
    &+M_{t}^{-n} \cdot \max_{u\in[s,t]}\left(\; \max_{k<n }\mathbb E \,\Xi^{({\cal L-\cal K})}_{u,k}+\max_{k:\;2\le k\le n }\mathbb E \, \Xi^{({\cal L-\cal K})}_{u,k}\cdot \Xi^{({\cal L-\cal K})}_{u,n-k+2}\cdot M_u^{-1}+\mathbb E \,\Xi^{({\cal L})}_{u,n+1}\right).\nonumber
\end{align}
We now write the left-hand side of the previous display as follows:
\begin{align}\label{kolkisaf} 
{\cal Q}_t\circ   (\mathcal{L} - \mathcal{K})_{t, \boldsymbol{\sigma}, \textbf{a}}
 =(\mathcal{L} - \mathcal{K})_{t, \boldsymbol{\sigma}, \textbf{a}}-{\cal P} \circ   (\mathcal{L} - \mathcal{K})_{t, \boldsymbol{\sigma}, \textbf{a}}\cdot \vartheta_{\textbf{a}}
 =(\mathcal{L} - \mathcal{K})_{t, \boldsymbol{\sigma}, \textbf{a}}
 +\mathrm{O}_{\prec} \left(M_t^{-n} \cdot \Lambda^{(\cal L-K)}_{t,\; n-1}\right)
\end{align}
This completes the proof of Lemma \ref{lem:STOeq_Qt}. 
\end{proof}

 \subsection{Proof of Theorem \ref{lem:main_ind}, Step 3}
The argument is the exact same as the proof of Step 3 in \cite{YY_25}, since it uses only \eqref{am;asoi222}. Thus, let us give a summary of the argument. For the sake of stochastic domination, we can drop the second line of \eqref{am;asoi222}. We will proceed inductively in $n$. In particular, we can trade the first two terms on the right-hand side of \eqref{am;asoi222} for $1$, assuming we proven the main estimate \eqref{Eq:LGxb} of Step 3 for $k<n$. This gives us the following, in which we think of $\Lambda\geq1$ constants as our ``best bounds available" for $\Xi$ parameters: 
\begin{align*}
\Lambda_{t,n}^{({\cal L}-{\cal K})}\prec1+\max_{u\in[s,t]}\Big(\Lambda^{({\cal L})}_{u,2n+2}\Big)^{\frac12}+\max_{u\in[s,t]}\Lambda_{u,n+1}^{({\cal L})}.
\end{align*}
Because the scaling in $\Xi^{({\cal L}-{\cal K})}$ is smaller than the scaling in $\Xi^{({\cal L})}$ by $M_{u}^{-1}$, we have
\begin{align*}
\Lambda_{u,n+1}^{({\cal L})}\prec1+M_{u}^{-1}\Lambda_{u,n+1}^{({\cal L}-{\cal K})}.
\end{align*}
Combining the previous two displays gives
\begin{align*}
\Lambda^{({\cal L}-{\cal K})}_{t,n}\prec1+\max_{u\in[s,t]}\Big(\Lambda^{({\cal L})}_{u,2n+2}\Big)^{\frac12}+M_{u}^{-1}\max_{u\in[s,t]}\Lambda^{({\cal L}-{\cal K})}_{u,n+1}.
\end{align*}
At this point, we can formally continue by plugging the previous estimate (but replacing $n$ by $n+1$) for the last term on the right-hand side. Because $M_{u}^{-1}\leq W^{-\delta}$ is small, this type of argument closes. We also need control on the second term on the right-hand side. However, as explained in the proof of Step 3 in \cite{YY_25}, the a priori upper bound $\Lambda^{({\cal L})}$ that is consistent with this bootstrapping is 
\begin{align*}
\max_{u\in[s,t]}\Big(\Lambda^{({\cal L})}_{u,2n+2}\Big)\prec M_{u}^{\frac12}+M_{u}^{-\frac12}\max_{u\in[s,t]}\Lambda_{u,n+1}^{({\cal L}-{\cal K})}.
\end{align*}
Thus, we ultimately get 
\begin{align*}
\Lambda^{({\cal L}-{\cal K})}_{t,n}\prec M_{t}^{\frac12},
\end{align*}
which is enough to obtain $\Lambda^{({\cal L})}_{t,n}\prec1$, i.e. the desired estimate, because of our ${\cal K}$ estimate in Lemma \ref{ML:Kbound}.

\subsection{Proof of  Theorem \ref{lem:main_ind}: Step 4 and 5.}

For step 4, the argument is the same as the proof of step 4 in \cite{YY_25}, so we give a summary. We use step $3$ to control all $\Xi^{({\cal L})}$ terms in \eqref{am;asoi222} by $\mathrm{O}_{\prec}(1)$. Thus, we arrive at
\begin{align}
 \Lambda^{({\cal L-\cal K})}_{t,n}
 \;\prec\;  &\;  1+ \max_{u\in [s,t]}
  \left(\; \max_{k<n }\Lambda^{({\cal L-\cal K})}_{u,k}+\max_{2\le k\le n }\Xi^{({\cal L-\cal K})}_{u,k}\cdot \Lambda^{({\cal L-\cal K})}_{u,n-k+2}\cdot M_{u}^{-1}\right)\nonumber\\
  &\prec\; 1+ \max_{u\in [s,t]}
  \left(\; \max_{k<n }\Lambda^{({\cal L-\cal K})}_{u,k}+\max_{2\le k\le n }\Lambda^{({\cal L-\cal K})}_{u,k}\cdot M_{u}^{-\frac12}\right),\label{step4_main}
\end{align}
where the second line follows from 
$$\Lambda^{({\cal L}-{\cal K})}_{u,n-k+2}\prec M_{u}^{1/2}+M_{u}^{1/2}\Lambda^{({\cal L})}_{u,n-k+2}\prec M_{u}^{1/2}.$$
Now, we use \eqref{GavLGEX} to get $\Lambda^{({\cal L}-{\cal K})}_{t,1}\prec1$. At this point, because $M_{u}\geq W^{\delta}$ for some $\delta>0$, we can induct in $n$ using \eqref{step4_main} to deduce $\Lambda^{({\cal L}-{\cal K})}_{t,n}\prec1$. This is the desired estimate for step 4.

For step 5, we use step 4 to get $({\cal L}-{\cal K})_{t,\, \boldsymbol{\sigma},  \,\textbf{a}}\prec M_{t}^{-2} $ in the case of $n=2$. Since ${\cal T}_{t,D}(\ell)$ is bounded away from $0$ for $\ell=O(\ell_{t}^{*})$, this implies that $({\cal L}-{\cal K})_{t,\, \boldsymbol{\sigma},  \,\textbf{a}}/{\cal T}_{t, D}(|a_1-a_2|_{L}) \prec1$ in the case where $|a_{1}-a_{2}|_{L}=O(\ell_{t}^{*})$. This is the desired estimate. So, it suffices to assume that $|a_{1}-a_{2}|_{L}>6\ell_{t}^{*}$. In this case, the desired estimate \eqref{Eq:Gdecay_flow} follows immediately from \eqref{53}. This completes the proofs of step 4 and 5. \qed

 \subsection{Proof of Lemma \ref{ML:exp}: Step 6}\label{subsection:step6}

 We recall that Lemma \ref{ML:exp} asks for an improved estimate on the expectation of ${\cal L}$ of length $n=2$. The first ingredient is the following improved estimate on loops of length $1$, which is Lemma 5.15 in \cite{YY_25}.
 \begin{lemma}\label{lemma:step6-1}
Fix any $a\in\Z_{L}^{2}$, any $\tau,\kappa>0$, any $E\in[-2+\kappa,2-\kappa]$, and $0\leq u\leq 1-N^{-1+\tau}$. We have 
\begin{align}
|\E{\cal L}_{u,(+),(a)}-{\cal K}_{u,(+),(a)}|=|\E\langle(G_{u}-m)E_{a}\rangle|\prec M_{u}^{-2}.
\end{align}
\end{lemma}
\begin{proof}
We start with the identity $G_{u}-m=-m(H_{u}+m)G_{u}$. This gives us the first line of the following calculation, after which we use Gaussian integration by parts (note that the variance of the entries of $H_{u}$ are all scaled by $u$):
\begin{align*}
\E\langle(G_{u}-m)E_{a}\rangle&=m\cdot\E\langle(-H-m)G_{u}E_{a}\rangle\\
&=u\cdot m\sum_{b}\E\left(\langle(G_{u}-m)E_{b}\rangle W^{2}\cdot S^{(B)}_{bb'}\cdot \langle E_{b'}G_{u}E_{a}\rangle\right)\\
&=u\cdot m\sum_{b}\E\left(\langle(G_{u}-m)E_{b}\rangle S^{(B)}_{ba}\langle GE_{a}\rangle\right).
\end{align*}
(In the above calculation, we use the structure of $S^{(B)}_{ba}:=\frac15\mathbf{1}(|b-a|_{L}\leq1)$ and of $E_{b'}$ from \eqref{Def_matE}.) At this point, we decompose $\langle G_{u}E_{a}\rangle=m+\langle(G_{u}-m)E_{a}\rangle$ and plug it into the last line above. As in (5.124) of \cite{YY_25}, this gives us a linear equation for the vector $a\mapsto\E\langle(G_{u}-m)E_{a}\rangle$. Recalling that $\Theta^{(B)}_{um^{2}}:=(1-um^{2}S^{(B)})^{-1}$, we arrive at
\begin{align*}
\E\langle(G_{u}-m)E_{a}\rangle&=\sum_{a'}(\Theta^{(B)}_{um^{2}})_{aa'}\cdot\E\left(u\cdot m\sum_{b}\langle(G_{u}-m)E_{b}\rangle S^{(B)}_{ba'}\langle(G_{u}-m)E_{a'}\rangle\right).
\end{align*}
By \eqref{Eq:L-KGt}, we know that $\langle(G_{u}-m)E_{a}\rangle\prec M_{u}^{-1}$. Moreover, since $|1-um^{2}|$ is bounded uniformly away from $0$, the random walk representation in Lemma \ref{lem_propTH} gives $\sum_{a'}|(\Theta^{(B)}_{um^{2}})_{aa'}|=O(1)$. If we plug these two bounds into the right-hand side above, then the proof is complete.
\end{proof}
\begin{proof}[Proof of Lemma \ref{ML:exp}]

We will take expectation of the integrated loop hierarchy after applying the sum-zero operator. In particular, consider \eqref{int_K-L+Q2} with $s=0$ and $n=2$, and take the expectation of this identity. When we do this, the $m=1$ term in \eqref{int_K-L+Q2} vanishes because $n=2$, and the martingale term vanishes because we take expectation. Thus, we get
\begin{align}
{\cal Q}_{t}\circ \E({\cal L}-{\cal K})_{t,\boldsymbol{\sigma},\textbf{a}}&=\int_{0}^{t}\Big({\cal U}_{u,t,\boldsymbol{\sigma}} \circ {\cal Q}_{u} \circ \E {\cal E}^{(({\cal L}-{\cal K})\times({\cal L}-{\cal K}))}_{u,\boldsymbol{\sigma}}\Big)_{\textbf{a}}du\label{Eexpint_K-L}\\
&+\int_{0}^{t}\Big({\cal U}_{u,t,\boldsymbol{\sigma}} \circ {\cal Q}_{u} \circ \E {\cal E}^{(\tilde{G})}_{u,\boldsymbol{\sigma}}\Big)_{\textbf{a}}du\nonumber\\
&+\int_{0}^{t}\Big({\cal U}_{u,t,\boldsymbol{\sigma}} \circ {\cal Q}_{u} \circ \E\big[{\cal Q}_u , \varTheta_{u,\boldsymbol{\sigma}} \big]\circ (\mathcal{L} - \mathcal{K}) _{u, \boldsymbol{\sigma}}\Big)_{\textbf{a}}du\nonumber\\
&+\int_{0}^{t}\Big({\cal U}_{u,t,\boldsymbol{\sigma}} \circ {\cal Q}_{u} \circ \E\big[{\cal P} \circ
      \left(\mathcal{L} - \mathcal{K}\right)_{u, \boldsymbol{\sigma}}
        \cdot \dot \vartheta_{u }\big]\Big)_{\textbf{a}}du.\nonumber
\end{align}
(Note that there is no initial data term, since ${\cal L}-{\cal K})$ vanishes at time $0$.)

We can now control the expectation in the first term on the right-hand side, which is defined in \eqref{def_ELKLK}, as follows. First, in \eqref{def_ELKLK}, we use $(u,\tau,D)$ decay from Lemma \ref{lem_decayLoop} and the $O(1)$ range of $S^{(B)}$ to restrict the sum over $a,b$ on the right-hand side of \eqref{def_ELKLK} to a total of $W^{\tau}\ell_{u}^{2}$ many indices (for $\tau>0$ small). Then, we bound each loop on the right-hand side of \eqref{def_ELKLK} using \eqref{Eq:L-KGt}. This gives the following analogue of (5.129) in \cite{YY_25}:
\begin{align}
\E{\cal E}^{((L-K)\times(L-K))}_{u,\boldsymbol{\sigma}}\prec W^{2}\ell_{u}^{2}M_{u}^{-4}\prec \eta_{u}^{-1}M_{u}^{-3}.\label{eq:LKLKstep6}
\end{align}
(We recall that $M_{u}=W^{2}\ell_{u}^{2}\eta_{u}$.) For the second term on the right-hand side of \eqref{Eexpint_K-L}, we can control the expectation therein in the following way. Recall the definition of ${\cal E}^{(\tilde{G})}_{u,\boldsymbol{\sigma}}$ from \eqref{def_EwtG}. Each summand in \eqref{def_EwtG} is a product of a one-loop and a three-loop. In particular, we can use the same fast decay property from Lemma \ref{lem_decayLoop} to restrict the sum on the right-hand side of \eqref{def_EwtG} to $\mathrm{O}_{\prec}(\ell_{u}^{2})$-many indices. We get

\begin{align}
    \mathbb E \mathcal{E}^{(\widetilde{G})}_{u, \boldsymbol{\sigma},\textbf{a}}
   \; \prec \; & \;
W^{2}\ell_{u}^{2}\cdot  \max_a \max_{\textbf{a}}\max_{ \boldsymbol{\sigma}\in \{+,-\}^3} \Big|\mathbb E
\Big [ \langle (G_{u}-m) E_a\rangle \cdot {\cal L}_{u,\textbf{a}, \boldsymbol{\sigma} } \Big ] \Big|
 \nonumber \\
  \; \prec \; & \;
W^{2}\ell_{u}^{2}\cdot \max_a  \big|\mathbb E \langle (G_{u}-m) E_a\rangle \big|
 \cdot 
 \max_{\textbf{a}}\max_{ \boldsymbol{\sigma}\in \{+,-\}^3} |{\cal K}_{u,\textbf{a}, \boldsymbol{\sigma} }|
\nonumber \\
 + \; \;& W^{2}\ell_{u}^{2}\cdot \max_a \big|{({\cal L-\cal K})}_{u,(+), (a),   }\rangle \big|
 \cdot 
 \max_{\textbf{a}} \max_{ \boldsymbol{\sigma}\in \{+,-\}^3} |{({\cal L-\cal K})}_{u,\textbf{a}, \boldsymbol{\sigma} }|.
 \end{align}
 We now use Lemma \ref{lemma:step6-1} along with the estimates \eqref{eq:bcal_k} and \eqref{Eq:L-KGt} to obtain 
 \begin{align}
     \mathbb E \mathcal{E}^{(\widetilde{G})}_{u, \boldsymbol{\sigma},\textbf{a}} \prec \eta_{u}^{-1}M_{u}^{-3}.\label{eq:wtGstep6}
 \end{align}
Let us now move to the last term in \eqref{Eexpint_K-L}. Since the lengths of ${\cal L}$ and ${\cal K}$ are both equal to $n=2$, we can use Lemma \ref{lem_WI_K} (the Ward identity) to write
\begin{align}
\E[{\cal P} \circ ({\cal L}-{\cal K})_{u,\boldsymbol{\sigma}}]&=\frac{1}{2W^{2}\eta_{u}}\left(\E[{\cal L}_{u,(+)}-{\cal K}_{u,(+)}]-\E[{\cal L}_{u,(-)}-{\cal K}_{u,(-)}]\right) \prec \ell_{u}^{2}M_{u}^{-3},\label{eq:step6_improvedexpectation}
\end{align}
where the last bound follows from Lemma \ref{lemma:step6-1} and $M_{u}=W^{2}\ell_{u}^{2}\eta_{u}$. We now combine the above display with \eqref{eq:thetadot_bound} to deduce
\begin{align}
\|\E[{\cal P}\circ ({\cal L}-{\cal K})_{u,\boldsymbol{\sigma}} \cdot \dot{\vartheta}_{u}]\|_{\max}\prec \eta_{u}^{-1}M_{u}^{-3}.\label{eq:p_term_step6}
\end{align}
Finally, for the third term on the right-hand side of \eqref{Eexpint_K-L}, we follow \eqref{kkuuwsaf5}, except we use the improved expectation bound \eqref{eq:step6_improvedexpectation}. In particular, we have 
\begin{align}
\|\E\big[{\cal Q}_u , \varTheta_{u,\boldsymbol{\sigma}} \big]\circ (\mathcal{L} - \mathcal{K}) _{u, \boldsymbol{\sigma}}\|_{\max}&\prec \|\vartheta_{u}\|_{\max}\cdot \eta_{u}^{-1}\max_{a}|\E[{\cal P} \circ ({\cal L}-{\cal K})_{u,\boldsymbol{\sigma}}]_{a}|\\
&\prec \eta_{u}^{-1}M_{u}^{-3}.\label{commutator_step6}
\end{align}
We clarify that the last bound above uses also the bound $\|\vartheta_{u}\|_{\max}\prec \ell_{u}^{-2}$ for $n=1$ which can be deduced from Definition \ref{Def:QtPt} and Lemma \ref{lem_propTH}. Now, we recall \eqref{eq:case4_B}, and we recall from the sentence before \eqref{eq:case4_B} that each term which is acted on by ${\cal U}_{u,t,\boldsymbol{\sigma}}$ in \eqref{Eexpint_K-L} has $(u,\tau,D)$ decay for any $\tau,D>0$. This lets us apply the ${\cal U}_{u,t,\boldsymbol{\sigma}}|_{\max\to\max}$ operator norm from case 4 of Lemma \ref{lem:sum_decay}. If we combine this operator norm estimate with \eqref{Eexpint_K-L}, \eqref{eq:LKLKstep6}, \eqref{eq:wtGstep6}, \eqref{eq:p_term_step6}, and \eqref{commutator_step6}, we deduce the following (for any $D>0$ fixed):
\begin{align*}
\|{\cal Q}_{t} \circ \E({\cal L}-{\cal K})_{t,\boldsymbol{\sigma}}\|_{\max}\prec\int_{0}^{t}\eta_{u}^{-1}M_{u}^{-3}du + W^{-D}\prec M_{t}^{-3}.
\end{align*}
Finally, we again use \eqref{eq:step6_improvedexpectation} and the bound $\|\vartheta_{t}\|_{\max}\prec \ell_{t}^{-2}$ to deduce the estimate
\begin{align*}
\|{\cal P}\circ \E({\cal L}-{\cal K})_{t,\boldsymbol{\sigma}}\cdot\vartheta_{t}\|_{\max}\prec M_{t}^{-3}.
\end{align*}
Combining the previous two displays and using ${\cal Q}_{t}\circ \E({\cal L}-{\cal K})_{t,\boldsymbol{\sigma}}+{\cal P}\circ \E({\cal L}-{\cal K})_{t,\boldsymbol{\sigma}}\cdot\vartheta_{t}=\E({\cal L}-{\cal K})_{t,\boldsymbol{\sigma}}$ completes the proof.
\end{proof}

\section{Proof of Lemma \ref{lem_ConArg}}\label{sec:Inh_LE}
As explained in the proof of Lemma 5.1 in \cite{YY_25}, Cauchy-Schwarz allows us to reduce to loops ${\cal L}_{t,\boldsymbol{\sigma},\textbf{a}}$ of even length $2m$ that are symmetric. Precisely, if we use the chain notation for ${\cal C}$ in Definition \ref{def_Gchain}, then
\begin{align}\label{sym_calL}
        \mathcal{L}_{t, \boldsymbol{\sigma}', \textbf{a}'} = & \langle E_{a_0} \cdot \mathcal{C}_{t, \boldsymbol{\sigma}, \textbf{a}} \cdot E_{a_m} \cdot \mathcal{C}^\dagger_{t, \boldsymbol{\sigma}, \textbf{a}} \rangle, \quad \boldsymbol{\sigma} = (z_1, z_2, \ldots, z_m), \quad \textbf{a} = (a_1, a_2, \ldots, a_{m-1}) \nonumber \\
 \boldsymbol{\sigma}' = & (z_1, z_2, \ldots, z_m, \bar{z}_m, \bar{z}_{m-1}, \ldots, \bar{z}_1), \quad \textbf{a}' = ( a_1, a_2, \ldots, a_{m-1}, a_m, a_{m-1}, \ldots, a_1, a_0).
\end{align}
The reduction to even-loops via Cauchy-Schwarz is justified by the following, which is (6.4) in \cite{YY_25}:
\begin{align}\label{adummzps}
    \max_{\textbf{a}, \boldsymbol{\sigma}} \left|\mathcal{L}^{(\text{length} = 2m+1)}_{t, \boldsymbol{\sigma}, \textbf{a}}\right|^2 \le \left|\max_{\textbf{a}, \boldsymbol{\sigma}} \mathcal{L}^{(\text{length} = 2m)}_{t, \boldsymbol{\sigma}, \textbf{a}} \cdot \max_{\textbf{a}, \boldsymbol{\sigma}} \mathcal{L}^{(\text{length} = 2m+2)}_{t, \boldsymbol{\sigma}, \textbf{a}}\right|.  
  \end{align}
We record this bound because it will be important shortly for another reason. Our proof of Lemma \ref{lem_ConArg} proceeds inductively in $m$, where we recall that $2m$ is the length of ${\cal L}$. Recall $\Omega$ from \eqref{usuayzoo}. We claim 
    \begin{equation}\label{LWLIMZ}
      {\bf 1}_{\Omega}\cdot    \mathcal{L}_{t_2, \boldsymbol{\sigma}', \textbf{a}'} \prec \left( W^{2} \ell_{t_1}^{2} \eta_{t_2} \right)^{-2m+1} + {\bf 1}_{\Omega}\frac{\eta_{t_1}}{\eta_{t_2}}
        \cdot
        \left(
        \mathcal{L}_{t, \boldsymbol{\sigma}_m^{(1)}, \textbf{a}_m} - \mathcal{L}_{t, \boldsymbol{\sigma}_m^{(2)}, \textbf{a}_m} \right) M_{t_{1}}^{-1},
\end{equation}
Recall that $2m$ is the length of ${\cal L}$. Above, we have also used the notation
\begin{align*}
\textbf{a}_l &= (a_0, a_1, \ldots, a_{l-1}, a_{l-1}, \ldots, a_1) \\
\boldsymbol{\sigma}_l^{(1)}  &= (\sigma_1, \sigma_2, \ldots, \sigma_{l-1}, \sigma_l, \bar{\sigma}_{l-1}, \ldots, \bar{\sigma}_2, \bar{\sigma}_1),\\
\boldsymbol{\sigma}_l^{(2)}  &= (\sigma_1, \sigma_2, \ldots, \sigma_{l-1}, \bar{\sigma}_l, \bar{\sigma}_{l-1}, \ldots, \bar{\sigma}_2, \bar{\sigma}_1).
\end{align*}
In particular, $\mathcal{L}_{t_2, \boldsymbol{\sigma}_m^{(1)}, \textbf{a}_m}$ and $\mathcal{L}_{t_2, \boldsymbol{\sigma}_m^{(2)}, \textbf{a}_m}$ are  loops of length $2m-1$, and $\boldsymbol{\sigma}_m^{(1)}, \; \boldsymbol{\sigma}_m^{(2)}\in \{+,-\}^{2m-1}$. The estimate \eqref{LWLIMZ} is (6.13) in \cite{YY_25}. It follows from the same argument; indeed, the only difference between \eqref{LWLIMZ} and (6.13) in \cite{YY_25} is the scaling of $W^{2}$ versus $W$ and $\ell_{\cdot}^{2}$ versus $\ell_{\cdot}$. But in the proof of (6.13) in \cite{YY_25}, the factors of $W$ come from the scaling of $E_{a}$ matrices, the form of the right-hand side of \eqref{55} when we apply induction in $m$, and the coefficient in the Ward identity in Lemma \ref{lem_WI_K}. In our setting, the correct factor is $W^{2}$. The scaling of $\ell_{\cdot}$ in the proof of (6.13) in \cite{YY_25} comes from the size of ${\cal I}_{a}$ sets therein (which are replaced by ${\cal I}_{a}^{(2)}$ here) and the form of the right-hand side of \eqref{55} when we apply induction in $m$. Thus, in our argument, $\ell_{\cdot}$ should be replaced by $\ell_{\cdot}^{2}$.

To perform the induction, we note that the event $\Omega$ in \eqref{usuayzoo} implies ${\bf 1}_{\Omega}\cdot  \mathcal{L}_{t_2, \boldsymbol{\sigma}_m^{(j)}, \textbf{a}_m}\prec 1$ when $m=1$ for $j=1,2$. Next, we apply \eqref{adummzps} with the left-hand side given by a loop of size $2m-1$ to bound it by a loop of length $2m-2$ and another of length $2m$. Then, we apply the inductive assumption that \eqref{res_lo_bo_eta} holds for the loop of length $2m-2$. This gives the following for $j=1,2$ on the left-hand side below:
$$
   {\bf 1}_{\Omega}\cdot  \mathcal{L}_{t_2, \boldsymbol{\sigma}_m^{(j)}, \textbf{a}_m} \prec {\bf 1}_{\Omega}\left[ \left( W^{2} \ell_{t_1}^{2} \eta_{t_2} \right)^{-2m+3} \max_{\boldsymbol{\sigma} , \textbf{a} } \left| \mathcal{L}^{\text{length=2m}}_{t_2, \boldsymbol{\sigma} , \textbf{a} } \right| \right]^{1/2}.
    $$
Plugging this into \eqref{LWLIMZ} gives
$$
      {\bf 1}_{\Omega}\cdot\max_{\boldsymbol{\sigma}' , \textbf{a}' } \left| \mathcal{L}^{\text{length=2m}}_{t_2, \boldsymbol{\sigma}' , \textbf{a}' } \right|
    \prec \left( W^{2} \ell_{t_1}^{2} \eta_{t_2} \right)^{-2m+1 } 
    +\left( W^{2} \ell_{t_1}^{2} \eta_{t_2} \right)^{-m+1/2}{\bf 1}_{\Omega}
     \max_{\boldsymbol{\sigma}' , \textbf{a}' } \left| \mathcal{L}^{\text{length=2m}}_{t_2, \boldsymbol{\sigma}' , \textbf{a}' } \right|^{1/2}
$$
If $W^{2}\ell_{t_{1}}^{2}\eta_{t_{2}}\geq W^{\tau}$ for any $\tau>0$, then we can move the last term to the left-hand side and deduce the desired estimate \eqref{res_lo_bo_eta}. If $W^{2}\ell_{t_{1}}^{2}\eta_{t_{2}}\leq W^{\tau}$ for $\tau>0$ small, then the desired estimate \eqref{res_lo_bo_eta} follows by the restriction to the event $\Omega$ from \eqref{usuayzoo} anyway. This completes the proof. \qed

\section{Evolution kernel estimates}\label{ks}

In this section we collect the bounds on $L_{\max\to \max}$ norm of $\mathcal{U}$ defined in \eqref{def_Ustz}. First, we state the trivial bound, proved in Lemma 7.1 of \cite{YY_25}. We note that the proof in \cite{YY_25} is independent of the dimension of the band matrix $d$. 

\begin{lemma}[$\|{\cal U}\|_{\max\to \max}$ estimate]\label{lem:sum_Ndecay}
Let $\cal A$ be a tensor $\left(\mathbb{Z}_L^2\right)^n\to \mathbb C$ and $n\ge 2$.  
Then 
\begin{align}\label{sum_res_Ndecay}
   \| {\cal U}_{s,t,\boldsymbol{\sigma}}\;\circ {\cal A}\|_{\max } \prec \|{\cal A}\|_{\max}
    \cdot \left(  \eta_s/ \eta_t\right)^{n } ,\quad s<t. 
\end{align}
\end{lemma}

The following lemma shows how $\mathcal{U}_{s,t,(+,-)}$ affects the decay property of the 2-tensors.

\begin{lemma}[Tail Estimates]\label{TailtoTail}
Recall  
$$
{\cal T}_{t}(\ell) := M_t^{-2}\exp \left(- \left| \ell /\ell_t\right|^{1/2} \right)
$$
Fix $\boldsymbol{\sigma}=(+,-)$. Assume that for ${\cal A}_{\textbf{a}}$, $\textbf{a}\in \mathbb Z^2_L$ and for some large $D>0$, we have 
$$
{\cal A}_{\textbf{a}}\le {\cal T}_{s }(|a_1-a_2|_L)+W^{-D}
$$
Recall $\cal U$ from Definition \ref{def_Ustz}. We have 
\begin{align}
\label{neiwuj} 
\left({\cal U}_{s,t,\boldsymbol{\sigma}} \circ 
{\cal A}\right)_{\textbf{a}} & \prec
 {\cal T}_{t}(|a_1-a_2|_L)+W^{-D} \cdot(\eta_s/\eta_t)^2, \quad {\rm if}\quad  |a_1-a_2|_{L}\ge \ell_t^* :=(\log W)^{3/2}\ell_t 
\end{align}
 \end{lemma}

The analogous statement was proved in Lemma 7.2 of \cite{YY_25} in case of one-dimensional band matrices $d=1$. The proof of \cite{YY_25} goes through for $d=2$ as well. The only two adjustments necessary are the infinity norm bound $\|\Theta^{(B)}_t\|_{\max} \prec \ell_t^{-2}\eta_t^{-1}$, and the summation bound $\sum_{b, |a-b|_{L}\le \ell} 1 \prec \ell^2$ for any $a\in\mathbb{Z}_L^2$ and $\ell\in\mathbb{Z}_+$.
 
 \bigskip

The main technical part of this section is the following improved $L_{\max\to \max}$ bounds on $\mathcal{U}$.
 
\begin{lemma}[${\cal U}_{s,t}$ on fast decay tensor]\label{lem:sum_decay}

Let $\cal A$ be a tensor $\left(\mathbb{Z}_L^2\right)^n\to \mathbb R$, $n\ge 2$. We say that $\cal A$ has $(\tau, D)$ decay at time $s$ if  for some $\tau,D>0$, the following bound holds:
\begin{equation}\label{deccA0}
    \max_i\|a_i-a_j\|\ge \ell_s W^{\tau} \implies  {\cal A}_{\textbf{a}}=O(W^{-D}),\quad \textbf{a}=(a_1,a_2\cdots, a_n) .
\end{equation}
Now, suppose that $\cal A$ has $(\tau, D)$ decay at time $s$, and that $\cal A$ has rank $n\geq2$. We have
\begin{align}\label{sum_res_1}
    \left({\cal U}_{s,t,\boldsymbol{\sigma}} \circ {\cal A}\right)_\textbf{a} \le C_n W^{C_n\tau}\cdot   \|{\cal A}\|_{\max}
    \cdot \left(\frac{\ell_t}{\ell_s}\right)^{2}
    \cdot \left(\frac{\ell_s^{2}\eta_s}{\ell_t^{2}\eta_t}\right)^{n} +W^{-D+C_n}.
\end{align}
We will now consider five cases in which the bound \eqref{sum_res_1} can be improved.

Case 1: Suppose that for $1\leq k\leq n$, we have
$$ 
\sigma_k=\sigma_{k-1}, \quad \boldsymbol{\sigma}=(\sigma_1,\cdots, \sigma_n).
$$ 
Then the following estimate holds:
\begin{align}\label{nonalternating}
    \left({\cal U}_{s,t,\boldsymbol{\sigma}} \circ {\cal A}\right)_\textbf{a} \le W^{C_n\tau}\cdot   \|{\cal A}\|_{\max}\left(\frac{\ell_s^{2}\eta_s}{\ell_t^{2}\eta_t}\right)^{n } +W^{-D+C_n}
\end{align}

Case 2: Suppose that ${\cal A}_{\textbf{b}}$ has the sum zero property, i.e.
\begin{align}\label{sumAzero}
 \sum_{a_2,\,\ldots,\, a_n}{\cal A}_{\textbf{a}}=0, \forall a_1.
\end{align}
Then the following estimate holds:
\begin{align}\label{sum_res_2}
    \left({\cal U}_{s,t,\boldsymbol{\sigma}} \circ {\cal A}\right)_\textbf{a} \le W^{C_n\tau}\cdot   \|{\cal A}\|_{\max}\cdot\frac{\ell_{t}}{\ell_{s}}
    \cdot \left(\frac{\ell_s^{2}\eta_s}{\ell_t^{2}\eta_t}\right)^{n } +W^{-D+C_n}
\end{align} 

Case 3: Suppose that ${\cal A}_{\textbf{b}}$ has the sum zero property \eqref{sumAzero}. Suppose also that ${\cal A}_{\textbf{b}}$ has the following form, which we explain afterwards:
\begin{align}\label{a-local-form}
{\cal A}_{\textbf{b}}=\sum_{u=0}^{K}\sum_{\textbf{x},\textbf{y}\in(\Z_{WL}^{2})^{u}}C_{\textbf{b},\textbf{x},\textbf{y}}\cdot\prod_{i=1}^{u}G_{s}(\sigma_{i})_{x_{i}y_{i}}.
\end{align}
Here, $K=O(1)$ is fixed. The coefficients $C_{\textbf{b},\textbf{x},\textbf{y}}$ are deterministic, and they satisfy the three conditions below. First, we have the following for some constant $C'=O(1)$:
\begin{align*}
\|C\|_{\max}\leq N^{C'}.
\end{align*}
Second, there exists $\tau>0$ such that 
\begin{align*}
\max_{i}\min_{k}\Big(|[x_{i}]-b_{k}|_{L}+|[y_{i}]-b_{k}|_{L}\Big) \geq W^{\tau}\ell_{s} \Rightarrow C_{\textbf{b},\textbf{x},\textbf{y}}=0.
\end{align*}
{Third, let $\Lambda>0$ be a deterministic constant such that 
\begin{align*}
\|{\cal A}\|_{\max}\prec\Lambda.
\end{align*}
}
Now, suppose $0\leq s\leq 1-N^{-1+\delta}$ for some fixed $\delta>0$, and that \eqref{Gt_bound_flow} and \eqref{Eq:Gdecay_w} hold at time $s$. Then
\begin{align}\label{sum_res_3}
\left({\cal U}_{s,t,\boldsymbol{\sigma}}\circ({\cal A}-\E{\cal A})\right)_{\textbf{a}}&\prec C_{n}W^{C_{n}\tau}\cdot{\Lambda} \cdot\left(\frac{\ell_{s}^{2}\eta_{s}}{\ell_{t}^{2}\eta_{t}}\right)^{n}+W^{-D+C_{n}}.
\end{align}
Case 4: Suppose that $\mathcal{A}_{\textbf{b}}$ has the sum zero property \eqref{sumAzero} and that $\mathcal{A}$ is symmetric, i.e.:
$$
\mathcal{A}_{a, a+s_2,\ldots, a+s_n} = \mathcal{A}_{a, a-s_2,\ldots, a-s_n}
$$
for all $a, s_2, \ldots, s_n \in \mathbb{Z}_L^2$. Then we have the following bound:
\begin{equation}\label{symmetric_tensor}
    \left(\mathcal{U}_{s,t,\boldsymbol{\sigma}} \circ \mathcal{A}\right)_\textbf{a} \le W^{C_n\tau}\cdot   \|\mathcal{A}\|_{\max} 
    \cdot \left(\frac{\ell_s^2\eta_s}{\ell_t^2\eta_t}\right)^{n } +W^{-D+C_n}.
\end{equation}

Case 5: Let $\mathcal{A}$ be a $2n$-tensor $\left(\mathbb{Z}^2_L\right)^{n} \times \left(\mathbb{Z}^2_L\right)^{n} \rightarrow \mathbb{R}$, $n\ge 2$ with $(\tau, D)$ decay at time $s$. Suppose that $\mathcal{A}$ has the double sum zero property, meaning that
\begin{equation*}
    \sum_{a_2, \ldots a_n \in \mathbb{Z}_L^2} \mathcal{A}_{\textbf{a}, \textbf{a}'} = 0\; \quad \forall a_1,\textbf{a}', \quad\text{and}\quad  \sum_{a'_2, \ldots a'_n \in \mathbb{Z}_L^2} \mathcal{A}_{\textbf{a}, \textbf{a}'} = 0\; \quad \forall \textbf{a}, a'_1.
\end{equation*}
Then we have the bound
\begin{equation}\label{eq:double_sum_zero_tensor}
    \left(\left(\mathcal{U}_{s,t,\boldsymbol{\sigma}}\otimes \mathcal{U}_{s,t,\boldsymbol{\bar{\sigma}}}\right) \circ \mathcal{A}\right)_{\textbf{a},\textbf{a}} \le W^{C_n\tau}\cdot   \|\mathcal{A}\|_{\max} 
    \cdot \left(\frac{\ell_s^2\eta_s}{\ell_t^2\eta_t}\right)^{2n} +W^{-D+C_n}.
\end{equation}.    

\end{lemma} 

The cases 1-2 of this Lemma \ref{lem:sum_decay} are analogous to Lemma 7.3 of \cite{YY_25}. However, in two dimensions the the additional factor $\frac{\ell_s}{\ell_t}$ coming from the sum zero property of $\mathcal{A}$ is insufficient for the analysis of the loop hierarchy in Section \ref{Sec:Stoflo}. Thus, we show that another $\frac{\ell_s}{\ell_t}$ factor can be obtained from symmetry, double sum zero property of $\mathcal{A}$ in cases 4-5, or a CLT-like cancellation in case 3.

\begin{proof}[Proof of Lemma \ref{lem:sum_decay}]
By definition of ${\cal U}_{s, t}$ in \eqref{def_Ustz} and \eqref{def_Ustz_2}, we have 
\begin{align}
      \left(\mathcal{U}_{s, t, \boldsymbol{\sigma}} \circ \mathcal{A}\right)_a=\sum_{b_1, \ldots, b_n} \prod_{i=1}^n\psi_i  \cdot \mathcal{A}_{\boldsymbol{b}}, \quad \boldsymbol{b}=\left(b_1, \ldots, b_n\right),
\end{align}
\begin{align}\label{def_psixi}
     \psi_i=\delta_{a_ib_i}+\Xi_i,  \quad \quad 
\Xi_i:=-(s-t)\xi_i \cdot \left(S^{(B)}\cdot \Theta^{(B)}_{t \xi_i}\right)_{a_ib_i},\quad\quad \xi_i=m(\sigma_i)m(\sigma_{i+1}).
 \end{align} 
We now expand
\begin{align*}
\sum_{b_{1},\ldots,b_{n}}\prod_{i=1}^{n}\psi_{i}\cdot{\cal A}_{\textbf{b}}&=\sum_{ A\subset\llbracket1,n\rrbracket}\sum_{b_{1},\ldots,b_{n}}\left(\prod_{j\in A^{c}}\Xi_{j}\right)\left(\prod_{j\in A}\delta_{a_{j}b_{j}}\right)\cdot{\cal A}_{\textbf{b}}.    
\end{align*}
By \eqref{def_psixi}, the entries of $\Xi$ are $O(\eta_{s})$. For any $A\neq\emptyset$, there must be at least one factor of $\delta_{a_{j}b_{j}}$ on the r.h.s. above. By the $(\tau,D)$ decay property of ${\cal A}_{\textbf{b}}$ at time $s$, this implies that all but $O(W^{\tau}\ell_{s}^{2})^{|A^{c}|}$ many indices in the sum over $b_{1},\ldots,b_{n}$ on the r.h.s. must be $O(W^{-D})$. Ultimately, we deduce that 
\begin{align}\label{iukwjn-d=2}
\sum_{\emptyset\neq A\subset\llbracket1,n\rrbracket}\sum_{b_{1},\ldots,b_{n}}\prod_{i=1}^{n}\psi_{i}\cdot{\cal A}_{\textbf{b}}&\leq\sum_{\emptyset\neq A\subset\llbracket1,n\rrbracket} C_{n}\eta_{s}^{|A^{c}|}(\ell_{s}^{2})^{|A^{c}|}\|{\cal A}\|_{\max}+W^{-D+C_{n}}\nonumber\\
&\leq C_{n}\left(\frac{\ell_{s}^{2}\eta_{s}}{\ell_{t}^{2}\eta_{t}}\right)^{n}\|{\cal A}\|_{\max}+W^{-D+C_{n}},
\end{align}
where the last bound follows because $\ell_{s}^{2}\eta_{s}$ is constant order and positive.

Thus we only need to prove the bounds \eqref{sum_res_1}, \eqref{nonalternating}, \eqref{sum_res_2}, \eqref{symmetric_tensor} and \eqref{sum_res_3} on the main term $\sum_{\textbf{b}}  \left(\prod_{i=1}^n  \Xi_i \right) {\cal A}_{\textbf{b}}$. By definition of $\Theta^{(B)}$, we have 
$$
\Xi_i \prec \eta_s \,\ell_t^{-2}\eta_t^{-1}.
$$
Then, by $(\tau,D)$ decay property of $\cal A$ at time $s$ in \eqref{deccA0}, we have
\begin{align}\label{llswuwg}
  \sum_{\textbf{b}}   \prod_{i=1}^n  \Xi_i  \cdot 
{\cal A}_{\textbf{b}}
\;\le\; & C W^{\tau}\cdot \frac{\ell_s^{2} \eta_s}{\ell_t^{2} \eta_t} \max_{b_n}
\cdot \left|\sum_{b_1\cdots b_{n-1}} \prod_{i=1}^{n-1}  \Xi_i  \cdot 
{\cal A}_{\textbf{b}}\right|+W^{-D+C}
\\\nonumber
\;\le\; & C_n W^{C_n\tau} \left(\frac{\ell_s^{2} \eta_s}{\ell_t^{2} \eta_t}\right)^{n-1} \max_{b_2,\cdots, b_n}  \left|\sum_{b_1 }\Xi_1 {\cal A}_{\textbf{b}}\right|+W^{-D+C_n}
\end{align}
By the decay of $(\tau,D)$ decay of $\Theta^{(B)}_t$ at time $t$, we have 
\begin{align}\label{tyzcsq}
  \sum_{b_1 }\Xi_1 \left|{\cal A}_{\textbf{b}}\right|\le CW^\tau (\eta_s/\eta_t)\|{\cal A}\|_{\max}+W^{-D}  .
\end{align}
Together with \eqref{llswuwg} and \eqref{iukwjn-d=2}, we deduce  \eqref{sum_res_1}. For a non-alternating tensor, we can assume without loss of generality that $\sigma_1=\sigma_2=+$. Then $\sum_{b_1}|\Xi_1|=O(1)$, and we get an improved bound
$$
\sum_{b_1 }\Xi_1\left| {\cal A}_{\textbf{b}}\right|\le  C\|{\cal A}\|_{\max} .
$$
Together with \eqref{llswuwg} and \eqref{iukwjn-d=2}, this proves \eqref{nonalternating}.

Now we assume that $\cal A$ has the sum zero property \eqref{sumAzero}. Since we proved \eqref{nonalternating} for case 1,  we can assume that $\sigma_i\ne \sigma_{i+1}$ for any $1\le i\le n$, and thus $\xi_i=|m|^2=1$ (in \eqref{def_psixi}). For each $\Xi_i$ with $i\ge 2$,  we write 
$$
\Xi_i=\Xi^0_i+\Xi^1_i+\Xi^2_i:=\Xi^0(a_i-b_1) + \Xi^1(a_i-b_1,r_i) + \Xi^2(a_i-b_1,r_i),
$$ 
where the shifts are defined by $r_i = b_i - b_1$ for $2\le i\le n$, and
\begin{align*}
    \Xi^0(a) &= (t-s)\left(S^{(B)}\Theta^{(B)}_t\right)_{a, 0  },\\
    \Xi_i^1(a, r) &= (t-s)\left[\frac12\left(S^{(B)}\Theta^{(B)}_t\right)_{a, r} - \frac12\left(S^{(B)}\Theta^{(B)}_t\right)_{a, -r}\right],\\
    \Xi_i^2(a, r) &= (t-s)\left[\frac12\left(S^{(B)}\Theta^{(B)}_t\right)_{a, r} + \frac12\left(S^{(B)}\Theta^{(B)}_t\right)_{a, -r} - \left(S^{(B)}\Theta^{(B)}_t\right)_{a, 0}\right].
\end{align*}
By Lemma \ref{lem_propTH} and $t-s = O(\eta_s)$, we have 
\begin{align}
    \left|\Xi^{0,1,2}(a)\right| &\prec \frac{\eta_s}{\ell_t^2\eta_t},\label{Xi-bound-0}\\
    \left|\Xi^{1,2}(a, r)\right| &\prec \frac{\eta_s|r|_L}{|a|_L+1}+\frac{\eta_{s}|r|_{L}}{\ell_{t}^{2}\eta_{t}^{1/2}},\label{Xi-bound-1}\\
    \left|\Xi^2(a, r)\right| &\prec \frac{\eta_s|r|_L^2}{(|a|_L+1)^2}+\frac{\eta_{s}|r|_{L}^{2}}{\ell_{t}^{2}}.\label{Xi-bound-2}
\end{align}
We expand the main term
\begin{align}
\sum_{\textbf{b}}   \left(\prod_{i=1}^n  \Xi_i\right)  \cdot 
{\cal A}_{\textbf{b}}&=\sum_{\textbf{b}} \Xi_1
\left(\prod_{i=2}^n  \left(\Xi^0_i+  \Xi^1_i + \Xi^2_i\right)\right)\cdot
{\cal A}_{\textbf{b}}\nonumber\\
&= \sum_{\llbracket2,n\rrbracket = A\sqcup B\sqcup C} \sum_{\textbf{b}}\prod_{i\in A} \Xi_i^0 \prod_{i\in B} \Xi_i^1 \prod_{i\in C} \Xi_i^2 \cdot \Xi_1 \mathcal{A}_{\textbf{b}}.\label{Xi-expansion}
\end{align}
The leading term, with $B, C =\emptyset$, vanishes due to the sum zero property of $\mathcal{A}$. For other partitions of $\llbracket 2,n\rrbracket$, we pick an $i\in B\cup C$ and bound the corresponding $\Xi^{1,2}_i$ with \eqref{Xi-bound-1}. The rest of the $\Xi^{0,1,2}_i$ factors are bounded by \eqref{Xi-bound-0}. By the decay \eqref{deccA0} of $\cal A$,  the main contribution to the last equation comes from $|r_i|_L\le W^\tau \ell_s$, thus the summation over each $r_i$ contributes a factor of $W^{2\tau}\ell_s$. Therefore, we obtain the following bound for  \eqref{llswuwg}.
\begin{align}\label{llswuwg2}
  \sum_{\textbf{b}}   \prod_{i=1}^n  \Xi_i  \cdot 
{\cal A}_{\textbf{b}}
\;\le\; & C_n W^{C_n\tau} \left(\frac{\ell_s^{2} \eta_s}{\ell_t^{2} \eta_t}\right)^{n-2}  \max_{b_2,\cdots, b_n}  \sum_{b_1 }\left[\sum_{i=2}^n \frac{\eta_s\ell_s^3}{|a_i-b_1|_L+1}+\frac{\eta_s\ell_s^3}{\ell_t^2\eta_t^{1/2}}\right]\Xi_1\left| {\cal A}_{\textbf{b}}\right|+W^{-D+C_n}
\end{align}
Due to the exponential decay property of $\Xi_1$, we can restrict the summation over $b_1$ to $|a_1-b_1|_L\le W^\tau\ell_t$. For $i\ge 2$, we have the following summation bound
\begin{equation}\label{eq-1sum}
    \sum_{|b_1-a_1|\le W^\tau\ell_t} \frac{1}{|a_i-b_1|_L+1} \le CW^\tau \ell_t.
\end{equation}
Using this together with the entrywise bound \eqref{Xi-bound-0} on $\Xi_1$ and \eqref{tyzcsq}, we have 
\begin{align}\label{llswuwg3}
  \sum_{\textbf{b}}   \prod_{i=1}^n  \Xi_i  \cdot 
{\cal A}_{\textbf{b}}
\;\le\; & C_n W^{C_n\tau} \left(\frac{\ell_s^{2} \eta_s}{\ell_t^{2} \eta_t}\right)^{n-2}  \left[\frac{(\ell_s^2\eta_s)^2}{\ell_t^2\eta_t}\cdot\frac{\ell_t}{\ell_s}+\frac{(\ell_s^2\eta_s)^2}{(\ell_t^2\eta_t)^{3/2}}\cdot\frac{\ell_t}{\ell_s}\right]\| {\cal A}\|_{\max}+W^{-D+C_n}.
\end{align}
This implies \eqref{sum_res_2}, since $\ell_t^2\eta_t=O(1)$. In case 4, we additionally have the symmetry property of $\mathcal{A}$, so that $\Xi^1(a, r) = -\Xi^1(a, -r)$ by definition, and thus
$$
\sum_{r_2,\ldots,r_n} \Xi^1(a_i-b_1, r_i) \mathcal{A}_{b_1, b_1+r_2,\ldots,b_1+r_n} = 0.
$$
Then the terms in the expansion of \eqref{Xi-expansion} with $|B| = 1$ and $C= \emptyset$ vanish as well. We split the remaining non-zero terms in \eqref{Xi-expansion} into two groups. The first group contains the partitions with $C \neq \emptyset$, where we use \eqref{Xi-bound-2} on one of $\Xi_i^2$ with $i\in C$ and \eqref{Xi-bound-0} on the other $\Xi^{0,1,2}_i$. The second group covers the partitions with $C = \emptyset$, $|B|\ge 2$, where we use \eqref{Xi-bound-1} for two terms $\Xi^1_{i,j}$ with $i, j\in B$ and \eqref{Xi-bound-0} on the other $\Xi^{0,1}_i$ terms. In total we get
\begin{align}
\sum_{\textbf{b}}   \prod_{i=1}^n  \Xi_i  \cdot 
{\cal A}_{\textbf{b}}
\;&\le\; C_n W^{C_n\tau} \left(\frac{\ell_s^{2} \eta_s}{\ell_t^{2} \eta_t}\right)^{n-2} \max_{b_2,\cdots, b_n}  \sum_{b_1}\left[\sum_{i=2}^n\frac{\ell_s^4\eta_s}{(|a_i-b_1|_L+1)^2}+ \frac{\ell_s^4\eta_s}{\ell_t^2}\right] \Xi_1 \left|{\cal A}_{\textbf{b}}\right|+W^{-D+C_n}\nonumber\\
&+\delta_{n\ge 3}C_n W^{C_n\tau} \left(\frac{\ell_s^{2} \eta_s}{\ell_t^{2} \eta_t}\right)^{n-3} \max_{b_2,\cdots, b_n}  \sum_{b_1}\left[\sum_{i=2}^n \frac{\ell_s^3\eta_s}{|a_i-b_1|_L+1}+\frac{\ell_s^3\eta_s}{\ell_t^2\eta_t^{1/2}}\right]^2 \Xi_1 \left|{\cal A}_{\textbf{b}}\right|\label{eq-2-first-der}\\
\end{align}
Similarly to the proof of case 2, we now bound $\Xi_1$ entrywise using \eqref{Xi-bound-0}, use the decay of $\Xi_1$ to restrict $b_1$ summation to $|b_1-a_1|_L\le W^\tau\ell_t$ and sum out the prefactors using \eqref{eq-1sum} and
\begin{equation}\label{eq-2sum}
   \sum_{|b_1-a_1|_L\le W^\tau\ell_t} \frac{1}{(|a_i-b_1|_L+1)(|a_j-b_1|_L+1)} \le W^{\tau} 
\end{equation}
for $i, j\ge 2$. Then 
\begin{align*}
\sum_{\textbf{b}}   \prod_{i=1}^n  \Xi_i  \cdot 
{\cal A}_{\textbf{b}} \;&\le\; C_n W^{C_n\tau} \left(\frac{\ell_s^{2} \eta_s}{\ell_t^{2} \eta_t}\right)^{n-2} \frac{(\ell_s^2\eta_s)^2}{\ell_t^2\eta_t}\|\mathcal{A}\|_{\max} \\
&+ \delta_{n\ge 3}C_n W^{C_n\tau} \left(\frac{\ell_s^{2} \eta_s}{\ell_t^{2} \eta_t}\right)^{n-3} \left[\frac{(\ell_s^2\eta_s)^3}{\ell_t^2\eta_t}+\frac{(\ell_s^2\eta_s)^3}{(\ell_t^2\eta_t)^{3/2}}+\frac{(\ell_s^2\eta_s)^3}{(\ell_t^2\eta_t)^2}\right] \|\mathcal{A}\|_{\max} + W^{-D+C_n}.
\end{align*}
This finishes the proof of \eqref{symmetric_tensor}.

{We now prove \eqref{sum_res_3}. Because ${\cal A}_{\textbf{b}}$ has the fast decay property, at the cost of an error term of the form $O(W^{-D})$ for any $D>0$ fixed, we can restrict ${\cal A}_{\textbf{b}}$ to indices $\textbf{b}=(b_{1},\ldots,b_{n})$ such that 
\begin{align*}
\max_{i}|b_{i}-b_{1}|_{L}\leq W^{\tau}\ell_{s}.
\end{align*}
In particular, we can assume that in the representation \eqref{a-local-form}, the coefficients $C_{\textbf{b},\textbf{x},\textbf{y}}$ satisfy
\begin{align}
\max_{i}\Big(|[x_{i}]-b_{1}|_{L}+|[y_{i}]-b_{1}|_{L}\Big)\geq W^{2\tau}\ell_{s} \Rightarrow C_{\textbf{b},\textbf{x},\textbf{y}}=0.\label{c-decay-new}
\end{align}
Take \eqref{Xi-expansion}. Because ${\cal A}_{\textbf{b}}$ has the sum-zero property by assumption, the term with $B,C=\emptyset$ vanishes as explained after \eqref{Xi-expansion}. For other partitions $\llbracket2,n\rrbracket=A\sqcup B\sqcup C$, there must be an index $i\in B\sqcup C$ with a factor of either $\Xi^{1}_{i}$ or $\Xi^{2}_{i}$; as before, call this factor $\Xi^{1,2}_{i}$, and without loss of generality (upon relabeling indices), assume that $i=2$. Thus, the quantity from \eqref{Xi-expansion} that we must control has the form
\begin{align*}
\sum_{\llbracket2,n\rrbracket=A\sqcup B\sqcup C}\sum_{b_{1}}\Xi_{1}\Big(\sum_{b_{2},\ldots,b_{n}}\Xi^{1,2}_{2}\prod_{i\in A}\Xi^{0}_{i}\prod_{i\in B\setminus\{2\}}\Xi^{1}_{i}\prod_{i\in C\setminus\{2\}}\Xi^{2}_{i}\cdot ({\cal A}_{\textbf{b}}-\E{\cal A}_{\textbf{b}})\Big)
\end{align*}
By the fast decay property of ${\cal A}_{\textbf{b}}$, the sum over $b_{2},\ldots,b_{n}$ can be restricted to $b_{2},\ldots,b_{n}$ such that $|b_{i}-b_{1}|_{L}\leq W^{\tau}\ell_{s}$ for every $i=2,\ldots,n$. Moreover, if we use the fast decay of $\Xi_{1}$, it suffices to restrict to $b_{1}$ such that $|b_{1}-a_{1}|_{L}\leq W^{\tau}\ell_{t}$ in the above display. In particular, the object we must control (in terms of the right-hand side of the desired estimate \eqref{sum_res_3}) is
\begin{align}
    \Gamma:=\sum_{|b_{1}-a_{1}|_{L}\leq W^{\tau}\ell_{t}}\Xi_{1}\Big(\sum_{b_{2},\ldots,b_{n}}\Xi^{1,2}_{2}\prod_{i\in A}\Xi^{0}_{i}\prod_{i\in B\setminus\{2\}}\Xi^{1}_{i}\prod_{i\in C\setminus\{2\}}\Xi^{2}_{i}\cdot ({\cal A}_{\textbf{b}}-\E{\cal A}_{\textbf{b}})\Big).\label{gamma_def_clt}
\end{align}
Now, recall that ${\cal A}_{\textbf{b}}\prec \Lambda$. Also, by \eqref{a-local-form} and the assumption that $s\leq N^{-1+\delta}$ for some fixed $\delta>0$, we have ${\cal A}_{\textbf{b}}=O(W^{C})$ for some $C=O(1)$. Thus, we deduce $\E{\cal A}_{\textbf{b}}\prec\Lambda +W^{-D'}$ for any $D'=O(1)$. So, if we follow the derivation of \eqref{llswuwg2}, we deduce the following for any $D=O(1)$:
\begin{align*}
\sum_{b_{2},\ldots,b_{n}}\Xi^{1,2}_{2}\prod_{i\in A}\Xi^{0}_{i}\prod_{i\in B\setminus\{2\}}\Xi^{1}_{i}\prod_{i\in C\setminus\{2\}}\Xi^{2}_{i} ({\cal A}_{\textbf{b}}-\E{\cal A}_{\textbf{b}})\prec \left(\frac{\ell_s^{2} \eta_s}{\ell_t^{2} \eta_t}\right)^{n-2} \eta_{s}\ell_{s}^{3}\Big(\frac{1}{1+|a_{2}-b_{1}|_{L}}+\frac{1}{\ell_{t}^{2}\eta_{t}^{1/2}}\Big)\Lambda + W^{-D}.
\end{align*}
We now use Lemma \ref{clt-lemma} in two ways. First, we use case 2 of Lemma \ref{clt-lemma} with
\begin{align*}
 Z_{ab_{1}}&:=(1+|a-b_{1}|_{L})^{-1}\cdot \eta_{s}^{-1}\ell_{t}^{2}\eta_{t} \cdot (\Xi_{1})_{a_{1}b_{1}}\cdot \mathbf{1}[|b_{1}-a_{1}|_{L}\leq W^{\tau}\ell_{t}],\\
 Y_{s,b_{1}}&:=\Big(\frac{1}{1+|a_{2}-b_{1}|_{L}}+\frac{1}{\ell_{t}^{2}\eta_{t}^{1/2}}\Big)^{-1}\cdot \sum_{b_{2},\ldots,b_{n}}\Xi^{1,2}_{2}\prod_{i\in A}\Xi^{0}_{i}\prod_{i\in B\setminus\{2\}}\Xi^{1}_{i}\prod_{i\in C\setminus\{2\}}\Xi^{2}_{i}\cdot ({\cal A}_{\textbf{b}}-\E{\cal A}_{\textbf{b}}).
\end{align*}
(In this application of Lemma \ref{clt-lemma}, the $\tau$ parameter there is equal to two times the $\tau$ parameter here; see the decay condition in \eqref{c-decay-new}, for example. Moreover, recall from \eqref{Xi-bound-0} that $\Xi_{1}\prec \eta_{s}\ell_{t}^{-2}\eta_{t}^{-1}$. Thus, our choice of $Z_{ab_{1}}$ satisfies the assumption \eqref{clt-z-form}.) Recall also that 
\begin{align*}
Y_{s,b_{1}}\prec \left(\frac{\ell_{s}^{2}\eta_{s}}{\ell_{t}^{2}\eta_{t}}\right)^{n-2}\cdot\eta_{s}\ell_{s}^{3}\cdot\Lambda.
\end{align*}
Thus, when we do this application of Lemma \ref{clt-lemma}, we arrive at the following for any $D=O(1)$:
\begin{align*}
\eta_{s}\ell_{t}^{-2}\eta_{t}^{-1}\cdot \sum_{b_{1}}Z_{a_{2}b_{1}}Y_{s,b_{1}}&\prec W^{C\tau}\ell_{s}\cdot \left(\frac{\ell_s^{2} \eta_s}{\ell_t^{2} \eta_t}\right)^{n-2} \cdot \eta_{s}\ell_{t}^{-2}\eta_{t}^{-1} \cdot \eta_{s}\ell_{s}^{3} \cdot \Lambda + W^{-D}\\
&\prec W^{C\tau}\left(\frac{\ell_s^{2} \eta_s}{\ell_t^{2} \eta_t}\right)^{n}\ell_{t}^{2}\eta_{t}\cdot \Lambda + W^{-D}.
\end{align*}
Recalling that $\ell_{t}^{2}\leq\eta_{t}^{-1}$ (see \eqref{eq:bcal_k} and \eqref{eta}) then yields
\begin{align}
\eta_{s}\ell_{t}^{-2}\eta_{t}^{-1}\cdot\sum_{b_{1}}Z_{a_{2}b_{1}}Y_{s,b_{1}}&\prec W^{C\tau}\left(\frac{\ell_s^{2} \eta_s}{\ell_t^{2} \eta_t}\right)^{n}\cdot \Lambda + W^{-D}.\label{clt-lemma-application-1}
\end{align}
Next, we use \eqref{clt-lemma-final-result} (using $\tilde{Z}$ instead of $Z$ in order to avoid confusion with the above notation) with
\begin{align*}
\tilde{Z}_{ab_{1}}:=\ell_{t}^{-2}\eta_{t}^{-\frac12}\cdot\eta_{s}^{-1}\ell_{t}^{2}\eta_{t}\cdot (\Xi_{1})_{ab_{1}}\cdot\mathbf{1}[|b_{1}-a|_{L}\leq W^{\tau}\ell_{t}]
\end{align*}
and the same choice of $Y_{s,b_{1}}$ as above. Note that by \eqref{Xi-bound-0}, this choice of $\tilde{Z}_{ab_{1}}$ satisfies $\max_{b}|\tilde{Z}_{ab_{1}}|\prec \ell_{t}^{-2}\eta_{t}^{-1/2}$. Thus, we deduce the following for any $D=O(1)$:
\begin{align*}
\eta_{s}\ell_{t}^{-2}\eta_{t}^{-1}\cdot \sum_{b_{1}}\tilde{Z}_{a_{1}b_{1}}Y_{s,b_{1}}&\prec W^{C\tau}\ell_{t}\ell_{s}\cdot \ell_{t}^{-2}\eta_{t}^{-\frac12}\cdot \eta_{s}\ell_{t}^{-2}\eta_{t}^{-1}\cdot\eta_{s}\ell_{s}^{3}\cdot \left(\frac{\ell_s^{2} \eta_s}{\ell_t^{2} \eta_t}\right)^{n-2}\cdot \Lambda+W^{-D}\\
&\prec W^{C\tau}\cdot \ell_{t}\eta_{t}^{\frac12} \cdot \left(\frac{\ell_s^{2} \eta_s}{\ell_t^{2} \eta_t}\right)^{n}\cdot \Lambda +W^{-D}.
\end{align*}
If we again use $\ell_{t}\eta_{t}^{\frac12}\leq1$, then we obtain
\begin{align*}
\eta_{s}\ell_{t}^{-2}\eta_{t}^{-1}\cdot\sum_{b_{1}}\tilde{Z}_{a_{1}b_{1}}Y_{s,b_{1}}&\prec W^{C\tau}\cdot \left(\frac{\ell_s^{2} \eta_s}{\ell_t^{2} \eta_t}\right)^{n}\cdot \Lambda +W^{-D}.
\end{align*}
Finally, we observe that, in the above notation, we have 
\begin{align*}
\Gamma=\eta_{s}\ell_{t}^{-2}\eta_{t}^{-1}\sum_{b_{1}}Z_{a_{2}b_{1}}Y_{s,b_{1}}+\eta_{s}\ell_{t}^{-2}\eta_{t}^{-1}\sum_{b_{1}}\tilde{Z}_{a_{1}b_{1}}Y_{s,b_{1}}.
\end{align*}
Combining the previous two displays with \eqref{clt-lemma-application-1} shows that $\Gamma$ is controlled by the right-hand side of the desired estimate \eqref{sum_res_3}. As noted immediately prior to \eqref{gamma_def_clt}, this completes the proof of case 3, i.e. \eqref{sum_res_3}.

}
Finally, we consider the case 5 of a double sum zero tensor $\mathcal{A}$. Due to \eqref{nonalternating}, we only need to consider the case when $\boldsymbol{\sigma}$ is alternating. Similarly to the previous cases, we have 
\begin{equation*}
    \left(\left(\mathcal{U}_{s,t,\boldsymbol{\sigma}}\otimes \mathcal{U}_{s,t,\boldsymbol{\bar{\sigma}}}\right) \circ \mathcal{A}\right)_{\textbf{a},\textbf{a}} = \sum_{\textbf{b}, \textbf{b}'} \prod_{i=1}^n \psi_i \psi'_i \cdot \mathcal{A}_{\textbf{b},\textbf{b}'},
\end{equation*}
where $\psi_i$ are defined in \eqref{def_psixi} and $\psi'_i$ are defined similarly, up to replacing the $b_i$ with $b'_i$ and $\xi_i$ with $\bar{\xi_i}$. More precisely,
\begin{align}\label{def_psixi_prime}
     \psi'_i=\delta_{a_ib'_i}+\Xi'_i,  \quad \quad 
\Xi'_i:=-(s-t)\bar{\xi_i} \cdot \left(S^{(B)}\cdot \Theta^{(B)}_{t \bar{\xi_i}}\right)_{a_ib'_i}.
\end{align} 
Using a $2n$-tensor version of \eqref{iukwjn-d=2}, we get that
\begin{align*}
    \left|\sum_{\textbf{b}, \textbf{b}'} \left(\prod_{i=1}^n \psi_i \psi'_i-\prod_{i=1}^n\Xi_i\Xi'_i\right) \cdot \mathcal{A}_{\textbf{b},\textbf{b}'}\right| \le C_n \left(\frac{\ell_{s}^{2}\eta_{s}}{\ell_{t}^{2}\eta_{t}}\right)^{2n}\|{\cal A}\|_{\max}+W^{-D+C_{n}}.
\end{align*}
Thus, we only need to bound $\sum_{\textbf{b}, \textbf{b}'} \prod_{i=1}^n \Xi_i\Xi'_i \cdot \mathcal{A}_{\textbf{b},\textbf{b}'}$. We decompose $\Xi_i = \Xi_i^0 + \Xi_i^\ast$ and $\Xi'_i =  \Xi_i^{\prime,0} + \Xi_i^{\prime,\ast}$, where $r_i = b_i - b_1$, $r_i' = b_i' - b_1$ and
\begin{align*}
    \Xi_i^0 &= \Xi^0(a_i - b_1), \quad &\Xi^\ast_i &= \Xi^1(a_i - b_1, r_i) + \Xi^2(a_i - b_1, r_i),\\
    \Xi_i^{\prime,0} &= \Xi^0(a_i - b_1), \quad &\Xi^{\prime,\ast}_i &= \Xi^1(a_i - b_1, r'_i) + \Xi^2(a_i - b_1, r'_i).
\end{align*}
Due to the double sum zero property of $\mathcal{A}$, we have
\begin{align*}
    \sum_{\textbf{b}\setminus b_1} \prod_{i=1}^n \Xi_i^0 \mathcal{A}_{\textbf{b},\textbf{b}'} = 0, \qquad \sum_{\textbf{b}'\setminus b'_1} \prod_{i=1}^n \Xi_i^{\prime,0} \mathcal{A}_{\textbf{b},\textbf{b}'} = 0.
\end{align*}
These properties allow us to expand
\begin{align*}
    \sum_{\textbf{b}\setminus b_1, \textbf{b}'\setminus b_1'} \prod_{i=1}^n \Xi_i\Xi'_i \cdot \mathcal{A}_{\textbf{b},\textbf{b}'} = \Xi_1 \Xi_1^{\prime}\sum_{B, B'\subset \llbracket 2,n\rrbracket,\, B,B'\neq \emptyset} \prod_{i\in B^{c}} \Xi_i^0 \prod_{i\in B^{\prime,c}} \Xi_i^{\prime,0} \sum_{\textbf{b}\setminus b_1, \textbf{b}'\setminus b_1'} \prod_{i\in B} \Xi_i^\ast \prod_{i\in B'} \Xi_i^{\prime,\ast} \mathcal{A}_{\textbf{b},\textbf{b}'}.
\end{align*}
We bound one of the factors $\Xi_i^\ast$ with $i\in B$ and one of the factors $\Xi_i^{\prime,\ast}$ using \eqref{Xi-bound-1}. We control the rest of the factors $\Xi_i^0, \Xi_i^{\prime,0}, \Xi_i^\ast, \Xi_i^{\prime,\ast}$ for $i\ge2$ and $\Xi'_1$ using \eqref{Xi-bound-0}. Then
\begin{equation}\label{eq-dsz}
    \sum_{\textbf{b}, \textbf{b}'} \prod_{i=1}^n \Xi_i\Xi'_i \cdot \mathcal{A}_{\textbf{b},\textbf{b}'} \le C_n W^{C_n\tau} \left(\frac{\ell_s^2\eta_s}{\ell_t^2\eta_t}\right)^{2n-3}\cdot \max_{\textbf{b}\setminus b_1, \textbf{b}'} \sum_{b_1}\left[\sum_{i=2}^n \frac{\ell_s^3\eta_s}{|a_i-b_1|_L+1}+\frac{\ell_s^3\eta_s}{\ell_t^2\eta_t^{1/2}}\right]^2 \cdot \Xi_1 \left|\mathcal{A}_{\textbf{b},\textbf{b}'} \right| + W^{-D+C_n}.
\end{equation}
The bound in this case is completely analogous to the bound on the second line of \eqref{eq-2-first-der} and we get
$$
\sum_{\textbf{b}, \textbf{b}'} \prod_{i=1}^n \Xi_i\Xi'_i \cdot \mathcal{A}_{\textbf{b},\textbf{b}'} \le C_n W^{C_n\tau} \left(\frac{\ell_s^2\eta_s}{\ell_t^2\eta_t}\right)^{2n}\cdot  \|\mathcal{A}\|_{\max}  + W^{-D+C_n},
$$
which completes the proof.
\end{proof}

We now record the following lemma, which was used in the proof of \eqref{sum_res_3}. 
\begin{lemma}\label{clt-lemma}
Fix $u=O(1)$, and consider a function of $b\in\Z_{L}^{2}$ of the following form:
\begin{align}
Y_{s,b}=\sum_{\textbf{x},\textbf{y}\in(\Z_{WL}^{2})^{u}}C_{b,\textbf{x},\textbf{y}}\cdot\prod_{i=1}^{u}G_{s}(\sigma_{i})_{x_{i}y_{i}}.\label{clt-yform}
\end{align}
Above, $u=O(1)$ is fixed. The coefficients $C_{b,\textbf{x},\textbf{y}}$ are deterministic, and they satisfy the following conditions. First, for some $C'=O(1)$, we have 
\begin{align*}
\|C\|_{\max}\leq N^{C'}.
\end{align*}
Second, there exists $\tau>0$ such that 
\begin{align*}
\max_{i}\Big(|[x_{i}]-b|_{L}+|[y_{i}]-b|_{L}\Big)\geq W^{\tau}\ell_{s} \Rightarrow C_{b,\textbf{x},\textbf{y}}=0.
\end{align*}
{Third, there exists a deterministic constant $\Lambda>0$ such that 
\begin{align*}
\max_{b}|Y_{s,b}|\prec\Lambda.   
\end{align*}
}
Moreover, suppose that $0\leq s\leq 1-N^{-1+\delta}$ for fixed $\delta>0$, and that \eqref{Gt_bound_flow} and \eqref{Eq:Gdecay_w} hold at time $s$. 

Now, let $Z=(Z_{ab})_{a,b\in\Z_{L}^{2}}$ be a deterministic matrix. We consider two separate cases.
\medskip

Case 1: the matrix $Z$ satisfies the following constraint: 
\begin{align}
|a-b|_{L}\geq W^{\tau}\ell_{t} \Rightarrow Z_{ab}=0.\label{clt-zrange}
\end{align}
In this case, for any $D>0$ and for some $C=O(1)$, we have
\begin{align}
\sum_{b}Z_{ab}(Y_{s,b}-\E Y_{s,b})&\prec W^{C\tau}\ell_{t}\ell_{s}\cdot\max_{b}|Z_{ab}|\cdot\Lambda+W^{-D}\cdot \max_{b}|Z_{ab}|.\label{clt-lemma-final-result}
\end{align}

{Case 2: suppose that we have the bound
\begin{align}
Z_{ab}\prec(1+|a-b|_{L})^{-1},\quad a,b\in\Z_{L}^{2}.\label{clt-z-form}
\end{align}
Then, for any $D>0$ and some $C=O(1)$, we have
\begin{align}
 \sum_{b}Z_{ab}(Y_{s,b}-\E Y_{s,b})&\prec W^{C\tau}\cdot \ell_{s}\cdot\Lambda+W^{-D}.   \label{clt-lemma-final-result2}
\end{align}
}\end{lemma}
Note that the trivial bound on the left-hand side of \eqref{clt-lemma-final-result} has $\ell_t$ instead of $\ell_s$ on the right-hand side of \eqref{clt-lemma-final-result}. The key observation leading to the improved bound is that $Y_{s,b_1}$ and $Y_{s, b_2}$ are essentially independent when $|b_1-b_2|_{L}\gg \ell_s$. Then the additional factor of $\frac{\ell_s}{\ell_t}$ comes from the CLT-like cancellation of the left-hand side. Here we explain the source of this independence. We claim that, heuristically, the resolvent entries $G_{s,xy}$ in the $\ell_s$ range around the $(b, b)$ block of the matrix (meaning that $|[x]-b|_{L}+|[y]-b|_{L} \le W^{\tau}\ell_s$) depend only on the entries of $H$ in the $\ell_s$ range around $b$. Indeed, we can compute the correlations of $H_{ij}$ with $G_{xy}$ by integration by parts:
\begin{equation}\label{eq:ibp-heuristic}
    \E H_{ij} G_{xy} = S_{ij}\E G_{xj} G_{iy},\quad \E \bar{H}_{ij} G_{xy} = S_{ij}\E G_{xi} G_{jy}.
\end{equation}
Due to the decay property of the resolvent entries \eqref{Eq:Gdecay_w}, the correlations \eqref{eq:ibp-heuristic} are $O(W^{-D})$ for any $D>0$ when {$|[i]-b|_{L}+|[j]-b|_{L}\gg W^{\tau} \ell_s$}. {This (heuristically) suggests that $G_{xy}$ and $G_{x'y'}$ are almost independent if $(x,y)$ is far from $(x',y')$.} From the definition \eqref{clt-yform} of $Y_{s,b}$, we see that the random functions $Y_{s,b_1}$ and $Y_{s,b_2}$ depend only on the resolvent entries in the $\ell_s$ range of $b_1$ and $b_2$ blocks, respectively, up to a $O(W^{-D})$ error term. 

Thus, the randomness in $Y_{s,b_1}$ and $Y_{s,b_2}$ essentially comes from two disjoint subsets of the entries of $H$ in case $|b_1-b_2|_{L}\gg {W^{\tau}}\ell_s$. This implies their {(almost)} independence. 

In order to make this argument rigorous, we compute $\E Y_{s,b_1} Y_{s,b_2}$ by interpolating between this correlation and $\E Y'_{s,b_1} Y_{s,b_2}$, where $Y'_{s,b_1}$ is an independent copy of $Y_{s,b_1}$. Interpolation is achieved by viewing $Y_{s,b_1}$ as a function of entries of $H$ and replacing the entries by their independent copies one by one. Then on a single step of the interpolation the change in the correlation is bounded in one of the two ways. One option is that the replaced entry $H_{ij}$ is located far away from $b_1$ ($|[i]-b_1|_{L}+|[j]-b_1|_{L}\gg {W^{\tau}}\ell_s$), then the change in the function $Y_{s,b_1}$ is $O(W^{-D})$. The other option is that $H_{ij}$ is far away from $b_2$ ($|[i]-b_2|_{L}+|[j]-b_2|_{L}\gg {W^{\tau}}\ell_s$). In this case, we Taylor expand $Y_{s,b_1}$ in $H_{ij}$ variable and use the bounds on the correlations of $H_{ij}$ with $Y_{s,b_2}$ explained below \eqref{eq:ibp-heuristic}. {A similar argument controls higher moments of the left-hand side of \eqref{clt-lemma-final-result}.} We present the full proof with details below.

\begin{proof}
Throughout this argument, we will let $G_{s}=(\sqrt{s}H-z_{s})^{-1}$; since the final conclusion \eqref{clt-yform} is a statement about the law of $G_{s}$, this convention is okay. We compute $2p$-th moments. We have 
\begin{align}
\E\left|\sum_{b}Z_{ab}(Y_{s,b}-\E Y_{s,b})\right|^{2p}&=\sum_{b_{1},\ldots,b_{2p}}\prod_{k=1}^{p}Z_{ab_{k}}\prod_{k=p+1}^{2p}\overline{Z}_{ab_{k}}\E\left[\prod_{k=1}^{p}(Y_{s,b_{k}}-\E Y_{s,b_{k}})\prod_{k=p+1}^{2p}(\overline{Y}_{s,b_{k}}-\E \overline{Y}_{s,b_{k}})\right].\label{clt-lemmamoment}
\end{align}
Suppose that there exists an index, say $b_{1}$, such that $\min_{k=2,\ldots,2p}|b_{1}-b_{k}|_{L}\geq W^{2\tau}\ell_{s}$. We claim that
\begin{align}
\left|\E\left[\prod_{k=1}^{p}(Y_{s,b_{k}}-\E Y_{s,b_{k}})\prod_{k=p+1}^{2p}(\overline{Y}_{s,b_{k}}-\E \overline{Y}_{s,b_{k}})\right]\right|\leq c_{D} W^{-D}\label{clt-lemmafar}
\end{align}
for any $D>0$, where $c_{D}$ depends only on $D$. To prove this, we first realize $Y_{s,b_{k}}=Y_{s,b_{k}}(H)$ as a function of the entries of $H$. Let $H'$ be an i.i.d. copy of $H$. For any indices $i\leq j$, define
\begin{align}
\Delta_{ij}Y_{s,b_{k}}&:=Y_{s,b_{k}}(H_{11},H_{12},\ldots,H_{ij},H_{i(j+1)}',\ldots,H_{NN}')\label{clt-delta}\\
&-Y_{s,b_{k}}(H_{11},H_{12},\ldots,H_{ij}',H_{i(j+1)}',\ldots,H_{NN}').\nonumber
\end{align}
We also set $Y'_{s,b_{k}}:=Y_{s,b_{k}}(H')$. (Here, and throughout this argument, indices $i,j$ are reserved for elements of $\{1,\ldots,N\}$, and they will only be used as indices for $H$ and $H'$. We will identify these indices with elements in $\Z_{WL}^{2}$ when we apply the $[\cdot]$ operation, but the exact choice of identification $\{1,\ldots,N\}\simeq \Z_{WL}^{2}$ is unimportant.) By telescoping sum, we have 
\begin{align}
\sum_{i\leq j}\Delta_{ij}Y_{s,b_{k}}=Y_{s,b_{k}}(H)-Y_{s,b_{k}}(H').\label{clt-telescope-general}
\end{align}
Now, for the $k=1$ factor in \eqref{clt-lemmafar}, we rewrite it as $Y_{s,b_{1}}-\E Y_{s,b_{1}}=Y_{s,b_{1}}-Y_{s,b_{1}}(H')+Y_{s,b_{1}}(H')-\E Y_{s,b_{1}}$. Using \eqref{clt-telescope-general} to rewrite $Y_{s,b_{1}}-Y_{s,b_{1}}(H')$ then gives
\begin{align}
&\E\left[\prod_{k=1}^{p}(Y_{s,b_{k}}-\E Y_{s,b_{k}})\prod_{k=p+1}^{2p}(\overline{Y}_{s,b_{k}}-\E \overline{Y}_{s,b_{k}})\right]\\
&=\E\left[(Y_{s,b_{1}}(H')-\E Y_{s,b_{1}})\prod_{k=2}^{p}(Y_{s,b_{k}}-\E Y_{s,b_{k}})\prod_{k=p+1}^{2p}(\overline{Y}_{s,b_{k}}-\E \overline{Y}_{s,b_{k}})\right]\label{clt-lemma'}\\
&+\sum_{i\leq j}\E\left[\Delta_{ij}Y_{s,b_{1}}\prod_{k=2}^{p}(Y_{s,b_{k}}-\E Y_{s,b_{k}})\prod_{k=p+1}^{2p}(\overline{Y}_{s,b_{k}}-\E \overline{Y}_{s,b_{k}})\right].\label{clt-lemmatelescope}
\end{align}
Of course, we can certainly restrict to $i,j$ such that $S_{ij}\neq0$ in \eqref{clt-lemmatelescope}; otherwise, $H_{ij}$ and $H_{ij}'$ are both $0$, and thus $\Delta_{ij}Y_{s,b_{1}}=0$ in this case. We will implicitly make this assumption for the rest of the proof of \eqref{clt-lemmafar}.

Since $Y'_{s,b_{1}}$ is independent of all the $Y_{s,b_{k}}$, we know that \eqref{clt-lemma'} is zero. Since the number of summands in \eqref{clt-lemmatelescope} is $O(W^{C})$ for some $C=O(1)$, in order to prove \eqref{clt-lemmafar}, it suffices to control each summand in \eqref{clt-lemmatelescope}. Next, consider the function $\rho\mapsto Y_{s,b_{1}}(H_{11},\ldots,\rho H_{ij},H_{i(j+1)}',\ldots,H_{NN}')$. We think of this as a function $f_{s,b_{1}}(\rho H_{ij},\rho\overline{H}_{ij})$, so that $f_{s,b_{1}}$ is smooth in its two inputs. Taylor expanding in $\rho$ gives
\begin{align}
Y_{s,b_{1}}(H_{11},\ldots,H_{ij},H_{i(j+1)}',\ldots,H_{NN}')&=f_{s,b_{1}}(0)+\sum_{\alpha=1}^{D_{1}}\frac{1}{\alpha!}\frac{d^{\alpha}}{d\rho^{\alpha}}f_{s,b_{1}}(0,0)\nonumber\\
&+\frac{1}{(D_{1}+1)!}\frac{d^{D_{1}+1}}{d\rho^{D_{1}+1}}f_{s,b_{1}}(\rho_{ij}H_{ij},\rho_{ij}\overline{H}_{ij}),\label{clt-y-taylor}
\end{align}
where $D_{1}>0$ is any large constant independent of $W$, and where $|\rho_{ij}|$ is random (and depends on all $H_{uv},H'_{uv}$ entries). Using the chain rule, we compute the following for any $\alpha\geq1$:
\begin{align}
\frac{d^{\alpha}}{d\rho^{\alpha}}f_{s,b_{1}}(\rho H_{ij},\rho\overline{H}_{ij})=\sum_{v=0}^{\alpha}\binom{\alpha}{v}\partial_{1}^{v}f_{s,b_{1}}(\rho H_{ij},\rho\overline{H}_{ij})\cdot\partial_{2}^{\alpha-v}f_{s,b_{1}}(\rho H_{ij},\rho\overline{H}_{ij})\cdot H_{ij}^{v}\overline{H}_{ij}^{\alpha-v}.\label{clt-f-total}
\end{align}
Similarly, for the function $\rho\mapsto Y_{s,b_{1}}(H_{11},\ldots,\rho H_{ij}',H_{i(j+1)}',\ldots,H_{NN}')=f_{s,b_{1}}(\rho H_{ij}',\rho\overline{H}_{ij}')$, we have 
\begin{align}
Y_{s,b_{1}}(H_{11},\ldots,H_{ij}',H_{i(j+1)}',\ldots,H_{NN}')&=f_{s,b_{1}}(0)+\sum_{\alpha=1}^{D_{1}}\frac{1}{\alpha!}\frac{d^{\alpha}}{d\rho^{\alpha}}f_{s,b_{1}}(0,0)\\
&+\frac{1}{(D_{1}+1)!}\frac{d^{D_{1}+1}}{d\rho^{D_{1}+1}}f_{s,b_{1}}(\rho_{ij}'H_{ij}',\rho_{ij}'\overline{H}_{ij}'),\label{clt-y'-taylor}
\end{align}
where $|\rho_{ij}'|\leq1$ is also random and depends on all $H_{uv},H_{uv}'$ entries. If we now combine \eqref{clt-y-taylor}, \eqref{clt-y'-taylor}, \eqref{clt-f-total}, and \eqref{clt-delta}, we deduce
\begin{align}
\Delta_{ij}Y_{s,b_{1}}&=\sum_{\alpha=1}^{D_{1}}\frac{1}{\alpha!}\sum_{v=0}^{\alpha}\binom{\alpha}{v}\partial_{1}^{v}f_{s,b_{1}}(0,0)\cdot\partial_{2}^{\alpha-v}f_{s,b_{1}}(0,0)\cdot\left[H_{ij}^{v}\overline{H}_{ij}^{\alpha-v}-(H_{ij}')^{v}(\overline{H}_{ij}')^{\alpha-v}\right]\\
&+\frac{1}{(D_{1}+1)!}\frac{d^{D_{1}+1}}{d\rho^{D_{1}+1}}f_{s,b_{1}}(\rho_{ij}H_{ij},\rho_{ij}\overline{H}_{ij})-\frac{1}{(D_{1}+1)!}\frac{d^{D_{1}+1}}{d\rho^{D_{1}+1}}f_{s,b_{1}}(\rho_{ij}'H_{ij}',\rho_{ij}'\overline{H}_{ij}').
\end{align}
For convenience, let $\Gamma_{s}$ denote the product of the two products in \eqref{clt-lemmatelescope}, so that
\begin{align*}
\Gamma_{s}:=\prod_{k=2}^{p}(Y_{s,b_{k}}-\E Y_{s,b_{k}})\prod_{k=p+1}^{2p}(\overline{Y}_{s,b_{k}}-\E \overline{Y}_{s,b_{k}}).
\end{align*}
We have 
\begin{align}
\E[\Delta_{ij}Y_{s,b_{1}}\cdot\Gamma_{s}]&=\sum_{\alpha=1}^{D_{1}}\sum_{v=0}^{\alpha}\frac{1}{\alpha!}\binom{\alpha}{v}\E\left[\partial_{1}^{v}f_{s,b_{1}}(0,0)\partial_{2}^{\alpha-v}f_{s,b_{1}}(0,0)\cdot\Gamma_{s}\cdot\left\{H_{ij}^{v}\overline{H}_{ij}^{\alpha-v}-(H_{ij}')^{v}(\overline{H}_{ij}')^{\alpha-v}\right\}\right]\nonumber\\
&+\frac{1}{(D_{1}+1)!}\E\left[\frac{d^{D_{1}+1}}{d\rho^{D_{1}+1}}f_{s,b_{1}}(\rho_{ij}H_{ij},\rho_{ij}\overline{H}_{ij})\cdot\Gamma_{s}\right]\nonumber\\
&-\frac{1}{(D_{1}+1)!}\E\left[\frac{d^{D_{1}+1}}{d\rho^{D_{1}+1}}f_{s,b_{1}}(\rho_{ij}H_{ij}',\rho_{ij}\overline{H}_{ij}')\cdot\Gamma_{s}\right].\label{clt-taylor}
\end{align}
We will estimate each term on the right-hand side of \eqref{clt-taylor}. For this, we need some auxiliary estimates. Let $G^{i,j}=(\sqrt{s}H^{i,j}-z_{s})^{-1}$, where $H^{i,j}=H-(1-\gamma_{ij})H_{ij}e_{i}e_{j}^{\dagger}-(1-\gamma_{ij})H_{ji}e_{j}e_{i}^{\dagger}$, be the resolvent $G_{s}$ except we replace the entries $H_{ij},H_{ji}$ by $\gamma_{ij}H_{ij},\gamma_{ij}H_{ji}$, respectively. We will only assume that $\gamma_{ij}=O(1)$; they can be random and depend on $H_{uv},H'_{uv}$ entries, and we will be interested in either $\gamma_{ij}=0$ or $\gamma_{ij}=\rho_{ij}$ from \eqref{clt-y-taylor}.
\begin{itemize}
\item By resolvent expansion, Gaussianity of $H$ entries, and the assumption $\|G_{s}\|_{\max}\prec1$ from \eqref{Gt_bound_flow}, we have $G^{i,j}_{xy}=(G_{s})_{xy}+\mathrm{O}_{\prec}(W^{-1/2}\|G^{i,j}\|_{\max})$. Taking a maximum over $a,b$ gives 
\begin{align}
\|G^{i,j}\|_{\max}\prec\|G_{s}\|_{\max}\prec1.\label{clt-gij-bound}
\end{align}
\item By the same resolvent expansion, for any $K_{0}>0$ large and any indices $x,y=1,\ldots,N$, we have 
\begin{align*}
G^{i,j}_{xy}&=(G_{s})_{xy}+\sum_{k=1}^{K_{0}}(G_{s}(X^{i,j}G_{s})^{k})_{xy}+\left(G^{i,j}(X^{i,j}G_{s})^{K_{0}+1}\right)_{xy},
\end{align*}
where $X^{i,j}$ is supported at $(i,j)$ and $(j,i)$ entries, and its entries are all $\mathrm{O}_{\prec}(S_{ij}^{1/2})$. We now compute
\begin{align}
(G_{s}(X^{i,j}G_{s})^{k})_{xy}&=\sum_{\substack{\{i_{1},i_{2}\}=\{i,j\}\\\ldots\\\{i_{2k-1},i_{2k}\}=\{i,j\}}}(G_{s})_{xi_{1}}X^{i,j}_{i_{1}i_{2}}(G_{s})_{i_{2}i_{3}}\ldots X^{i,j}_{i_{2k-1}i_{2k}}(G_{s})_{i_{2k}y},\label{clt-resolvent}\\
(G^{i,j}(X^{i,j}G_{s})^{k})_{xy}&=\sum_{\substack{\{i_{1},i_{2}\}=\{i,j\}\\\ldots\\\{i_{2k-1},i_{2k}\}=\{i,j\}}}G^{i,j}_{xi_{1}}X^{i,j}_{i_{1}i_{2}}(G_{s})_{i_{2}i_{3}}\ldots X^{i,j}_{i_{2k-1}i_{2k}}(G_{s})_{i_{2k}y}.
\end{align}
In each identity, the number of summands is bounded by a constant depending only on $k$. Thus, as long as $K_{0}$ is finite, we can use $\|X^{i,j}\|_{\max}\prec W^{-1/2}$ and $\|G^{i,j}\|_{\max}+\|G_{s}\|_{\max}\prec1$ to deduce that $(G^{i,j}(X^{i,j}G_{s})^{K_{0}+1})_{xy}\prec W^{-(K_{0}+1)/2}$. 

Next, fix any $\tau',D'>0$. We also fix $a,b\in\Z_{L}^{2}$ so that $x\in{\cal I}^{(2)}_{a}$ and $y\in{\cal I}^{(2)}_{b}$. We claim that if $|a-b|_{L}>W^{2\tau'}\ell_{s}$, where $|\cdot|_L$ is periodic distance on $\Z_{L}^{2}$, then $(G_{s}(X^{i,j}G_{s})^{k})_{xy}\prec W^{-D'}$. Indeed, take \eqref{clt-resolvent}. If $|a-b|_{L}>W^{2\tau'}\ell_{s}$ and $k=O(1)$, then at least one of the following must be true:
\begin{itemize}
\item There exists some $i_{n},i_{n+1}$ indices such that $|i_{n}-i_{n+1}|_{L}\geq W^{1+\tau'}\ell_{s}$; indeed, if $|a-b|_{L}>W^{2\tau'}\ell{s}$, then $|x-y|\geq W^{1+2\tau'}\ell_{s}$. If this pair of indices corresponds to the factor $X^{i,j}_{i_{n}i_{n+1}}$, then this factor is equal to $0$. If this pair of indices corresponds to the factor $(G_{s})_{i_{n}i_{n+1}}$, then by the assumption \eqref{Eq:Gdecay_w} and \eqref{GijGEX}, we have $|(G_{s})_{i_{n}i_{n+1}}|_{L}\prec W^{-D'}$. Moreover, since $\|X^{i,j}\|_{\max}+\|G_{s}\|_{\max}\prec1$, this implies that the corresponding summand on the right-hand side of \eqref{clt-resolvent} is $\mathrm{O}_{\prec}(W^{-D'})$.
\item We have either $|a-i_{1}|_{L}\geq W^{1+\tau'}\ell_{s}$ or $|i_{2k}-b|_{L}\geq W^{1+\tau'}\ell_{s}$. The same reasoning in the previous paragraph implies that the corresponding summand in \eqref{clt-resolvent} is $\mathrm{O}_{\prec}(W^{-D'})$.
\end{itemize}
Ultimately, we deduce that for any $\tau',D'>0$, we have $(G_{s}(X^{i,j}G_{s})^{k})_{xy}\prec W^{-D'}$ if $|a-b|_{L}\geq W^{\tau'}\ell_{s}$. Finally, since $|(G_{s})_{xy}|\prec W^{-D'}$ if $|a-b|_{L}\geq W^{\tau'}\ell_{s}$, we conclude that for any $\tau',D'>0$, we have
\begin{align}
\mathbf{1}[|a-b|_{L}\geq W^{\tau'}\ell_{s}]\sup_{x\in{\cal I}^{(2)}_{a}}\sup_{y\in{\cal I}^{(2)}_{b}}|G^{i,j}_{xy}|\prec W^{-D'}.\label{clt-gij-decay}
\end{align}
\item Our last auxiliary estimate is a bound on $\partial_{1}^{v}f_{s,b_{1}}$ and $\partial_{2}^{\alpha-v}f_{s,b_{1}}$ in \eqref{clt-f-total}. Recall that $f_{s,b_{1}}(\rho H_{ij},\rho \overline{H}_{ij})=Y_{s,b_{1}}(H_{11},\ldots,\rho H_{ij},H_{i(j+1)}',\ldots,H_{NN}')$, and recall that $Y_{s,b_{1}}$ has the form \eqref{clt-yform}. In particular, $\partial_{1}f_{s,b_{1}}$ is the derivative of $Y_{s,b_{1}}$ from \eqref{clt-yform} with respect to $H_{ij}$ (treating $H_{ij}$ and $\overline{H}_{ij}$ as separate variables), and $\partial_{2}f_{s,b_{1}}$ is the derivative of $Y_{s,b_{1}}$ in $\overline{H}_{ij}$. Each such derivative can only act on entries of $G_{s}$ or $G_{s}^{\dagger}$ in \eqref{clt-yform}; when it hits one such entry, it returns two entries of $G_{s}$ or $G_{s}^{\dagger}$. More precisely, we have $\partial_{H_{ij}}(G_{s})_{xy}=-(G_{s})_{xi}(G_{s})_{jy}$, and similar standard formulas hold for $\overline{H}_{ij}$ instead of $H_{ij}$ and/or $G_{s}^{\dagger}$ instead of $G_{s}$. Therefore, since $\gamma_{\textbf{c}}=O(W^{C})$ in \eqref{clt-yform} for some $C=O(1)$ by assumption, we deduce from \eqref{clt-gij-bound} that for any $\alpha=O(1)$ and $0\leq v\leq \alpha$ and $\gamma_{ij}=O(1)$, we have
\begin{align}
|\partial_{1}^{v}f_{s,b_{1}}(\gamma_{ij} H_{ij},\gamma_{ij}\overline{H}_{ij})|+|\partial_{2}^{\alpha-v}f_{s,b_{1}}(\gamma_{ij} H_{ij},\gamma_{ij}\overline{H}_{ij})|\prec W^{C}.\label{clt-f-total-bound-1}
\end{align}
If we instead bound entries of Green's functions that appear in $\partial_{1}^{u}f_{s,b_{1}}$ and $\partial_{2}^{\alpha-u}f_{s,b_{1}}$ using the deterministic bound $O(W^{C})$ for some $C=O(1)$, then we instead have the following deterministic bound for any $\alpha=O(1)$ and $0\leq v\leq \alpha$ and $\gamma_{ij}=O(1)$, where now $C_{D_{1}}$ depends on $D_{1}>0$:
\begin{align}
|\partial_{1}^{v}f_{s,b_{1}}(\gamma_{ij} H_{ij},\gamma_{ij}\overline{H}_{ij})|+|\partial_{2}^{\alpha-v}f_{s,b_{1}}(\gamma_{ij} H_{ij},\gamma_{ij}\overline{H}_{ij})|= O(W^{C_{D_{1}}}).\label{clt-f-total-bound-2}
\end{align}
\item We note that \eqref{clt-gij-bound}, \eqref{clt-gij-decay}, \eqref{clt-f-total-bound-1}, and \eqref{clt-f-total-bound-2} all hold if we replace $H_{ij},\overline{H}_{ij}$ by $H_{ij}',\overline{H}_{ij}'$, respectively. Indeed, the proofs of \eqref{clt-gij-bound}, \eqref{clt-gij-decay} use only $|H_{ij}|\prec S_{ij}^{1/2}$. The proof of \eqref{clt-f-total-bound-1} uses only the distribution of $H_{ij}$. The proof of \eqref{clt-f-total-bound-2} is deterministic in $H_{ij}$.
\end{itemize}
We will now use the bounds \eqref{clt-gij-bound}, \eqref{clt-gij-decay}, \eqref{clt-f-total-bound-1}, and \eqref{clt-f-total-bound-2} to control the second term on the right-hand side of \eqref{clt-taylor}. First, we recall that $\Gamma_{s}$ is the product of the $k\geq2$ factors in \eqref{clt-lemmatelescope}. By \eqref{clt-yform}, we know that $|\Gamma_{s}|=O(W^{C})$ deterministically for some $C=O(1)$ depending only on $p$. This gives
\begin{align}
&\left|\E\left[\frac{d^{D_{1}+1}}{d\rho^{D_{1}+1}}f_{s,b_{1}}(\rho_{ij}H_{ij},\rho_{ij}\overline{H}_{ij})\cdot\Gamma_{s}\right]\right|\nonumber\\
&\leq O(W^{C})\E\left[\left|\frac{d^{D_{1}+1}}{d\rho^{D_{1}+1}}f_{s,b_{1}}(\rho_{ij}H_{ij},\rho_{ij}\overline{H}_{ij})\right|\right]\nonumber\\
&\leq O(W^{C})\sum_{v=0}^{D_{1}+1}\E\left|\partial_{1}^{v}f_{s,b_{1}}(\rho_{ij} H_{ij},\rho_{ij}\overline{H}_{ij})\cdot\partial_{2}^{D_{1}+1-v}f_{s,b_{1}}(\rho_{ij} H_{ij},\rho_{ij}\overline{H}_{ij})\cdot H_{ij}^{v}\overline{H}_{ij}^{D_{1}+1-v}\right|,\label{clt-remainder-bound-1}
\end{align}
where the last line follows by \eqref{clt-f-total} and the triangle inequality. Now, fix any $D_{2}>0$ and any $0\leq v\leq D_{1}+1$. By \eqref{clt-f-total-bound-1} for $\alpha=D_{1}+1$, there exists a good event ${\cal E}$ such that $\mathbb{P}({\cal E})=1-O(W^{-D_{2}})$, and such that on ${\cal E}$, we have $|\partial_{1}^{v}f_{s,b_{1}}(\rho_{ij} H_{ij},\rho_{ij}\overline{H}_{ij})\cdot\partial_{2}^{D_{1}-v}f_{s,b_{1}}(\rho_{ij} H_{ij},\rho_{ij}\overline{H}_{ij})|=O(W^{2C})$ for some $C=O(1)$. We now estimate the expectation in the last line above by separating into ${\cal E}$ and its complement ${\cal E}^{c}$. On ${\cal E}$, we use the bound $|\partial_{1}^{v}f_{s,b_{1}}(\rho_{ij} H_{ij},\rho_{ij}\overline{H}_{ij})\cdot\partial_{2}^{D_{1}-v}f_{s,b_{1}}(\rho_{ij} H_{ij},\rho_{ij}\overline{H}_{ij})|=O(W^{2C})$. On the bad event ${\cal E}^{c}$, we use the deterministic bound \eqref{clt-f-total-bound-2} for $\alpha=D_{1}+1$. This implies that 
\begin{align}
&\E\left|\partial_{1}^{v}f_{s,b_{1}}(\rho_{ij} H_{ij},\rho_{ij}\overline{H}_{ij})\cdot\partial_{2}^{D_{1}+1-v}f_{s,b_{1}}(\rho_{ij} H_{ij},\rho_{ij}\overline{H}_{ij})\cdot H_{ij}^{v}\overline{H}_{ij}^{D_{1}+1-v}\right|\\
&\leq O(W^{2C})\E|H_{ij}^{v}\overline{H}_{ij}^{D_{1}+1-v}|+O(W^{2C_{D_{1}}})\E[\mathbf{1}_{{\cal E}^{c}}|H_{ij}^{v}\overline{H}_{ij}^{D_{1}+1-v}|]\\
&\leq O(W^{2C}W^{-(D_{1}+1)/2})+O(W^{2C_{D_{1}}})\mathbb{P}({\cal E}^{c})^{1/2}\\
&\leq O(W^{2C}W^{-(D_{1}+1)/2})+O(W^{2C_{D_{1}}}W^{-D_{2}/2}).
\end{align}
Fix any $D'>0$. We can choose $D_{1}>0$ large enough and then $D_{2}>0$ large enough depending on $D_{1}$ such that the last line above is $O(W^{-D'})$. In particular, combining this with \eqref{clt-remainder-bound-1} implies the following. For any $D_{3}>0$, there exists $D_{1}>0$ such that 
\begin{align}
\E\left[\frac{d^{D_{1}+1}}{d\rho^{D_{1}+1}}f_{s,b_{1}}(\rho_{ij}H_{ij},\rho_{ij}\overline{H}_{ij})\cdot\Gamma_{s}\right]=O(W^{-D_{3}}).\label{clt-remainder-bound}
\end{align}
The same argument provides the same bound on the last term in \eqref{clt-taylor}, i.e. for any $D_{3}>0$, there exists $D_{1}>0$ such that 
\begin{align}
\E\left[\frac{d^{D_{1}+1}}{d\rho^{D_{1}+1}}f_{s,b_{1}}(\rho_{ij}H'_{ij},\rho_{ij}\overline{H}'_{ij})\cdot\Gamma_{s}\right]=O(W^{-D_{3}}).\label{clt-remainder-bound-same}
\end{align}
We are left to control the first term on the right-hand side of \eqref{clt-taylor}, which is a sum over $\alpha,v$; fix any such $\alpha,v$. We first expand $H_{ij}^{v}\overline{H}_{ij}^{\alpha-v}$ as a linear combination of Hermite polynomials in $H_{ij},\overline{H}_{ij}$ of degree at most $\alpha$. We do the same to $(H_{ij}')^{v}(\overline{H}_{ij}')^{\alpha-v}$. In these expansions, the zero-th degree (constant) terms cancel each other out. Then, by Gaussian integration-by-parts, we get
\begin{align}
&\E\left[\partial_{1}^{v}f_{s,b_{1}}(0,0)\partial_{2}^{\alpha-v}f_{s,b_{1}}(0,0)\cdot\Gamma_{s}\cdot\left\{H_{ij}^{v}\overline{H}_{ij}^{\alpha-v}-(H_{ij}')^{v}(\overline{H}_{ij}')^{\alpha-v}\right\}\right]\\
&=\sum_{\substack{q_{1},q_{2}=0,\ldots,\alpha\\q_{1}+q_{2}>0}}c_{q_{1},q_{2}}\E\left[\left\{S_{ij}^{q_{1}+q_{2}}\partial_{H_{ij}}^{q_{1}}\partial_{\overline{H}_{ij}}^{q_{2}}\right\}\left(\partial_{1}^{v}f_{s,b_{1}}(0,0)\partial_{2}^{\alpha-v}f_{s,b_{1}}(0,0)\cdot\Gamma_{s}\right)\right]\\
&-\sum_{\substack{q_{1},q_{2}=0,\ldots,\alpha\\q_{1}+q_{2}>0}}c_{q_{1},q_{2}}\E\left[\left\{S_{ij}^{q_{1}+q_{2}}\partial_{H_{ij}'}^{q_{1}}\partial_{\overline{H}_{ij}'}^{q_{2}}\right\}\left(\partial_{1}^{v}f_{s,b_{1}}(0,0)\partial_{2}^{\alpha-v}f_{s,b_{1}}(0,0)\cdot\Gamma_{s}\right)\right].
\end{align}
Above, $c_{q_{1},q_{2}}=O(1)$ are deterministic constants. The last line vanishes, since nothing inside the expectation depends on $H_{ij}',\overline{H}_{ij}'$. For the first term on the right-hand side above, we note that because we evaluate the $f_{s,b_{1}}$ derivatives at $(0,0)$, they have no dependence on $H_{ij},\overline{H}_{ij}$, so the derivatives act only on $\Gamma_{s}$ (which we recall to be the product of the $k\geq2$ factors in \eqref{clt-lemmatelescope}). Thus, we have 
\begin{align}
&\E\left[\partial_{1}^{v}f_{s,b_{1}}(0,0)\partial_{2}^{\alpha-v}f_{s,b_{1}}(0,0)\cdot\Gamma_{s}\cdot\left\{H_{ij}^{v}\overline{H}_{ij}^{\alpha-v}-(H_{ij}')^{v}(\overline{H}_{ij}')^{\alpha-v}\right\}\right]\nonumber\\
&=\sum_{\substack{q_{1},q_{2}=0,\ldots,\alpha\\q_{1}+q_{2}>0}}c_{q_{1},q_{2}}\E\left[\partial_{1}^{v}f_{s,b_{1}}(0,0)\partial_{2}^{\alpha-v}f_{s,b_{1}}(0,0)\cdot\left\{S_{ij}^{q_{1}+q_{2}}\partial_{H_{ij}}^{q_{1}}\partial_{\overline{H}_{ij}}^{q_{2}}\right\}\Gamma_{s}\right].\label{clt-ibp}
\end{align}
We now study further the derivatives of $\Gamma_{s}$ in \eqref{clt-ibp}; again, recall that $\Gamma_{s}$ is the product of $k\geq2$ factors in \eqref{clt-lemmatelescope}. Thus, by the Leibniz rule, we have the following for some deterministic constants $c_{\textbf{w}_{1},\textbf{w}_{2}}=O(1)$:
\begin{align}
&\left\{\partial_{H_{ij}}^{q_{1}}\partial_{\overline{H}_{ij}}^{q_{2}}\right\}\Gamma_{s}\label{clt-ibp-leibniz}\\
&=\sum_{\substack{\mathbf{w}_{1}:=(w_{1,2},\ldots,w_{1,2p})\in\Z_{\geq0}^{2p-1}\\w_{1,2}+\ldots+w_{1,2p}=q_{1}\\\mathbf{w}_{2}:=(w_{2,2},\ldots,w_{2,2p})\in\Z_{\geq0}^{2p-1}\\w_{2,2}+\ldots+w_{2,2p}=q_{2}}}c_{\textbf{w}_{1},\textbf{w}_{2}}\prod_{k=2}^{p}\partial_{H_{ij}}^{w_{1,k}}\partial_{\overline{H}_{ij}}^{w_{2,k}}(Y_{s,b_{k}}-\E Y_{s,b_{k}})\cdot\prod_{k=p+1}^{2p}\partial_{H_{ij}}^{w_{1,k}}\partial_{\overline{H}_{ij}}^{w_{2,k}}(\overline{Y}_{s,b_{k}}-\E\overline{Y}_{s,b_{k}}).\nonumber
\end{align}
Now, we use the form \eqref{clt-yform} for $Y_{s,b_{k}}$, and suppose that $w_{1,k}+w_{2,k}>0$. Since the $C_{b,\textbf{x},\textbf{y}}$ coefficients therein are deterministic, we have
\begin{align}
\partial_{H_{ij}}^{w_{1,k}}\partial_{\overline{H}_{ij}}^{w_{2,k}}(Y_{s,b_{k}}-\E Y_{s,b_{k}})&=\sum_{\textbf{x},\textbf{y}\in(\Z_{WL}^{2})^{u}}C_{b_{k},\textbf{x},\textbf{y}}\sum_{\substack{\boldsymbol{\gamma}_{1}=(\gamma_{11},\ldots,\gamma_{1u})\\\gamma_{11}+\ldots+\gamma_{1u}=w_{1,k}\\\boldsymbol{\gamma}_{2}=(\gamma_{21},\ldots,\gamma_{2u})\\\gamma_{21}+\ldots+\gamma_{2u}=w_{2,k}}}\prod_{r=1}^{u}\partial_{H_{ij}}^{\gamma_{1r}}\partial_{\overline{H}_{ij}}^{\gamma_{2r}}G_{s}(\sigma_{r})_{x_{r}y_{r}}.\label{clt-ibp-leibniz-formula}
\end{align}
(The expectation on the left-hand side above vanishes under the derivatives.) Now, recall $w_{1,k}+w_{2,k}>0$. In this case, there must exist a factor of $G_{s}(\sigma)_{\beta i}$ and a factor of $G_{s}(\sigma')_{\beta j}$ (for some $\sigma,\sigma'\in\{+,-\}$), where $\beta=x_{r}$ or $\beta=y_{r}$ for some $r$. This follows by differentiation formulas for the resolvent, e.g. $\partial_{H_{ij}}(G_{s})_{xy}=-(G_{s})_{xi}(G_{s})_{jy}$. Given this information, suppose now that $|[i]-b_{k}|_{L}+|[j]-b_{k}|_{L}\geq 3W^{\tau}\ell_{s}$; without loss of generality, suppose that $|[i]-b_{k}|_{L}\geq 3W^{\tau}\ell_{s}/2$. Take the factor $G_{s}(\sigma)_{\beta i}$, where $\beta=x_{r}$ or $\beta=y_{r}$ for some $r$. If $|[\beta]-b_{k}|\geq W^{\tau}\ell_{s}$, then the corresponding coefficient $C_{b_{k},\textbf{x},\textbf{y}}$ vanishes by assumptions. If $|[\beta]-b_{k}|_{L}\leq W^{\tau}\ell_{s}$, then we have $|[i]-[\beta]|_{L}\geq W^{\tau}\ell_{s}/2$. In this case, we can use the fast decay property for $G_{s}$ (see \eqref{Eq:Gdecay_w} and \eqref{GijGEX}) to deduce the bound $G_{s}(\sigma)_{\beta i}\prec W^{-D_{0}}$ for any $D_{0}>0$. Finally, we can bound the remaining factors in \eqref{clt-ibp-leibniz-formula} by $\mathrm{O}_{\prec}(N^{C'})$ for some $C'=O(1)$ (this uses the assumption \eqref{Gt_bound_flow}). We deduce
\begin{align}
\mathbf{1}[w_{1,k}+w_{2,k}>0]\mathbf{1}[|[i]-b_{k}|_{L}+|[j]-b_{k}|_{L}\geq 3W^{\tau}\ell_{s}]|\partial_{H_{ij}}^{w_{1,k}}\partial_{\overline{H}_{ij}}^{w_{2,k}}(Y_{s,b_{k}}-\E Y_{s,b_{k}})|\prec W^{-D'}\label{clt-ibp-decay}
\end{align}
for any $D'=O(1)$. On the other hand, even if $w_{1,k}+w_{2,k}=0$, we can use the the bounds $\|C\|_{\max}=O(W^{C'})$ and $\|G_{s}\|_{\max}=O(W^{C'})$ for some $C'=O(1)$ to get
\begin{align}
|\partial_{H_{ij}}^{w_{1,k}}\partial_{\overline{H}_{ij}}^{w_{2,k}}(Y_{s,b_{k}}-\E Y_{s,b_{k}})|\prec W^{C'}\label{clt-ibp-other-factor-bound}
\end{align}
for some (possibly different) $C'=O(1)$. We emphasize that the previous two bounds are true if we replace $Y_{s,b_{k}}$ by its complex conjugate as well. We also clarify that the bound $\|G_{s}\|_{\max}=O(W^{C'})$ comes from the assumption that $s\leq 1-N^{-1+\delta}$ for some $\delta>0$, and thus $\eta_{s}=O(N^{1-\delta})$.

Now, take \eqref{clt-ibp-leibniz}, and recall that $q_{1}+q_{2}>0$ therein; see \eqref{clt-ibp} for this constraint. Thus, at least one pair $(w_{1,k},w_{2,k})$ must satisfy $w_{1,k}+w_{2,k}>0$. For this pair, we bound the corresponding factor using \eqref{clt-ibp-decay}. The other factors in \eqref{clt-ibp-leibniz} can be bounded using \eqref{clt-ibp-decay} and \eqref{clt-ibp-other-factor-bound}. Ultimately, we deduce
\begin{align}
\mathbf{1}\left[\min_{k=2,\ldots,2p}(|[i]-b_{k}|_{L}+|[j]-b_{k}|_{L})\geq 3W^{\tau}\ell_{s}\right]\cdot \left\{\partial_{H_{ij}}^{q_{1}}\partial_{\overline{H}_{ij}}^{q_{2}}\right\}\Gamma_{s}&\prec W^{-D'},\label{clt-ibp-bound1}\\
\left\{\partial_{H_{ij}}^{q_{1}}\partial_{\overline{H}_{ij}}^{q_{2}}\right\}\Gamma_{s}&=O(W^{C}),\label{clt-ibp-bound2}
\end{align}
where $D'=O(1)$ is arbitrary and $C=O(1)$ is fixed. This bounds the last factor in the expectation in \eqref{clt-ibp}. 

To control $\partial_{1}^{v}f_{s,b_{1}}(0,0)\partial_{2}^{\alpha-v}f_{s,b_{1}}(0,0)$, we note that $\partial_{1}^{v}f_{s,b_{1}}(0,0)\partial_{2}^{\alpha-v}f_{s,b_{1}}(0,0)$ has the same form as the summand on the right-hand side of \eqref{clt-ibp-leibniz} (or its complex conjugate) with $p=1$ and $w_{1,2}+w_{2,2}>1$, except we evaluate the Green's functions at $H_{ij},\overline{H}_{ij}=0$. (Indeed, recall from \eqref{clt-y-taylor} that $f_{s,b_{1}}(H_{ij},\overline{H}_{ij})=Y_{s,b_{1}}$.) Therefore, if we instead use the Green's function estimates \eqref{clt-gij-bound} and \eqref{clt-gij-decay} to account for setting $H_{ij},\overline{H}_{ij}=0$, then for any $\tau',D'>0$ and for some $C=O(1)$, we have the same bounds
\begin{align}
\mathbf{1}[|[i]-b_{1}|_{L}+|[j]-b_{1}|_{L}\geq 3W^{\tau}\ell_{s}]\cdot\partial_{1}^{v}f_{s,b_{1}}(0,0)\partial_{2}^{\alpha-v}f_{s,b_{1}}(0,0)&\prec W^{-D'},\label{clt-ibp-bound3}\\
\partial_{1}^{v}f_{s,b_{1}}(0,0)\partial_{2}^{\alpha-v}f_{s,b_{1}}(0,0)&=O(W^{C}).\label{clt-ibp-bound4}
\end{align}
Now, consider \eqref{clt-ibp}, and recall the assumption $\min_{k=2,\ldots,2p}|b_{k}-b_{1}|_{L}\geq W^{2\tau}\ell_{s}$ that was stated before \eqref{clt-lemmafar}. This implies that for any $i,j$ indices, we either have $|[i]-b_{1}|_{L}+|[j]-b_{1}|_{L}\geq 3W^{\tau}\ell_{s}$, or $|[i]-b_{k}|_{L}+|[j]-b_{k}|_{L}\geq 3W^{\tau}\ell_{s}$ for all $k=2,\ldots,2p$. So, by \eqref{clt-ibp-bound1}, \eqref{clt-ibp-bound2}, \eqref{clt-ibp-bound3}, and \eqref{clt-ibp-bound4}, we get the following. For any $D_{4},D_{5}>0$ and any $i,j$ indices, there exists another good event ${\cal E}_{good}$ such that $\mathbb{P}({\cal E}_{good})=1-O(W^{-D_{4}})$ and
\begin{align}
\mathbf{1}_{{\cal E}_{good}}\partial_{1}^{v}f_{s,b_{1}}(0,0)\partial_{2}^{\alpha-v}f_{s,b_{1}}(0,0)\cdot\left\{S_{ij}^{q_{1}+q_{2}}\partial_{H_{ij}}^{q_{1}}\partial_{\overline{H}_{ij}}^{q_{2}}\right\}\Gamma_{s}=O(W^{-D_{5}}).\label{clt-ibp-bound5}
\end{align}
When we estimate the expectation in \eqref{clt-ibp}, we separate into the good event ${\cal E}_{good}$ and the bad event ${\cal E}_{good}^{c}$. On the former, we use \eqref{clt-ibp-bound5}. On the bad event, we use \eqref{clt-ibp-bound2} and \eqref{clt-ibp-bound4}. This gives
\begin{align}
&\E\left[\partial_{1}^{v}f_{s,b_{1}}(0,0)\partial_{2}^{\alpha-v}f_{s,b_{1}}(0,0)\cdot\left\{S_{ij}^{q_{1}+q_{2}}\partial_{H_{ij}}^{q_{1}}\partial_{\overline{H}_{ij}}^{q_{2}}\right\}\Gamma_{s}\right]\\
&=O(W^{-D_{5}})+O(W^{2C})\mathbb{P}({\cal E}_{good}^{c})\leq O(W^{-D_{5}})+O(W^{2C-D_{4}}).
\end{align}
We plug this bound into \eqref{clt-ibp}, which we then plug into the first term on the right-hand side of \eqref{clt-taylor}. Since $D_{4},D_{5}>0$ are any large constants independent of $W$ in the previous display, this shows that for any $D>0$ independent of $W$, the first term on the right-hand side of \eqref{clt-taylor} is $O(W^{-D})$. Combining this with \eqref{clt-remainder-bound} and \eqref{clt-remainder-bound-same} and \eqref{clt-taylor} ultimately gives $\E[\Delta_{ij}Y_{s,b_{1}}]=O(W^{-D})$ for any $D>0$ independent of $W$. By \eqref{clt-lemmatelescope}, this completes the proof of \eqref{clt-lemmafar}.

{We now finish the proof. Take any large $D_{0}>0$. By \eqref{clt-lemmafar}, if we give up an error term of $O(W^{-D_{0}})$, we can restrict to indices $b_{1},\ldots,b_{2p}$ so that for every $i$, we have $\min_{k\neq i}|b_{i}-b_{k}|_{L}\leq W^{2\tau}\ell_{s}$. In particular, for any $j=1,\ldots,p$, choose $j$ many ``free indices", and label them $b_{1},\ldots,b_{j}$. The remaining indices $b_{j+1},\ldots,b_{2p}$ must be assigned to one of the $b_{1},\ldots,b_{j}$ in the following sense. If $b_{j+1}$ is assigned to $b_{1}$, then $|b_{j+1}-b_{1}|_{L}\leq W^{3\tau}\ell_{s}$. Thus, from \eqref{clt-lemmamoment} and \eqref{clt-lemmafar}, we have
\begin{align}
\E\left|\sum_{b}Z_{ab}(Y_{s,b}-\E Y_{s,b})\right|^{2p}&\prec\sum_{j=1}^{p}\sum_{\substack{k_{1},\ldots,k_{j}\geq1\\k_{1}+\ldots+k_{j}=2p-j}}\prod_{i=1}^{j}\sum_{b_{i}}|Z_{ab_{i}}|\left(\sum_{|b_{i,2}-b_{i}|_{L}\leq W^{3\tau}\ell_{s}}|Z_{ab_{i,2}}|\right)^{k_{i}}\cdot \E(\|Y_{s}\|_{\max}^{2p})\label{eq:cltmomentfinal}\\
&+O(W^{-D}\max_{b}|Z_{ab}|^{2p}).\nonumber
\end{align}
We now focus on case $1$, i.e. the proof of \eqref{clt-lemma-final-result}. By our assumption on the range of $Z_{ab}$, we can restrict each $b_{i}$ summation to $O(W^{C\tau}\ell_{t}^{2})$ many terms. Thus, the total number of summands in the first term on the right-hand side of \eqref{eq:cltmomentfinal} is $O(W^{2pC\tau}\ell_{t}^{2p}\ell_{s}^{2p})$. (There are at most $p$-many $b_{i}$-sums, and a total of $2p$-sums overall.) Then we bound entries of $Z$ by its max-norm. Also, recall the assumption $\|Y_{s}\|_{\max}\prec\Lambda$, and note that $\|Y_{s}\|_{\max}=O(W^{C})$ for some $C=O(1)$, which follows from \eqref{clt-yform} and $s\leq N^{-1+\delta}$ for some $\delta>0$. This implies that $\E(\|Y_{s}\|_{\max}^{2p})\prec\Lambda +W^{-D'}$ for any $D'=O(1)$. Ultimately, we deduce (for any $D=O(1)$) that
\begin{align*}
\E\left|\sum_{b}Z_{ab}(Y_{s,b}-\E Y_{s,b})\right|^{2p} \prec W^{C\tau}\ell_{t}^{2p}\ell_{s}^{2p}\max_{b}|Z_{ab}|^{2p}\cdot\Lambda^{2p}+W^{-D}\max_{b}|Z_{ab}|^{2p}.
\end{align*}
Since $D,p=O(1)$ are arbitrary, we obtain the desired result \eqref{clt-lemma-final-result} in case $1$.

We now move to case $2$. On the right-hand side of \eqref{eq:cltmomentfinal}, we fix some sequence of $k_{1},\ldots,k_{p}$ and some $i\in\{1,\ldots,p\}$. Consider the object
\begin{align}
\sum_{b_{i}}|Z_{ab_{i}}|\cdot \left(\sum_{|b_{i,2}-b_{i}|_{L}\leq W^{3\tau}\ell_{s}}|Z_{ab_{i,2}}|\right)^{k_{i}}.\label{clt-lemma-case-2}   
\end{align}
We emphasize that $k_{i}\geq1$. We now rewrite \eqref{clt-lemma-case-2} as follows:
\begin{align*}
\sum_{b_{i}}|Z_{ab_{i}}|\sum_{|b_{i,2}-b_{i}|_{L}\leq W^{3\tau}\ell_{s}}|Z_{ab_{i,2}}|\cdot\left(\sum_{|b_{i,2}-b_{i}|_{L}\leq W^{3\tau}\ell_{s}}|Z_{ab_{i,2}}|\right)^{k_{i}-1}.
\end{align*}
Since $k_{i}\geq1$, we can use the a priori estimate of $|Z_{ab}|\prec(1+|a-b|_{L})^{-1}$ to show that the last factor in the previous display is $\mathrm{O}_{\prec}(W^{3\tau(k_{i}-1)}\ell_{s}^{k_{i}-1})$. On the other hand, we can use this same bound on $|Z_{ab}|$ to get
\begin{align*}
\sum_{b_{i}}|Z_{ab_{i}}|\sum_{|b_{i,2}-b_{i}|_{L}\leq W^{3\tau}\ell_{s}}|Z_{ab_{i,2}}|&\prec\sum_{b_{i}}\frac{1}{|a-b_{i}|^{2}+1}\sum_{|b_{i,2}-b_{i}|_{L}\leq W^{3\tau}\ell_{s}}\frac{|a-b_{i}|_{L}+1}{|a-b_{i,2}|_{L}+1}\\
&\leq\sum_{b_{i}}\frac{1}{|a-b_{i}|^{2}+1}\sum_{|b_{i,2}-b_{i}|_{L}\leq W^{3\tau}\ell_{s}}\Big(1+\frac{|b_{i}-b_{i,2}|_{L}}{1+|a-b_{i,2}|_{L}}\Big)\\
&\prec W^{6\tau}\ell_{s}^{2} \cdot \sum_{b_{i}}\frac{1}{|a-b_{i}|^{2}+1}\prec W^{6\tau}\ell_{s}^{2}.
\end{align*}
Therefore, we deduce
\begin{align}
\sum_{b_{i}}|Z_{ab_{i}}|\sum_{|b_{i,2}-b_{i}|_{L}\leq W^{3\tau}\ell_{s}}|Z_{ab_{i,2}}|\cdot\left(\sum_{|b_{i,2}-b_{i}|_{L}\leq W^{3\tau}\ell_{s}}|Z_{ab_{i,2}}|\right)^{k_{i}-1}&\prec W^{6\tau}W^{3\tau(k_{i}-1)}\ell_{s}^{k_{i}+1}.
\end{align}
The benefit of this bound is that there is no dependence on the range of $Z$. Now, when we take a product over all $i=1,\ldots,j$, we deduce
\begin{align*}
\prod_{i=1}^{j}\sum_{b_{i}}|Z_{ab_{i}}|\cdot \left(\sum_{|b_{i,2}-b_{i}|_{L}\leq W^{3\tau}\ell_{s}}|Z_{ab_{i,2}}|\right)^{k_{i}}\prec \prod_{i=1}^{j} W^{C\tau k_{i}}\ell_{s}^{k_{i}+1}\prec W^{(2p-j)C\tau}\ell_{s}^{2p-j}\cdot\ell_{s}^{j}\prec W^{2pC\tau}\ell_{s}^{2p}
\end{align*}
since the $k_{i}$ sum to $j$ over $i=1,\ldots,2p-j$. In particular, if we plug the previous display into the first term on the right-hand side of \eqref{eq:cltmomentfinal} and use $|Z_{ab}|\prec1$, we obtain the following for any $D=O(1)$:
\begin{align*}
\E\left|\sum_{b}Z_{ab}(Y_{s,b}-\E Y_{s,b})\right|^{2p}\prec W^{2pC\tau}\ell_{s}^{2p}\cdot \E(\|Y_{s}\|_{\max}^{2p})+W^{-D}.
\end{align*}
If we again use $\|Y_{s}\|_{\max}\prec \Lambda$ and $\|Y_{s}\|_{\max}=O(W^{C})$ for some $C=O(1)$, then the previous display yields the desired estimate \eqref{clt-lemma-final-result2} (since $D,p=O(1)$ are arbitrary). This completes the proof.}
\end{proof}

\section{Proof of Lemma \ref{lem_propTH}}\label{sec_TH}
The matrix $S^{(B)}$ is the transition matrix for a symmetric random walk on $\Z_{L}^{2}$. Thus, it is a real-symmetric matrix with operator norm $1$. This implies that the random walk representation below converges absolutely for all $|\xi|<1$:
\begin{align}
(\Theta^{(B)}_{\xi})_{ab}&=\sum_{k=0}^{\infty}\xi^{k}[(S^{B})^{k}]_{ab}=\delta_{ab}+\sum_{k=1}^{\infty}\xi^{k}[(S^{B})^{k}]_{ab}.\label{theta_rw}
\end{align}
Since $S^{(B)}$ satisfies translation invariance and symmetry, by \eqref{theta_rw}, the same is true for $\Theta^{(B)}$. The commutativity in point 3 of Lemma \ref{lem_propTH} also follows by \eqref{theta_rw}.

Thus, we are left to prove properties 5 and 6 in Lemma \ref{lem_propTH}. Let $k\mapsto Y_{k}$ be a discrete-time lazy simple random walk on $\Z^{2}$, so that $\mathbb{P}(Y_{k+1}=\beta|Y_{k}=\alpha)=1/5$ if and only if $|\beta-\alpha|\leq1$, where $|\cdot|$ here denotes $L^{1}$ distance on $\Z^{2}$. Let $p_{k}(\alpha,\beta)=\mathbb{P}(Y_{k}=\beta|Y_{0}=\alpha)$. We have the standard representation
\begin{align}
[(S^{(B)})^{k}]_{ab}=\sum_{\beta\sim b}p_{k}(a,\beta),
\end{align}
where $\beta\sim b$ means that $\beta,b$ are equal after projecting $\Z^{2}\to\Z_{L}^{2}$. We now present bounds for $p_{k}$ that will be used to bound \eqref{theta_rw}. First, the simple random walk $Y_{k}$ is a discrete-time martingale whose step sizes are $O(1)$ deterministically. Thus, we can use the Azuma-Hoeffding martingale inequality to deduce
\begin{align}
p_{k}(\alpha,\beta)\leq \exp\Big(-\frac{|\alpha-\beta|^{2}}{Ck}\Big)
\end{align}
for some constant $C=O(1)$. Moreover, a standard computation for $Y_{k}$ in terms of binomial coefficients and Stirling's formula gives $p_{k}(\alpha,\beta)\leq C/(1+k)$. 

Now, fix any $k\geq1$, and let $k_{1},k_{2}$ be non-negative integers such that $k_{1}+k_{2}=k$ and $k_{1},k_{2}\geq (k/2)-1$. By the semigroup property, we have
\begin{align*}
p_{k}(\alpha,\beta)=\sum_{\zeta\in\Z}p_{k_{1}}(\alpha,\zeta)p_{k_{2}}(\zeta,\beta)=\sum_{|\zeta-\alpha|\geq|\alpha-\beta|/2}p_{k_{1}}(\alpha,\zeta)p_{k_{2}}(\zeta,\beta)+\sum_{|\zeta-\beta|\geq|\alpha-\beta|/2}p_{k_{1}}(\alpha,\zeta)p_{k_{2}}(\zeta,\beta).
\end{align*}
We control the first term on the far right-hand side by the on-diagonal estimate for $p_{k_{2}}$ and the off-diagonal estimate for $p_{k_{1}}$. Since $1+k_{j}\leq 1+k\leq 2(1+k_{j})$ for $j=1,2$, when we apply these estimates, we can insert $k$ for $k_{1},k_{2}$. Ultimately, this gives
\begin{align*}
\sum_{|\zeta-\alpha|\geq|\alpha-\beta|/2}p_{k_{1}}(\alpha,\zeta)p_{k_{2}}(\zeta,\beta)\leq \frac{C}{1+k}\sum_{|\zeta-\alpha|\geq|\alpha-\beta|/2}e^{-|\alpha-\zeta|^{2}/Ck}\leq \frac{C'}{1+k}e^{-|\alpha-\beta|^{2}/C'k},
\end{align*}
where the last bound follows because of $|\zeta-\alpha|\geq|\alpha-\beta|/2$. We can reverse the roles of $\alpha,\beta$ as well to get
\begin{align*}
\sum_{|\zeta-\beta|\geq|\alpha-\beta|/2}p_{k_{1}}(\alpha,\zeta)p_{k_{2}}(\zeta,\beta)\leq \frac{C'}{1+k}e^{-|\alpha-\beta|^{2}/C'k}.
\end{align*}
Therefore, we deduce the bound
\begin{align*}
p_{k}(\alpha,\beta)\leq\frac{C}{1+k}e^{-|\alpha-\beta|^{2}/Ck}.
\end{align*}
Now plug the previous display into \eqref{theta_rw}. We also use the bound $|\xi|^{k}\leq\exp(-k|1-\xi|)$. This gives
\begin{align}
(\Theta^{(B)}_{\xi})_{ab}&\leq \delta_{ab}+C\sum_{k=1}^{\infty}e^{-k|1-\xi|}\sum_{\beta\sim b}\frac{1}{1+k}e^{-|a-\beta|^{2}/Ck}\\
&\leq \delta_{ab}+C'\int_{1}^{\infty}e^{-s|1-\xi|}\sum_{\beta\sim b}\frac{1}{s}e^{-\frac{|a-\beta|^{2}}{C's}}ds.\label{theta_rw_int}
\end{align}
We will break the integration over $[1,\infty)$ into integration over $[1,L^{2}]$ and $[L^{2},\infty)$. On the former domain, consider the sum over $\beta\sim b$. Let $\beta_{0}$ be a member of this equivalence class that is closest to $a$, so $|a-\beta_{0}|^{2}=|a-b|_{L}^{2}$. Every other $\beta\sim b$ has the form $\beta=\beta_{0}+L\gamma$, where $\gamma\in\Z^{2}$. Therefore, if $s\leq L^{2}$, we have
\begin{align}
\sum_{\beta\sim b}\frac{1}{s}e^{-\frac{|a-\beta|^{2}}{C's}}&\leq\frac{1}{s}e^{-|a-b|_{L}^{2}/C's}\sum_{\gamma\in\Z^{2}}e^{-cL^{2}|\gamma|^{2}/C'L^{2}}\leq \frac{C}{s}e^{-|a-b|_{L}^{2}/C's}\label{eq:hke_short_time}
\end{align}
for some $C=O(1)$. Thus, the contribution of the integral in \eqref{theta_rw_int} over $[1,L^{2}]$ is bounded above by
\begin{align*}
\int_{1}^{L^{2}}e^{-s|1-\xi|}\frac{1}{s}e^{-|a-b|_{L}^{2}/C's}ds&=\int_{1}^{L^{2}\wedge(|a-b|_{L}|1-\xi|^{-1/2})}e^{-s|1-\xi|}\frac{1}{s}e^{-|a-b|_{L}^{2}/C's}ds\\
&+\int_{L^{2}\wedge(|a-b|_{L}|1-\xi|^{-1/2})}^{L^{2}}e^{-s|1-\xi|}\frac{1}{s}e^{-|a-b|_{L}^{2}/C's}ds.
\end{align*}
(The implicit understanding is that if $|a-b|_{L}|1-\xi|^{-1/2}\leq1$, then the first integral is dropped, and the second integral starts from $1$.) For the first integral, we bound $\exp(-s|1-\xi|)$ by $1$. We then evaluate the remaining exponential factor at $s=|a-b|_{L}|1-\xi|^{-1/2}$ for the sake of an upper bound. Then, we integrate $ds/s$. This gives
\begin{align*}
\int_{1}^{L^{2}\wedge(|a-b|_{L}|1-\xi|^{-1/2})}e^{-s|1-\xi|}\frac{1}{s}e^{-|a-b|_{L}^{2}/C's}ds&\leq e^{-c|a-b|_{L}|1-\xi|^{1/2}}\int_{1}^{L^{2}}\frac{ds}{s}\\
&\leq e^{-c|a-b|_{L}|1-\xi|^{1/2}}\log L^{2}.
\end{align*}
For the second integral, we similarly have 
\begin{align*}
\int_{L^{2}\wedge(|a-b|_{L}|1-\xi|^{-1/2})}^{L^{2}}e^{-s|1-\xi|}\frac{1}{s}e^{-|a-b|_{L}^{2}/C's}ds&\leq e^{-|a-b|_{L}|1-\xi|^{1/2}}\int_{1}^{L^{2}}\frac{ds}{s}\\
&\leq e^{-c|a-b|_{L}|1-\xi|^{1/2}}\log L^{2},
\end{align*}
where the first bound follows by bounding the sub-Gaussian factor by $1$, by evaluating $\exp(-s|1-\xi|)$ at $s=|a-b|_{L}|1-\xi|^{-1/2}$ for the sake of an upper bound, and by integrating $ds/s$. (We note that it suffices to assume that the lower limit of integration is equal to $|a-b|_{L}|1-\xi|^{-1/2}$, since the integral would otherwise vanish. Also, recall that we are assuming $|a-b|_{L}|1-\xi|^{-1/2}\geq1$.) Combining the previous three displays yields
\begin{align}
\int_{1}^{L^{2}}e^{-s|1-\xi|}\frac{1}{s}e^{-|a-b|_{L}^{2}/C's}ds&\leq e^{-c|a-b|_{L}|1-\xi|^{1/2}}\log L^{2}.\label{eq:theta_rw_first}
\end{align}
We now look at the contribution of the integral in \eqref{theta_rw_int} for $s\geq L^{2}$. In this case, we instead have the pointwise bound
\begin{align}
\sum_{\beta\sim b}\frac{1}{s}e^{-\frac{|a-\beta|^{2}}{C's}}&\leq\frac{1}{L^{2}}e^{-|a-b|_{L}^{2}/C's}\sum_{\gamma\in\Z^{2}}\frac{L^{2}}{s}e^{-cL^{2}|\gamma|^{2}/C's}\leq \frac{C}{L^{2}}\label{eq:hke_long_time}
\end{align}
where all $C,C'$ constants are $O(1)$. Thus, for $s\geq L^{2}$, we have the estimate
\begin{align*}
\int_{L^{2}}^{\infty}e^{-s|1-\xi|}\sum_{\beta\sim b}\frac{1}{s}e^{-|\alpha-\beta|^{2}}{C's}ds&\leq \frac{C}{L^{2}}\int_{L^{2}}^{\infty}e^{-s|1-\xi|}ds\\
&\leq \frac{C}{L^{2}|1-\xi|}e^{-L^{2}|1-\xi|}.
\end{align*}
Combining the previous display with \eqref{theta_rw_int} and \eqref{eq:theta_rw_first} and recalling that $\hat{\ell}(\xi):=\min(|1-\xi|^{-1/2},L)$ gives
\begin{align}
(\Theta^{(B)}_{\xi})_{ab} &\prec \delta_{ab}+ e^{-c|a-b|_{L}|1-\xi|^{1/2}}+\frac{e^{-L^{2}|1-\xi|}}{L^{2}|1-\xi|}\\
&\prec \frac{e^{-c|a-b|/\hat{\ell}(\xi)}}{|1-\xi|[\hat{\ell}(\xi)]^{2}}+\frac{e^{-cL/\hat{\ell}(\xi)}}{|1-\xi|[\hat{\ell}(\xi)]^{2}}.
\end{align}
This gives the desired pointwise estimate \eqref{prop:ThfadC} in property 5 of Lemma \ref{lem_propTH}.

We are left to show property 6 of Lemma \ref{lem_propTH}. Recall $p_{k}$ as the transition probability kernel for a lazy simple random walk on $\Z^{2}$. A standard computation for $Y_{k}$ in terms of binomial coefficients and Stirling's formula gives the following for any $\sigma\in\Z^{2}$: 
\begin{align*}
\nabla_{\sigma}p_{k}(\alpha,\beta):=p_{k}(\alpha,\beta+\sigma)-p_{k}(\alpha,\beta)=O\Big(\frac{|\sigma|}{1+k^{3/2}}\Big).
\end{align*}
Moreover, since $Y_{k}$ is a space-homogeneous random walk, the kernel $p_{k}(\alpha,\beta)$ depends only on $\alpha-\beta$. Thus, we have $\nabla_{\sigma}p_{k}(\alpha,\beta)=p_{k}(\alpha-\sigma,\beta)-p_{k}(\alpha,\beta)$. By another application of the semigroup property, we get 
\begin{align}
\nabla_{\sigma}p_{k}(\alpha,\beta)=\sum_{|\zeta-\alpha|\geq|\alpha-\beta|/2}\nabla_{\sigma}p_{k_{1}}(\alpha,\zeta)p_{k_{2}}(\zeta,\beta)+\sum_{|\zeta-\beta|\geq|\alpha-\beta|/2}\nabla_{\sigma}p_{k_{1}}(\alpha,\zeta)p_{k_{2}}(\zeta,\beta).
\end{align}
For the second term on the right-hand side, we can apply the gradient bound and the off-diagonal estimate for $p_{k_{2}}(\zeta,\beta)$. This gives
\begin{align*}
\left|\sum_{|\zeta-\beta|\geq|\alpha-\beta|/2}\nabla_{\sigma}p_{k_{1}}(\alpha,\zeta)p_{k_{2}}(\zeta,\beta)\right|\leq \frac{C|\sigma|}{1+k^{3/2}}e^{-|\alpha-\beta|^{2}/Ck}.
\end{align*}
On the other hand, we can use summation-by-parts to get
\begin{align*}
&\sum_{|\zeta-\alpha|\geq|\alpha-\beta|/2}\nabla_{\sigma}p_{k_{1}}(\alpha,\zeta)p_{k_{2}}(\zeta,\beta)\\
&=\sum_{\zeta\in\Z}p_{k_{1}}(\alpha,\zeta)\nabla_{\sigma}\left\{\mathbf{1}_{|\zeta-\alpha|\geq|\alpha-\beta|/2}p_{k_{2}}(\zeta,\beta)\right\}\\
&=\sum_{|\zeta-\alpha|\geq|\alpha-\beta|/2}p_{k_{1}}(\alpha,\zeta)\nabla_{\sigma}p_{k_{2}}(\zeta,\beta)+\sum_{\zeta\in\Z}p_{k_{1}}(\alpha,\zeta)p_{k_{2}}(\zeta,\beta)\nabla_{\sigma}\mathbf{1}_{|\zeta-\alpha|\geq|\alpha-\beta|/2}.
\end{align*}
We already showed that the first term in the last line is $O(|\sigma|(1+k^{3/2})^{-1}e^{-|\alpha-\beta|^{2}/Ck})$. As long as $|\alpha-\beta|\geq B$ for some universal constant $B>0$, the second line in the last line is supported on points $\zeta\in\Z^{2}$ whose distance from the circle of radius $|\alpha-\beta|/2$ around $\alpha$ is $O(|\sigma|)$. This set is at most $O(|\sigma||\alpha-\beta|)$ in size because we are on the two-dimensional lattice $\Z^{2}$. On the other hand, for \emph{any} $\zeta\in\Z^{2}$, at least one of $|\alpha-\zeta|$ and $|\beta-\zeta|$ is greater than or equal to $|\alpha-\beta|/2$. Thus, we have $p_{k_{1}}(\alpha,\zeta)p_{k_{2}}(\zeta,\beta)\leq C(1+k^{2})^{-1}\exp(-|\alpha-\beta|^{2}/Ck)$ for some $C>0$. We deduce that the second sum in the last line above satisfies
\begin{align*}
\left|\sum_{\zeta\in\Z}p_{k_{1}}(\alpha,\zeta)p_{k_{2}}(\zeta,\beta)\nabla_{\sigma}\mathbf{1}_{|\zeta-\alpha|\geq|\alpha-\beta|/2}\right|\leq \frac{C|\sigma|}{1+k^{2}}e^{-|\alpha-\beta|^{2}/Ck}\cdot |\alpha-\beta|.
\end{align*}
We can drop the last factor $|\alpha-\beta|$ if we modify the exponent $C$ slightly and include an additional factor of $1+k^{1/2}$; this is because $|x|e^{-|x|}$ is uniformly bounded for $x\in\R$. Ultimately, we have shown that 
\begin{align*}
|\nabla_{\sigma}p_{k}(\alpha,\beta)|\leq \frac{C|\sigma|}{1+k^{3/2}}e^{-|\alpha-\beta|^{2}/Ck}.
\end{align*}
(Technically, we assumed that $|\alpha-\beta|\geq B$ for some fixed constant $B$. However, the previous display follows by the on-diagonal gradient estimate if $|\alpha-\beta|\leq B$.) We can plug the previous display into the identity \eqref{theta_rw}. This gives the following, in which $\nabla_{\sigma}$ acts on the $b$-variable:
\begin{align}
|\nabla_{\sigma}(\Theta^{(B)}_{\xi})_{ab}|&\leq \mathbf{1}_{|a-b|_{L}\leq|\sigma|}+\sum_{k=1}^{\infty}\xi^{k}\sum_{\beta\sim b}\frac{C|\sigma|}{k^{3/2}}e^{-|a-\beta|^{2}/Ck}\\
&\leq\mathbf{1}_{|a-b|_{L}\leq|\sigma|}+C'|\sigma|\int_{1}^{\infty}e^{-s|1-\xi|}\sum_{\beta\sim b}\frac{1}{s^{3/2}}e^{-|a-\beta|^{2}/Cs}ds.\label{eq:theta_rw_int_grad}
\end{align}
We will now restrict the integration to $[1,L^{2}]$ and $[L^{2},\infty)$ and treat these integrals separately. On the former domain with $s\leq L^{2}$, we can use the estimate \eqref{eq:hke_short_time} and drop the factor $\exp(-s|1-\xi|)$. This gives
\begin{align*}
\int_{1}^{L^{2}}e^{-s|1-\xi|}\sum_{\beta\sim b}\frac{1}{s^{3/2}}e^{-|a-\beta|^{2}/Cs}ds&\leq C'\int_{1}^{L^{2}}s^{-\frac32}e^{-|a-b|_{L}^{2}/Cs}ds\\
&\leq C''\int_{1}^{\infty}s^{-\frac32}e^{-(1+|a-b|_{L}^{2})/Cs}ds\\
&\leq \frac{C''}{1+|a-b|_{L}}\int_{0}^{\infty}s^{-\frac32}e^{-1/s}ds\\
&\leq \frac{C'''}{1+|a-b|_{L}}.
\end{align*}
The second line follows since $\exp(1/Cs)=O(1)$ for $s\geq1$. The third line follows by the change-of-variables $s\mapsto |a-b|_{L}^{2}s$. For the integral on $[L^{2},\infty)$. We instead use the heat kernel estimate \eqref{eq:hke_long_time}. This gives
\begin{align*}
\int_{L^{2}}^{\infty}e^{-s|1-\xi|}\sum_{\beta\sim b}\frac{1}{s^{3/2}}e^{-|a-\beta|^{2}/Cs}ds&\leq C'\int_{L^{2}}^{\infty}L^{-2}s^{-\frac12}e^{-s|1-\xi|}ds\\
&\leq C'L^{-2}|1-\xi|^{-\frac12}\int_{L^{2}|1-\xi|}^{\infty}s^{-\frac12}e^{-s}ds.
\end{align*}
The second line follows by the change-of-variables $s\mapsto |1-\xi|^{-1}s$. We now consider two cases. First, if $|1-\xi|^{-1/2}\leq L$, i.e. $L^{2}|1-\xi|\geq1$, then we can drop $s^{-1/2}$ in the last line above and integrate the exponential. Second, if $|1-\xi|^{-1/2}\geq L$, so that $L^{2}|1-\xi|\leq1$, we bound the remaining integral in the last line by $O(1)$. This gives
\begin{align*}
L^{-2}|1-\xi|^{-\frac12}\int_{L^{2}|1-\xi|}^{\infty}s^{-\frac12}e^{-s}ds&\leq \mathbf{1}_{L^{2}|1-\xi|\geq1}\frac{e^{-L^{2}|1-\xi|}}{L^{2}|1-\xi|^{1/2}}+\mathbf{1}_{L^{2}|1-\xi|\leq1}\frac{1}{L^{2}|1-\xi|^{1/2}}.
\end{align*}
Recall that $\hat{\ell}(\xi)=\min(|1-\xi|^{-1/2},L)$. The last term in the above display forces $\hat{\ell}(\xi)=L$. Thus, this term is $O([\hat{\ell}(\xi)]^{-2}|1-\xi|^{-1/2})$. For the first term, we instead have $\hat{\ell}(\xi)=|1-\xi|^{-1/2}$. In particular, for this term, we have a bound of $|1-\xi|^{1/2}=[\hat{\ell}(\xi)]^{-2}|1-\xi|^{-1/2}$. We deduce that the above display is $O([\hat{\ell}(\xi)]^{-2}|1-\xi|^{-1/2})$, and ultimately that
\begin{align*}
|\nabla_{\sigma}(\Theta^{(B)}_{\xi})_{ab}|&\prec \mathbf{1}_{|a-b|_{L}\leq|\sigma|}+\frac{|\sigma|}{|a-b|_{L}+1}+\frac{|\sigma|}{[\hat{\ell}(\xi)]^{2}|1-\xi|^{-1/2}}.
\end{align*}
This implies the desired first derivative estimate \eqref{prop:BD1}. We now show the second derivative estimate \eqref{prop:BD2}. For any $\sigma_{1},\sigma_{2}\in\Z^{2}$, another standard computation with binomial coefficients gives 

$$|\nabla_{\sigma_{1}}\nabla_{\sigma_{2}}p_{k}(\alpha,\beta)| \leq \frac{C|\sigma_{1}||\sigma_{2}|}{1+k^{2}}.$$

Another semigroup property and summation-by-parts estimate takes this estimate and produces for us
\begin{align*}
|\nabla_{\sigma_{1}}\nabla_{\sigma_{2}}p_{k}(\alpha,\beta)|\leq\frac{C|\sigma_{1}||\sigma_{2}|}{1+k^{2}}e^{-|\alpha-\beta|^{2}/Ck}
\end{align*}
for some $C=O(1)$. If we plug the previous display into \eqref{theta_rw}, then we obtain the following estimate, where the gradients act on the $b$-variable:
\begin{align}
|\nabla_{\sigma_{1}}\nabla_{\sigma_{2}}(\Theta^{(B)}_{\xi})_{ab}|&\leq \mathbf{1}_{|a-b|_{L}=O(|\sigma_{1}|+|\sigma_{2}|)}+\sum_{k=1}^{\infty}\xi^{k}\sum_{\beta\sim b}\frac{C|\sigma_{1}||\sigma_{2}|}{1+k^{2}}e^{-|a-\beta|^{2}/Ck}\\
&\leq \mathbf{1}_{|a-b|_{L}=O(1)}+C|\sigma_{1}||\sigma_{2}|\int_{1}^{\infty}e^{-s|1-\xi|}\sum_{\beta\sim b}s^{-2}e^{-|a-\beta|^{2}/Cs}ds.
\end{align}
We will again split the integral into integrals on $[1,L^{2}]$ and $[L^{2},\infty)$, respectively. On the former domain, we use \eqref{eq:hke_short_time} and drop the $\exp(-s|1-\xi|)$ factor. This gives
\begin{align*}
\int_{1}^{L^{2}}e^{-s|1-\xi|}\sum_{\beta\sim b}s^{-2}e^{-|a-\beta|^{2}/Cs}ds&\leq C\int_{1}^{\infty}s^{-2}e^{-|a-b|_{L}^{2}/Cs}ds\\
&\leq C'\int_{1}^{\infty}s^{-2}e^{-(1+|a-b|_{L}^{2})/Cs}ds\\
&\leq \frac{C'}{1+|a-b|_{L}^{2}}\int_{0}^{\infty}s^{-2}e^{-1/s}ds\leq \frac{C'}{1+|a-b|_{L}^{2}}.
\end{align*}
For the integral on $[L^{2},\infty)$, we instead use \eqref{eq:hke_long_time}. This gives
\begin{align*}
\int_{L^{2}}^{\infty}e^{-s|1-\xi|}\sum_{\beta\sim b}s^{-2}e^{-|a-\beta|^{2}/Cs}ds&\leq C L^{-2}\int_{L^{2}}^{\infty}e^{-s|1-\xi|}s^{-1}ds\\
&=C L^{-2}\int_{L^{2}|1-\xi|}^{\infty}e^{-s}s^{-1}ds\\
&\leq C L^{-2}e^{-L^{2}|1-\xi|}.
\end{align*}
If $L^{2}|1-\xi|\leq1$, then $\hat{\ell}(\xi):=\min(|1-\xi|^{-1/2},L)=L$. In this case, the last line is $O([\hat{\ell}(\xi)]^{-2})$. On the other hand, if $L^{2}|1-\xi|\geq1$, then $\hat{\ell}(\xi)=|1-\xi|^{-1/2}$. In this case, since $|x|e^{-|x|}=O(1)$ uniformly in $x\in\R$, the last line is $O|1-\xi|=O([\hat{\ell}(\xi)]^{-2})$. In particular, the last line is $O([\hat{\ell}(\xi)]^{-2})$ in general. Ultimately, we have
\begin{align*}
|\nabla_{\sigma_{1}}\nabla_{\sigma_{2}}(\Theta^{(B)}_{\xi})_{ab}|&\prec \frac{|\sigma_{1}||\sigma_{2}|}{1+|a-b|_{L}^{2}}+\frac{|\sigma_{1}||\sigma_{2}|}{[\hat{\ell}(\xi)]^{2}}.
\end{align*}
This proves \eqref{prop:BD2}. \qed

\printbibliography

\end{document}